\documentclass{amsart}
\usepackage{lmodern}
\usepackage[english]{babel}
\usepackage{amsmath}
\usepackage{amssymb}
\usepackage{mathptmx} 
\usepackage{color}
\usepackage{graphicx}
\usepackage{caption}
\usepackage{subcaption}
\usepackage{float}
\usepackage[all,cmtip]{xy}
\usepackage{bm}
\usepackage{enumerate}
\usepackage{pinlabel}
\newtheorem{theorem}{\bf Theorem}[section]
\newtheorem{lemma}[theorem]{\bf Lemma}
\newtheorem{definition}[theorem]{\bf Definition}
\newtheorem{corollary}[theorem]{\bf Corollary}

\newtheorem{remark}[theorem]{\bf Remark}
\newtheorem{example}[theorem]{\bf Example}

\newcommand{\rme}{\mathrm{e}}
\newcommand{\rmi}{\mathrm{i}}

\newcommand{\sign}{\operatorname{sign}}

\newcommand{\defeq}{\mathrel{\mathop:}=}

\begin{document}

%%%% Article title to be placed here
\title{Closures of T-homogeneous braids are real algebraic}

\author{%%%% Author details
Benjamin Bode}
\date{}

\address{Instituto de Ciencias Matemáticas, CSIC, 28049 Madrid, Spain}
\email{benjamin.bode@icmat.es}
%%%%%%%%% Insert author address here
%\address{Osaka
%ben.bode.2013@my.bristol.ac.uk
%}

%%%% Subject entries to be placed here %%%%
%\subject{Mathematical Physics, Topology}

%%%% Keyword entries to be placed here %%%%
%\keywords{braid group action, }
%%% good choices?

%%%% Insert corresponding author and its email address}
%\corres{Benjamin Bode, Mark Dennis\\
%\email{benjamin.bode@bristol.ac.uk}, \email{mark.dennis@bristol.ac.uk}}
%BB: wrote Benjamin instead of Ben
%%% mrd: final final version, from now on all edits require comment

%%%% Abstract text to be placed here %%%%%%%%%%%%

%%%%%%%%%%%%%%%%%%%%%%%%%%%
%
%%%%%%%%%% Insert the texts which can accomdate on firstpage in the tag "fmtext" %%%%%

\maketitle
\begin{abstract}
A link in $S^3$ is called real algebraic if it is the link of an isolated singularity of a polynomial map from $\mathbb{R}^4$ to $\mathbb{R}^2$. It is known that every real algebraic link is fibered and it is conjectured that the converse is also true. We prove this conjecture for a large family of fibered links, which includes closures of T-homogeneous (and therefore also homogeneous) braids and braids that can be written as a product of the dual Garside element and a positive word in the Birman-Ko-Lee presentation. The proof offers a construction of the corresponding real polynomial maps, which can be written as semiholomorphic functions. We obtain information about their polynomial degrees. 
\end{abstract}

\section{Introduction}\label{sec:intro}
It is well known that the set of algebraic links, the links that arise as the links of isolated singularities of complex polynomials, consists of certain iterated cables of torus links \cite{brauner}. In the last chapter of his seminal work \cite{milnor}, Milnor discusses the real analogue of this classification problem.
Let $f:\mathbb{R}^4\to\mathbb{R}^2$ be a polynomial map with $f(0)=0$, whose Jacobian $Df$ vanishes at the origin, but has full rank at all other points of some neighbourhood of the origin. In this case, we say that the origin is an \textit{isolated singularity} of $f$. Exactly as in the complex case, the intersection $f^{-1}(0)\cap S_{\rho}^3$ of the vanishing set of $f$ and a 3-sphere of sufficiently small radius $\rho>0$ is a link, whose link type does not depend on $\rho$. We call this link the \textit{link of the singularity} and denote it by $L_f$.

Analogous to the terminology from the complex setting we say that a link $L$ is \textit{real algebraic} if $L=L_f$ for some polynomial $f$. In contrast to the set of algebraic links, the set of real algebraic links has not been classified yet. Note that a generic real polynomial in these dimensions does not have an isolated singularity. The difficulty of constructing polynomials with isolated singularities and a given link type has already been noted by Milnor \cite{milnor}.

While there are many differences between the complex and the real setting, the importance of (Milnor) fibrations can be found in both. In particular, Milnor showed that every real algebraic link is fibered \cite{milnor}. It is an open conjecture by Benedetti and Shiota that the set of real algebraic links and the set of fibered links are identical \cite{benedetti}. The main result of this paper is a construction of certain polynomials with isolated singularities, which allows us to prove that a large family of fibered links is indeed real algebraic.

All links in this family are closures of so-called \textit{P-fibered braids} or, equivalently, bindings of \textit{braided open book decompositions}, which can be described in terms of certain diagrams called \textit{Rampichini diagrams} \cite{bode:braided, mortramp, rampi}. We introduce an operation on Rampichini diagrams detailed in the later sections. We say that one P-fibered braid $B_2$ is obtained from another $B_1$ by the insertion of \textit{inner loops} if their Rampichini diagrams are related by such operations. The family of fibered links for which we can prove real algebraicity can be characterized in terms of properties of their Rampichini diagrams.

\begin{theorem}\label{thm:main}
Let $B_1$ be a P-fibered braid on $n$ strands with an odd, pure Rampichini diagram. Let $B_2$ be a P-fibered braid that is obtained from $B_1$ by the insertion of inner loops. Then the closure of $B_2$ is real algebraic.\\
Furthermore, the corresponding real polynomial map with an isolated singularity can be taken to be semiholomorphic (i.e., it can be written as a polynomial $f:\mathbb{C}^2\to\mathbb{C}$ in complex variables $u$, $v$ and the complex conjugate $\overline{v}$) and of degree $n$ with respect to the complex variable $u$.
\end{theorem}

The properties of being odd and pure will be defined at a later point (Definition~\ref{def:odd} and Definition~\ref{def:comp}). It is not straightforward to decide if a given link satisfies the condition from Theorem~\ref{thm:main}. However, for large families of fibered links we prove that this is the case.

\begin{theorem}\label{thm:Thomo}
Let $B$ be a T-homogeneous braid. Then the closure of $B$ is real algebraic and the corresponding polynomial can be taken to be semiholomorphic.
\end{theorem}
The family of T-homogeneous braids (cf. Definition~\ref{def:thomo}) was introduced by Rudolph \cite{rudolph2} as a generalization of the family of homogeneous braids (whose closures are known to be fibered \cite{stallings2}), so that we immediately obtain the following corollary.
%, which proves a conjecture by Karadereli and Öztürk \cite{ozturk}.   
\begin{corollary}\label{cor:homo}
Let $B$ be a homogeneous braid. Then the closure of $B$ is real algebraic and the corresponding polynomial can be taken to be semiholomorphic.
\end{corollary}

In the context of braided open books and Rampichini diagrams it is useful to represent braids in terms of the BKL-generators (Birman, Ko and Lee \cite{BKL}) $a_{i,j}$, also called band generators, instead of Artin generators $\sigma_i$. A word in BKL-generators is called BKL-positive if it does not include inverses of any of the generators $a_{i,j}$. Let  $\delta=a_{n-1,n}a_{n-2,n-1}\ldots a_{1,2}=\sigma_{n-1}\sigma_{n-2}\ldots\sigma_{1}$ denote the Dual Garside element \cite{ian, BKL}.
\begin{theorem}\label{thm:delta}
Let $B=\delta P$, where $P$ is a BKL-positive word. Then the closure of $B$ is real algebraic.
\end{theorem}
This family of links was shown to be fibered in \cite{ian}. In the BKL-presentation every braid on $n$ strands can be written as $\delta^k P$ for some $k\in\mathbb{Z}$, where $P=A_1A_2\ldots A_m$ is BKL-positive. This representation becomes unique for every given braid once appropriate conditions on the factors $A_i$ are imposed \cite{BKL}. We then call $\delta^k P$ the Dual Garside Normal Form of the braid. Since $\delta^{k-1}P$ is BKL-positive if $k\geq1$, Theorem~\ref{thm:delta} immediately implies the following corollary.
\begin{corollary}\label{thm:garside}
Let $B$ be a braid whose Dual Garside Normal Form contains a positive power of the Dual Garside element $\delta$. Then the closure of $B$ is real algebraic and the corresponding polynomial can be taken to be semiholomorphic. 
\end{corollary}

%The family of P-fibered braids has been introduced in ??, but has already appeared (without a name) in the context of real algebraic links in ??. It includes the set of homogeneous braids, T-homogeneous braids, generalised exchangeable braids and all braids on 3 strands, whose closure is fibered.

The remainder of the paper is structured as follows. Section~\ref{sec:back} reviews the relevant concepts in the study of P-fibered braids and Rampichini diagrams. We also introduce the definitions of inner loops and pure Rampichini diagrams, which feature in Theorem~\ref{thm:main}. Section~\ref{sec:odd} introduces odd Rampichini diagrams and proves a result about corresponding (singular) P-fibered braids with a certain ``odd'' symmetry. The proof of Theorem~\ref{thm:main} then follows a similar strategy as the proof of the main result in \cite{bode:AK}, which constructs weakly isolated singularities (as opposed to isolated singularities) for any link type. In comparison to the work in \cite{bode:AK} the individual steps require much more care as the fibration properties and isolation of the singularity are very strong conditions. We first use trigonometric interpolation and approximation to find a loop in the space of monic, complex polynomials $h_t$ of fixed degree, whose coefficients are finite Fourier series and whose roots form a singular braid $B_{sing}$ that is related to the braids $B_1$ and $B_2$ and that has to satisfy the ``odd'' symmetry requirement (cf. Section~\ref{sec:trig}). From this we construct a radially weighed homogeneous polynomial that is degenerate in the sense of \cite{oka} or \cite{bodeeder} and whose zeros form the cone over the singular braid $B_{sing}$. In Section~\ref{sec:deform} we show that trigonometric interpolation techniques can also be used to find a deformation of this polynomial that has an isolated singularity and whose link is obtained from the singular braid by an appropriate resolution of its singular crossings. In Section~\ref{sec:proof} the resulting link is shown to be the closure of the desired braid $B_2$, which concludes the proof of Theorem~\ref{thm:main}. Then Theorem~\ref{thm:Thomo} and Theorem~\ref{thm:delta} along with their corollaries are shown in Section~\ref{sec:homo}.
\ \\
\textbf{Acknowledgments:} This work is supported by the European Union’s Horizon 2020 Research and Innovation Programme under the Marie Sklodowska-Curie grant agreement No 101023017.
 %In Section~\ref{sec:trig} we show that using trigonometric interpolation and approximation we can find trigonometric parametrisations of the singular braid $B_{sing}$.

\section{Background}\label{sec:back}

\subsection{Braids}

A geometric braid on $n$ strands is a loop in the configuration space of $n$ distinct unmarked points in the complex plane. Via the fundamental theorem of algebra this space is identified with the space $X_n$ of monic polynomials in one complex variable of degree $n$ and with distinct roots. This way a braid parametrised as $(z_1(t),z_2(t),\ldots,z_n(t))$, where $t$ is going from $0$ to $2\pi$ corresponds to the loop $g_t(u)\defeq\underset{j=1}{\overset{n}{\prod}}(u-z_j(t))$. We call an equivalence class of geometric braids under braid isotopies a braid. Braid isotopies correspond to homotopies of loops (with fixed base point) in $X_n$, so that the braid group $\mathbb{B}_n$ on $n$ strands can be defined to be the fundamental group of $X_n$. It is common to visualise geometric braids as the set of parametric curves $\underset{j=1}{\overset{n}{\cup}}(z_j(t),t)$ in $\mathbb{C}\times[0,2\pi]$.

\begin{definition}
We say that a geometric braid $B$ is \textit{P-fibered} if the corresponding loop of polynomials $g_t$ defines a fibration via $\arg(g):(\mathbb{C}\times S^1)\backslash B\to S^1$, $g(u,\rme^{\rmi t})\defeq g_t(u)$. We say that a braid is P-fibered if it can be represented by a P-fibered geometric braid.
\end{definition}

The condition of being P-fibered has the following geometric interpretation in terms of the critical values $v_j(t)$, $j=1,2,\ldots,n-1$, of $g_t:\mathbb{C}\to\mathbb{C}$. Since the roots of $g_t$ are simple, all of its critical values $v_j(t)$ are non-zero. A geometric braid is P-fibered if and only if for every $j=1,2,\ldots,n-1$, the derivative $\tfrac{\partial \arg(v_j(t))}{\partial t}$ never vanishes. Thus as $t$ increases from 0 to $2\pi$ every critical value twists around $0\in\mathbb{C}$ in a fixed direction, either clockwise $\left(\tfrac{\partial \arg(v_j(t))}{\partial t}<0\right)$ or counterclockwise $\left(\tfrac{\partial \arg(v_j(t))}{\partial t}>0\right)$. This is illustrated in Figure~\ref{fig:solar}.

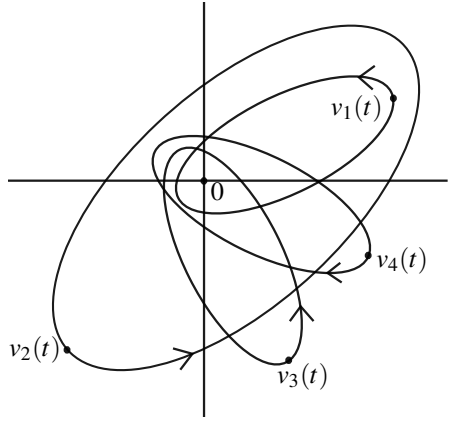
\begin{figure}
\centering
\labellist
\pinlabel $0$ at 780 840
\pinlabel $v_1(t)$ at 1300 1150
\pinlabel $v_2(t)$ at 100 250
\pinlabel $v_3(t)$ at 1100 150
\pinlabel $v_4(t)$ at 1470 580
\endlabellist
\includegraphics[height=5.5cm]{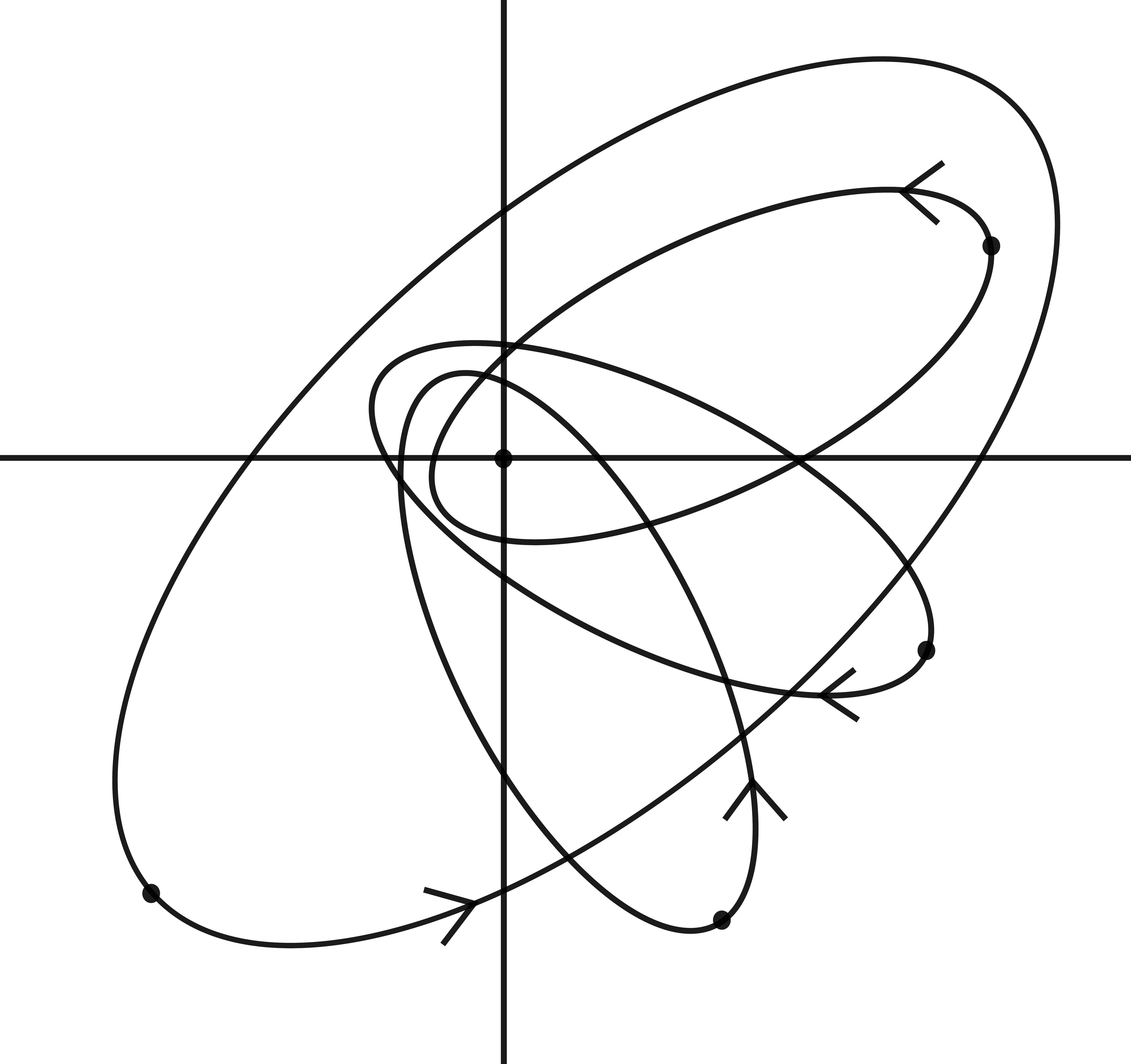}
\caption{The motion in the complex plane of the critical values $v_j(t)$ of the loop of polynomials $g_t$ corresponding to a P-fibered geometric braid. \label{fig:solar}}
\end{figure}

The closure of a P-fibered braid $B$ is a fibered link in the 3-sphere. Via this closing procedure (described in more detail in \cite{bode:braided}) the braided surface $F_{\varphi}=\arg(g_t)^{-1}(\varphi)$ is identified with a fiber surface, whose boundary is the closure of $B$. A fibration of a link complement over a circle that comes from a P-fibered geometric braid and the corresponding loop of polynomials is also called a \textit{braided open book decomposition} of $S^3$ \cite{bode:braided}.

Instead of the more common Artin presentation of the braid group $\mathbb{B}_n$, we will mostly work with the presentation of Birman, Ko and Lee \cite{BKL}, where the generators are $a_{i,j}$, $1\leq i<j\leq n$ and the relations are
\begin{align}
a_{i,j}a_{k,m}&=a_{k,m}a_{i,j}&\text{ if }(i-k)(i-m)(j-k)(j-m)>0,\label{eq:1rel}\\
a_{i,j}a_{j,k}&=a_{i,k}a_{i,j}=a_{j,k}a_{i,k}&\text{ for all }i,j,k\text{ with }1\leq k<j<i\leq n,\label{eq:2rel}
\end{align}
where we set $a_{j,i}=a_{i,j}$ for all $i,j\in\{1,2,\ldots,n\}$. A geometric realisation of a BKL-generator (also called \textit{band generator}) $a_{i,j}$ is depicted in Figure~\ref{fig:BKL}. The better known Artin generator $\sigma_i$ is equal to $a_{i,i+1}$.

\begin{figure}[H]
\centering
\includegraphics[height=4cm]{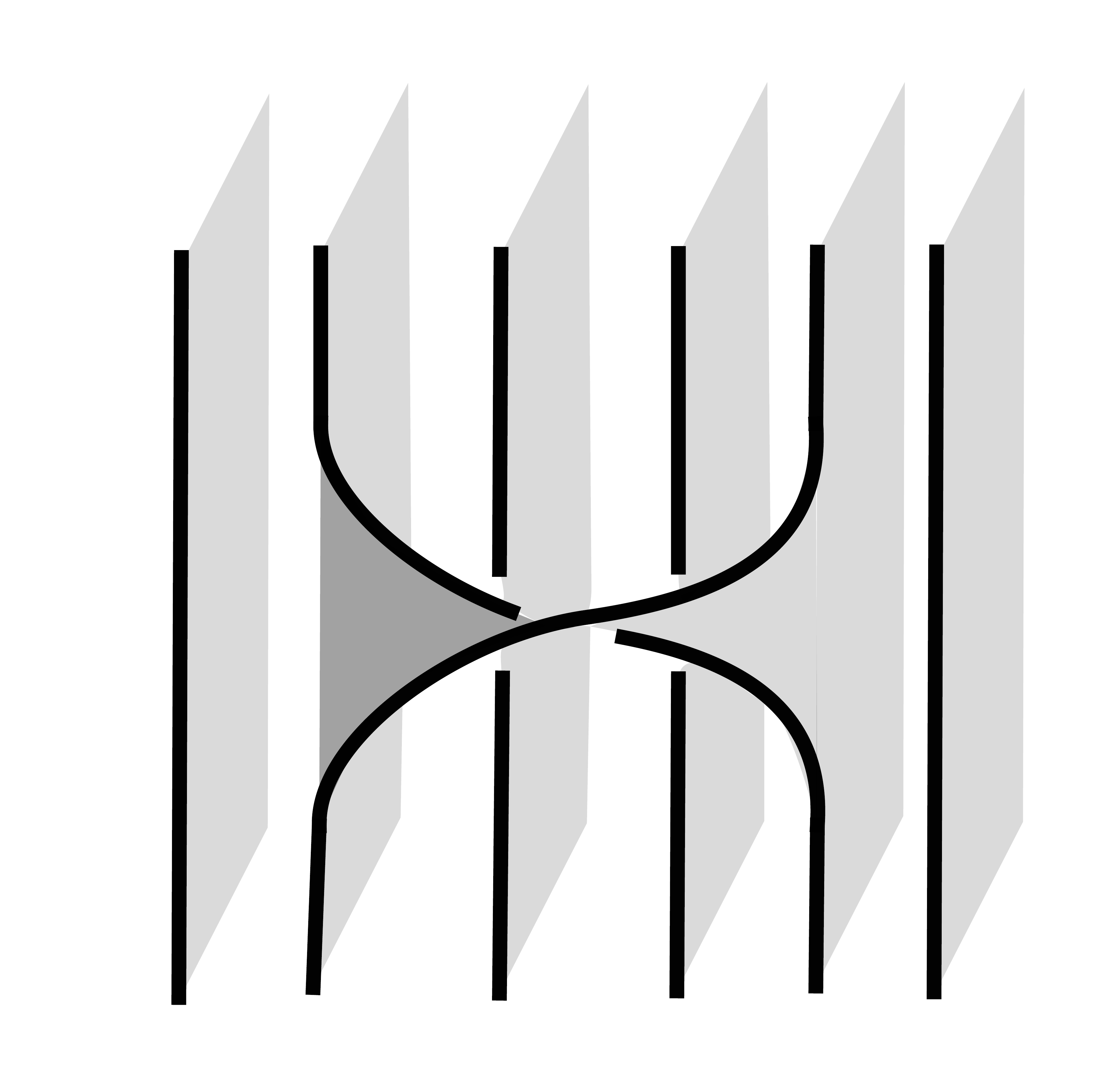}
\caption{The band generator $a_{2,5}$ in $\mathbb{B}_6$. \label{fig:BKL}}
\end{figure}

We can associate in a straightfoward way a \textit{braided surface} (in the sense of Rudolph \cite{rudolphbraid,rudolph2}) or \textit{banded surface} to a given BKL-word. Starting with $n$ parallel disks the braided surface is obtained by inserting a half-twisted band between the $i$th and $j$th disk for each generator $a_{i,j}$ in the given word, where the sign of the twist corresponds to the sign of the generator. Note that the topology of this surface does not change under the BKL-relations Eqs.~\eqref{eq:1rel} and \eqref{eq:2rel}. However, it does change under the trivial group relation $a_{i,j}a_{i,j}^{-1}=a_{i,j}^{-1}a_{i,j}=e$.

%The braid word $A=a_{1,2}a_{2,3}^{-1}a_{1,2}a_{2,3}^{-1}$ on 3 strands for example is a subword of $B=a_{1,2}a_{1,2}a_{2,3}^{-1}a_{2,3}^{-1}a_{1,2}a_{1,2}a_{1,2}a_{2,3}^{-1}a_{2,3}^{-1}$.

In the construction of polynomials described in the later sections we will also encounter singular braids. These differ from usual (geometric) braids in that we allow a finite number of simple intersections (double points, or \textit{singular crossings}) between the strands of a singular braid. There is now a rich theory of such singular braids \cite{baez,birman} and related variants, such as welded braids \cite{fenn}, but we will not really need any of these results apart from the definition itself.

%\subsection{Mixed polynomials}
%Semiholomorphic polynomials are a special type of \textbf{mixed polynomials} as introduced by Oka. In the dimensions that we are interested in, the set of mixed polynomials $f:\mathbb{C}^2\to\mathbb{C}$ consists of polynomials in two complex variables $u$ and $v$, and their complex conjugates, $\overline{u}$ and $\overline{v}$, so that $f$ takes the form
%\begin{equation}
%f(u,v)=\sum_{i,j,k,\ell}c_{i,j,k,\ell}u^i\overline{u}^jv^k\overline{v}^\ell,
%\end{equation}
%with all but finitely many $c_{i,j,k,\ell}\in\mathbb{C}$ equal to to zero. Note that every polynomial map from $\mathbb{R}^4$ to $\mathbb{R}^2$ can be written as a mixed polynomial. A mixed polynomial is semiholomorphic if and only if $c_{i,j,k,\ell}\neq0$ implies $j=0$.

%Oka introduced mixed polynomials as a tool to study singularities of real polynomial maps and their topological properties. 

%Newton boundary, boundary polynomial

%radially weighted homogeneous
%
%non-degenerate, degenerate

\subsection{Rampichini diagrams}\label{subsec:rampi}
In this subsection we review Rampichini diagrams and their connection to P-fibered braids. More details and proofs can be found in \cite{bode:braided,rampi}

\begin{remark}
Throughout this paper the expression $k\text{ mod }n$ refers to the representative of $k$ modulo $n$ in $\{1,2,\ldots,n\}$. 
We use the cycle notation with arrows for permutations, i.e., $(1\to 2)$ denotes the transposition that swaps 1 and 2. In diagrams, where there is little space and no chance of confusion with vectors, we also use $(1,2)$ for $(1\to 2)$.
\end{remark}

Let $g:\mathbb{C}\to\mathbb{C}$ be an element of $X_n$, that is, a monic polynomial of degree $n$ with distinct roots. Since $g$ is monic, it satisfies $\underset{\rho\to\infty}{\lim}\arg(g(\rho\rme^{\rmi \chi}))=n\chi$. Thus it can be interpreted as a simple branched cover from the disk $D$ to itself with $g(\rme^{\rmi \chi})=\rme^{\rmi n\chi}$. The $n$th roots of unity (labeled by 1 through $n$ as in Figure~\ref{fig:foliation}a)) are thus the $n$ preimage points of $\arg(g)=0$ on $\partial D$. They divide $\partial D$ into $n$ arcs $A_i$, where the $A_i$ connects $\rme^{\rmi 2\pi i/n}$ and $\rme^{\rmi 2\pi(i+1\text{ mod }n)/n}$.

\begin{figure}[H]
\labellist
\Large
\pinlabel a) at 100 900
\pinlabel b) at 1200 900
\pinlabel $A_i$ at 1590 940
\pinlabel $A_j$ at 1730 85
\pinlabel $i$ at 1850 940
\pinlabel $j$ at 1450 110
\pinlabel $A_1$ at 220 800
\pinlabel $A_2$ at 230 250
\pinlabel $A_3$ at 930 250
\pinlabel $A_4$ at 930 800
\pinlabel 1 at 540 1000
\pinlabel 2 at 100 510
\pinlabel 3 at 554 50
\pinlabel 4 at 1010 520
\small 
\pinlabel $c_k(t)$ at 1600 550
\endlabellist
\centering
\includegraphics[height=5cm]{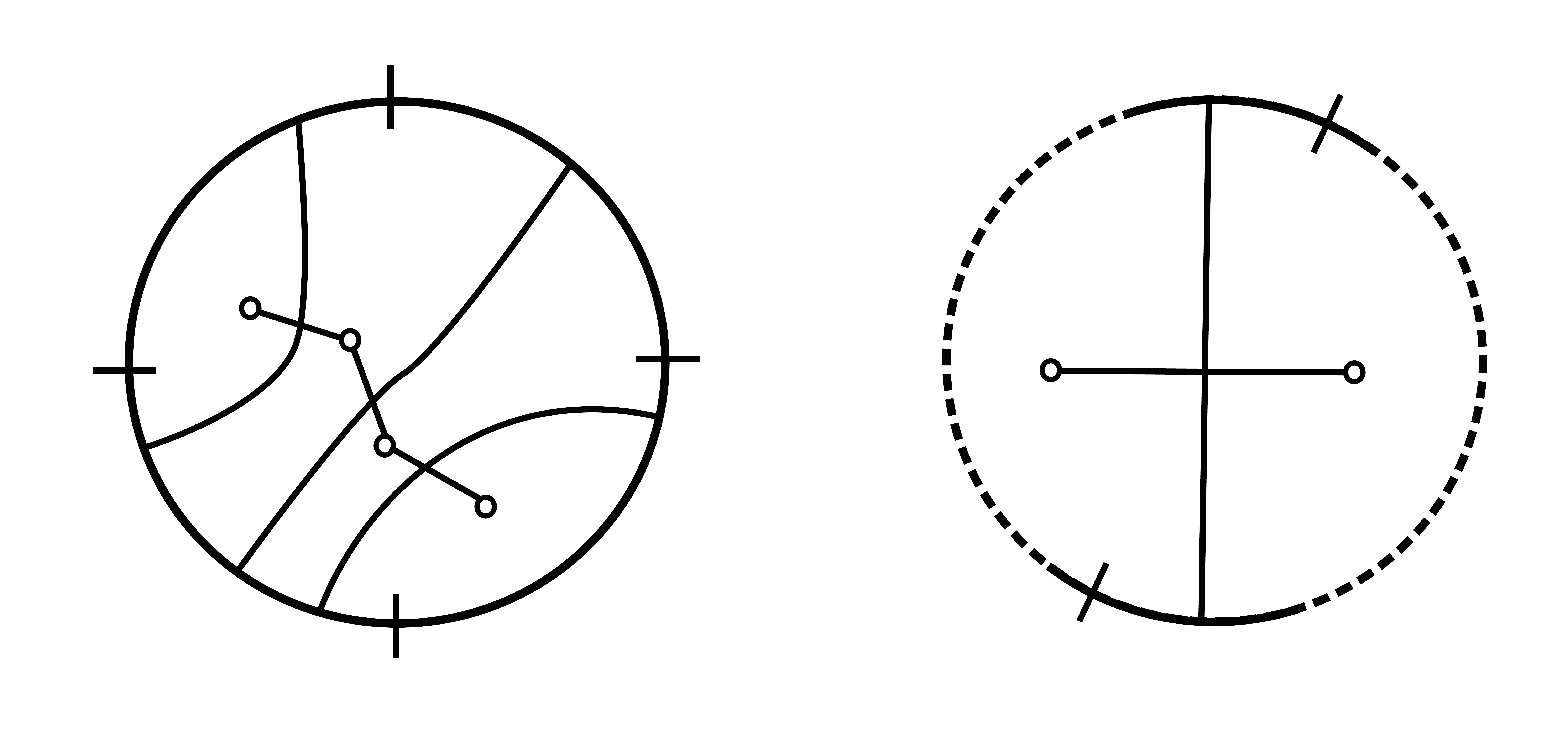}
\caption{a) The singular leaves of the singular foliation on $D_t$ induced by $\arg(g_t)$ and the definition of the segments $A_i$. Small circles are roots of $g_t$. b) The singular leaf containing $c_k(t)$. The transposition associated to $c_k(t)$ and $v_k(t)=g_t(c_k(t))$ is $(i\to j)$. \label{fig:foliation}}
\end{figure}
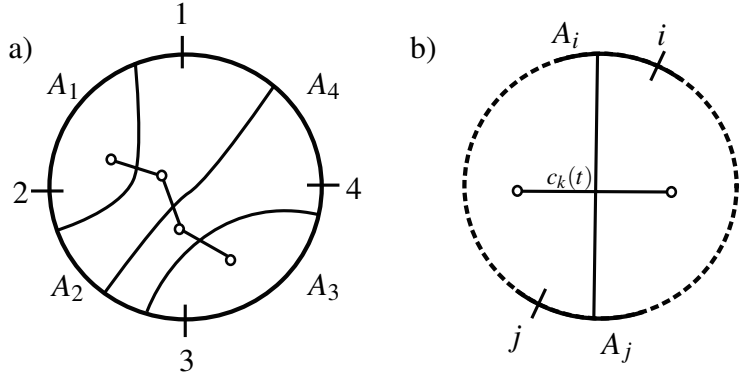

Furthermore, we assume that the critical values $v_j$, $j=1,2,\ldots,n-1$, of $g$ have distinct arguments and are ordered such that $0<\arg(v_1)<\arg(v_2)<\ldots<\arg(v_{n-1})<2\pi$. We denote the $n-1$ critical points of $g$ by $c_j$, $j=1,2,\ldots,n-1$, indexed such that $g(c_j)=v_j$. The map $\arg(g)$ induces a singular foliation on $D$, with the roots of $g$ as elliptic singular points and the critical points as hyperbolic points (cf. Figure~\ref{fig:foliation}a)).

The leaf of a critical point $c_k$ has the shape of a cross, consisting of one line that connects two roots and one line that connects two points on $\partial D$, say on $A_i$ and $A_j$, intersecting in $c_k$. Then we associate to a critical point $c_k$ the transposition $\tau_k=(i\to j)$ (see Figure~\ref{fig:foliation}b)). Since there is a 1-to-1 correspondence between critical points and critical values we can also think of $\tau_k$ as a transposition associated with the critical value $v_k$. The ordered set of transpositions $\tau_k$, $k=1,2,\ldots,n-1$, is called the \textit{cactus} of $g$ and satisfies $\underset{k=1}{\overset{n-1}{\prod}}\tau_k=(1\to n\to n-1\to \ldots\to3\to2)$.

Consider now a path $g_t$ in $X_n$, with the parameter $t$ going from $0$ to $2\pi$, whose endpoints satisfy the properties of $g$ above. After a small homotopy of the path that does not change its endpoints we may assume that for all but finitely many values of $t\in[0,2\pi]$ the arguments of the critical values $v_j(t)$ are distinct and at the remaining values of $t$ there are only two critical values with the same argument (but different absolute values). For every value of $t$ we order the critical values by their argument in $[0,2\pi)$. Note that this means that the labeling of the critical values changes at the values of $t$ where two critical values have the same argument or where one critical value crosses the line $\arg=0$ (either from below or from above). We have a cactus $\tau_k(t)$, $k=1,2,\ldots,n-1$, for every value of $t$ for which the arguments of the critical values are distinct.

We store this information on the different cacti in a so-called \textit{square diagram} as in Figure~\ref{fig:rampi}a). It consists of a square containing curves that are labeled by transpositions in $S_n$, the symmetric group on $n$ elements. The square represents a cylinder $S^1\times[0,2\pi]$ with coordinates $(\varphi,t)$, both of which go from $0$ to $2\pi$, where $\varphi$ is $2\pi$-periodic. A point $(\varphi,t)$ in the square lies on one of the curves in the diagram if and only if there is a critical value $v_j(t)$ of $g_t$ with $\arg(v_j(t))=\varphi$. Every point on a curve in the diagram thus corresponds to a critical value $v_j(t)$. Note that this means that every horizontal line ($t=const.$) has exactly $n-1$ intersections with the curves in the square, when counted with multiplicities. The label of a point on a curve (away from intersection points) is the transposition $\tau_j(t)$ associated to the corresponding critical value $v_j(t)$. At intersection points of curves in the square there is no well-defined transposition.

\begin{figure}[H]
\labellist
\Large
\pinlabel a) at  50 840
\pinlabel b) at 1000 840
\small
\pinlabel 0 at 150 110
\pinlabel 0 at 210 50
\pinlabel $2\pi$ at 850 50
\pinlabel $2\pi$ at 150 750
\pinlabel $(2,3)$ at 270 210
\pinlabel $(1,4)$ at 270 410
\pinlabel $(1,2)$ at 270 580
\pinlabel $(1,4)$ at 410 790
\pinlabel $(1,3)$ at 530 790
\pinlabel $(1,2)$ at 690 790
\pinlabel $(1,3)$ at 415 610
\pinlabel $(2,3)$ at 440 435
\pinlabel $(1,4)$ at 440 280
\pinlabel $(3,4)$ at 550 620
\pinlabel $(1,2)$ at 495 520
\pinlabel $(1,3)$ at 630 340
\pinlabel $(1,2)$ at 720 180
\pinlabel $(1,4)$ at 900 540
\pinlabel $(3,4)$ at 900 390
\pinlabel $(1,2)$ at 900 210
\pinlabel $(3,4)$ at 610 525
\pinlabel $(1,2)$ at 1820 210
\pinlabel $(3,4)$ at 1820 380
\pinlabel $(1,4)$ at 1820 530
\pinlabel $(1,4)$ at 1300 790
\pinlabel $(1,3)$ at 1425 790
\pinlabel $(1,2)$ at 1590 790 
\pinlabel 0 at 1060 110
\pinlabel 0 at 1110 50
\pinlabel $2\pi$ at 1750 50
\pinlabel $2\pi$ at 1060 750
\Large
\pinlabel $t$ at 100 450
\pinlabel $\varphi$ at 500 20
\pinlabel $t$ at 1030 450
\pinlabel $\varphi$ at 1400 20
\endlabellist
\centering
\includegraphics[height=6cm]{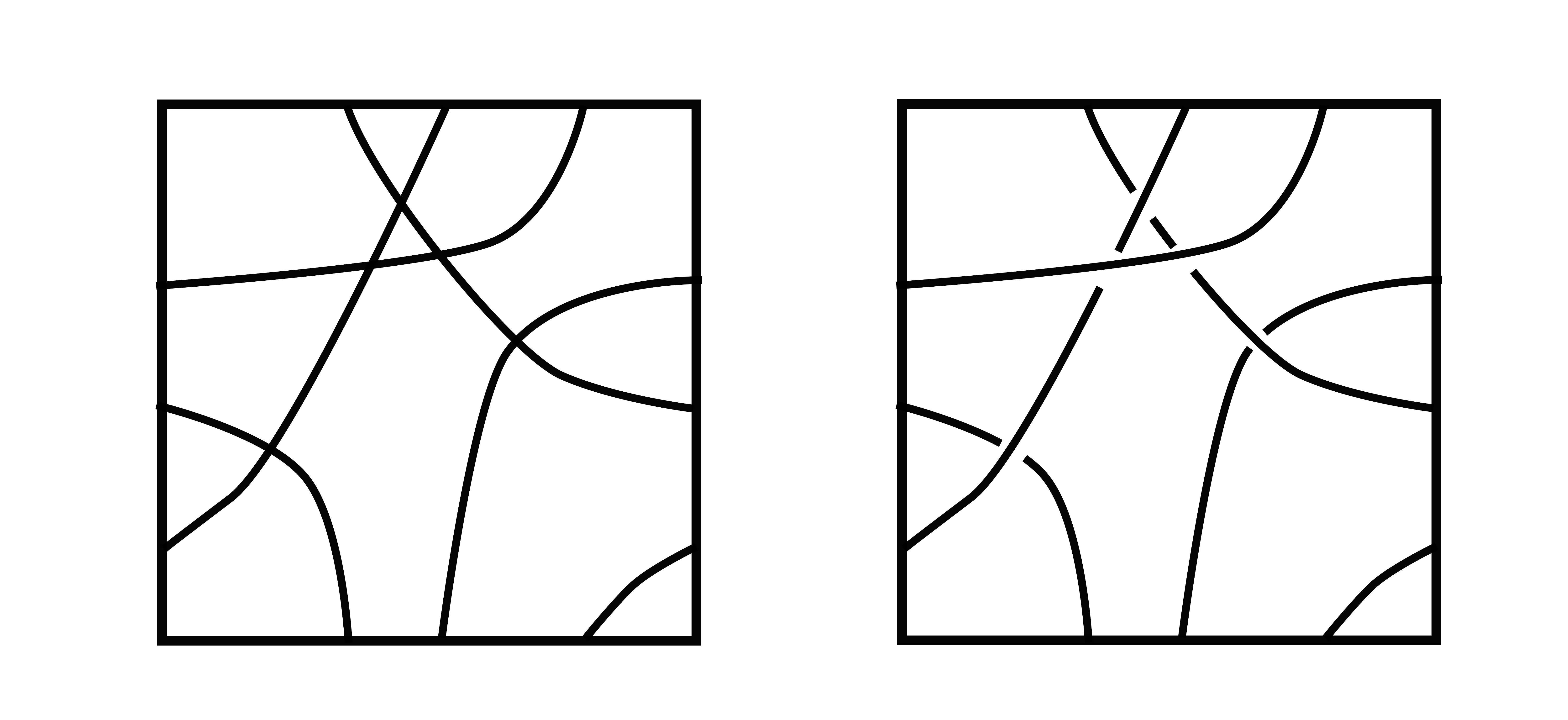}
\caption{A Rampichini diagram. The transpositions at $t=0$ are $\tau_1=(1\to 4)$, $\tau_2=(1\to3)$ and $\tau_3=(1\to2)$. A band word for the fiber $F_{\varphi=2\pi-\varepsilon}$ is given by $a_{1,2}a_{3,4}^{-1}a_{1,4}$. \label{fig:rampi}}
\end{figure}
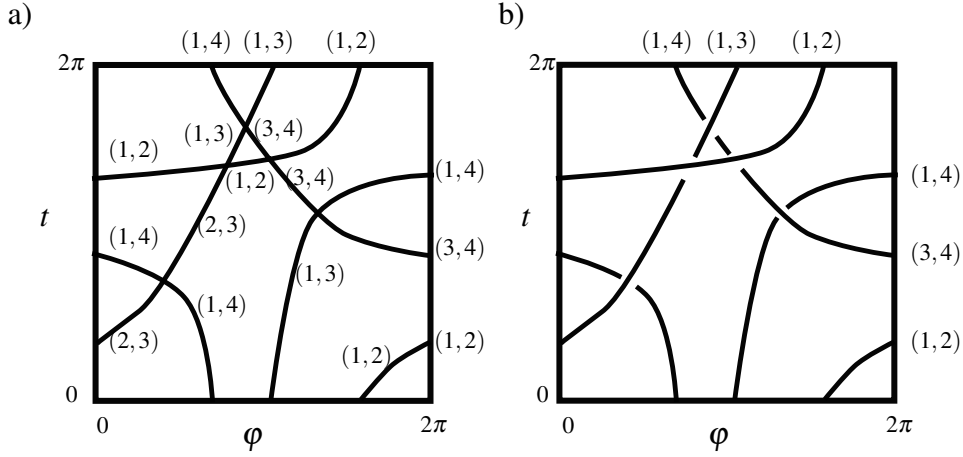

The labels in a square diagram can only change at the right edge $(\varphi=2\pi)$ of the square, where all labels are shifted by $+1\text{ mod }n$ or at intersection points between the different curves. An intersection point between two curves in the square at $(\varphi_*,t_*)$ occurs if and only if $\arg(v_j(t_*))=\arg(v_{j+1}(t_*))=\varphi_*$ for some $j$, where the indexing of the critical values at $t=t_*$ is taken to be that at $t_*-\varepsilon$ for sufficiently small $\varepsilon>0$. The labels $\tau_j$ and $\tau_{j+1}$ at the intersection point change as follows. If $|v_j(t)(t_*)|<|v_{j+1}(t_*)|$, then 
\begin{align}\label{eq:label1}
\tau_j(t_*+\varepsilon)&=\tau_{j+1}(t_*-\varepsilon)\nonumber\\
\tau_{j+1}(t_*+\varepsilon)&=\tau_{j+1}(t_*-\varepsilon)\tau_{j}(t_*-\varepsilon)\tau_{j+1}(t_*-\varepsilon),
\end{align}
and if $|v_j(t)(t_*)|>|v_{j+1}(t_*)|$, then 
\begin{align}\label{eq:label2}
\tau_j(t_*+\varepsilon)&=\tau_j(t_*-\varepsilon)\tau_{j+1}(t_*-\varepsilon)\tau_j(t_*-\varepsilon)\nonumber\\
\tau_{j+1}(t_*+\varepsilon)&=\tau_{j}(t_*-\varepsilon).
\end{align}

In other words, the indices of the labels are swapped and the label of the critical value with smaller absolute value is conjugated by the label of the other critical value. In Figure~\ref{fig:rampi}a) we only label arcs of the curves with transpositions (as opposed to every single point), since it is understood that the labels do not change along the arcs.

%Which of the band relations in Eq.~\eqref{eq:1rel} and Eq.~\eqref{eq:2rel} describes the change of the labels is determined by the absolute values of $v_j(t_*)$ and $v_{j+1}(t_*)$. The label of the critical value with larger absolute value remains unchanged, while the label of the critical value with small absolute value is conjugated by the label of the other critical value. 

If $g_t$ is a loop, then its square diagram can be interpreted as a torus as the critical values at $t=0$ along with their labels match those at $t=2\pi$. The roots of $g_t$ form a P-fibered geometric braid if and only if the curves in the corresponding square diagram are strictly monotone increasing or strictly monotone decreasing, i.e., they can be locally interpreted as graphs of strictly monotone increasing functions $t(\varphi)$, (corresponding to $\tfrac{\partial \arg(v_j(t))}{\partial t}>0$) or strictly monotone decreasing functions (corresponding to $\tfrac{\partial \arg(v_j(t))}{\partial t}<0$). In this case we call the square diagram a \textit{Rampichini diagram}.

Instead of using transpositions as labels we may use band generators, where $(i\to j)$ is replaced by $a_{i,j}$ if the corresponding line in the Rampichini diagram is strictly monotone increasing and by $a_{i,j}^{-1}$ if the line is strictly monotone decreasing. This way the rules Eq.~\eqref{eq:label1} and Eq.~\eqref{eq:label2} become the band relations  Eq.~\eqref{eq:1rel} and Eq.~\eqref{eq:2rel}.  

Rampichini diagrams have the following nice properties. At every fixed value of $t$ apart from those with non-distinct arguments of critical values the labels $\tau_j(t)\in S_n$ at that height $t$, indexed with increasing $\varphi$, satisfy $\underset{j=1}{\overset{n}{\prod}}\tau_j(t)=(1\to n\to n-1\to \ldots,\to3\to2)$, since we have a cactus at all such values of $t$. At every fixed value of $\varphi$ (away from intersection points of curves in the square) the labels (thought of as BKL-generators) spell a band word for the fiber surface $F_{\varphi}=\arg(g_t)^{-1}(\varphi)$ when read from the bottom to the top (i.e., with increasing $t$).

Every P-fibered braid can be represented by a Rampichini diagram (after a small homotopy of $g_t$ in $X_n$ to ensure that the the critical values are distinct) and, conversely, every Rampichini diagram describes a P-fibered braid \cite{bode:braided, rampi}. We say that a BKL-word $w$ is a \textit{P-fibered braid word} if it appears in a Rampichini diagram $R$ as the band word for a fiber surface $F_{\varphi}$. In this case we also say that $R$ is a Rampichini diagram for the P-fibered braid word $w$.

In order to keep Rampichini diagrams clear and less cluttered, we will from now on omit the labels apart from those at the edges of the square. Instead we will at each intersection point in the diagram mark one curve as the undercrossing curve by deleting the curve in a small neighbourhood of the intersection (as is common in knot diagrams), cf. Figure~\ref{fig:rampi}b). By choosing the curve whose critical value has the smaller absolute value to be the undercrossing strand, we can move between these two visual representations of Rampichini diagrams without losing any information, since at each crossing the label of the undercrossing curve is conjugated by the label of the overcrossing curve.

\begin{remark}
Originally, the decision to use crossings (as in knot diagrams) instead of labels in Rampichini diagrams was motivated by purely aesthetic aspects, as Rampichini diagrams with many intersection points needed many labels, which makes the visual representation rather confusing. However, it turns out that it is no coincidence that the resulting Figure~\ref{fig:rampi}b) is a link diagram of a link in a thickened torus (i.e., a virtual link) with some labels on the edge of the square. We know that every P-fibered braid corresponds to a Rampichini diagram and vice versa. The critical values $v_j(t)$ of the corresponding polynomial $g_t$, which are known to be non-zero, form a closed (affine) braid 
\begin{equation}
\underset{t\in[0,2\pi]}{\cup}\underset{j=1}{\overset{n}{\cup}}(v_j(t),\rme^{\rmi t})
\end{equation} 
in $(\mathbb{C}\backslash\{0\})\times S^1$, which is a thickened torus (see \cite{bode:adicact} and also recent work by Manturov and Nikonov \cite{manturov}). It is this link that appears in a Rampichini diagram with our new crossing convention. The resulting virtual link is braided both in the $t$-direction and the $\varphi$-direction, i.e., transverse to all lines of constant $t$ and all lines of constant $\varphi$. Any virtual Reidemeister move that preserves this braiding property corresponds to a homotopy of the loop in the space of critical values, which lifts to a homotopy of the loop in $X_n$ \cite{bode:braided}. In other words, the braided open book described by a Rampichini diagram does not change under isotopies of the corresponding virtual link as long as the link remains doubly-braided throughout the isotopy. (In fact, it is sufficient if it remains braided in the $t$-direction throughout the isotopy and ``doubly braided'' at the end of the isotopy.) We thus have a 1-1-correspondence between Rampichini diagrams and links in a thickened torus that are ``doubly braided'' and labeled with certain transpositions. (Note that not every choice of labels on the boundary of the square diagram can be completed to a Rampichini diagram \cite{mortramp,rampi}.)
\end{remark}

\begin{definition}\label{def:simple}
We call a Rampichini diagram $R$ simple if the corresponding banded fiber surface consists of $n$ disks connected by exactly $n-1$ bands. In other words, any horizontal line $(t=constant)$ has $n-1$ intersections with the curves in $R$ (counted with multiplicity) and any vertical line $(\varphi=constant)$ also has $n-1$ intersections with the curves in $R$ (counted with multiplicity).
\end{definition}
%Geometrically this means that for every fiber disk $D_t$ there are $n-1$ tangential intersection points with the collection of fibers $F_{\varphi}$ and for every fiber surface $F_{\varphi}$ there are $n-1$ tangential intersection points with the collection of fiber disks $D_t$.

Every fiber surface $F_{\varphi}$ coming from a simple Rampichini diagram is a braided surface with $n$ disks and $n-1$ bands. Thus they all have genus 0 and a connected boundary. Therefore, they are bounded by an unknot. The Rampichini diagram in Figure~\ref{fig:rampi}a) is simple. It represents a braid on 4 strands, since every horizontal line has 3 intersection points with the curves in the square (counted with multiplicities), and every fiber surface $F_{\varphi}$ has 3 bands, since each vertical line has 3 intersection points with the curves (again, counted with multiplicities.)

Let 
\begin{equation}
\tilde{V}_n=\{(v_1,v_2,\ldots,v_{n-1})\in(\mathbb{C}\backslash\{0\})^{n-1}:v_i\neq v_j\text{ if }i\neq j\}
\end{equation}
and
\begin{equation}
V_n=\tilde{V}_n/S_{n-1},
\end{equation}
where the action of the symmetric group $S_{n-1}$ permutes the different $v_i$.  Let $X'_n$ be the space of polynomials in $X_n$ with distinct critical values. Then we may think of $V_n$ as the space of sets of critical values of polynomials in $X'_n$. The map $\theta_n:X'_n\to V_n$ that sends a polynomial to its set of critical values maps a loop $g_t$ in $X'_n$ to a loop in $V_n$.

\begin{definition}\label{def:comp}
Let $g_t$ be a loop in $X'_n$ and $R$ its Rampichini diagram. Then we say that $R$ is pure if the loop $\theta_n(g_t)$ in $V_n$ lifts to a loop in $\tilde{V}_n$.
\end{definition}

In other words, a Rampichini diagram is pure if the corresponding critical values form a pure braid in $(\mathbb{C}\backslash\{0\})\times S^1$. In this case we have two different ways to index the critical values. The first one is described above, where at every value of $t$, we order the critical values by their arguments. This means that the indices are not constant along curves of critical values. They change at intersection points in the Rampichini diagram and at the right edge of the square. The second one is given by the indices at $t=0$ and keeping the indices constant along each curve of critical values, i.e., the indexing comes from thinking of $\theta_n(g_t)$ as a loop in $\tilde{V}_n$. In order to distinguish these two different orderings of critical values, we write $v(t)_j$ for the first ($0\leq\arg(v(t)_j)<\arg(v(t)_k)<2\pi$ if and only if $j<k$) and $v_j(t)$ for the second ($j<k$ if and only if $0\leq\arg(v(0)_j)<\arg(v(0)_k)<2\pi$). Note that the index of a transposition $\tau_j(t)$ always refers to the indexing of the first kind, i.e., $\tau_j(t)$ is the transposition associated with $v(t)_j$.

Since the critical values of loops of polynomials corresponding to P-fibered braids satisfy $\tfrac{\partial\arg(v_j(t))}{\partial t}\neq 0$, every curve of critical values $v_j(t)$ has an associated sign $\varepsilon_j\in\{\pm 1\}$, which is equal to $\sign\left(\tfrac{\partial\arg(v_j(t))}{\partial t}\right)$ and which determines whether the corresponding curve in the Rampichini diagram is strictly monotone increasing or decreasing. We call $\varepsilon_j$ the \textit{sign of the critical value} $v_j(t)$ and say that a critical value $v_j$ is positive/negative if $\varepsilon_j$ is positive/negative.

The composition of two paths in $V_n$ corresponds to a square diagram that is obtained by gluing the top edge of one square diagram along the bottom edge of the other square diagram. Since the paths have matching endpoints, the labels and endpoints in the respective square diagrams also match. Likewise, if the $\varphi$-coordinates and corresponding labels at the top edge of one square diagram match those at the bottom edge of another square diagram, we may glue the two diagrams along these edges to obtain a new diagram.

\subsection{Inner loops}
Let $(v_1,v_2,\ldots,v_{n-1})\in\tilde{V}_n$ with $\arg(v_j)\neq\arg(v_k)$ if $j\neq k$. The $v_i$s are not necessarily ordered by their arguments.

\begin{definition}
An inner loop is a loop $\gamma_{j}$ (or $\gamma_{j}^{-1}$) in $\tilde{V}_n$ based at $(v_1,v_2,\ldots,v_{n-1})$, where every critical value except one is stationary and one critical value $v_{j}$ moves in a counterclockwise (or clockwise) loop exactly once around the origin, such that whenever $\arg(v_{j}(t))=\arg(v_i(t))$ with $j\neq i$ we have $|v_{j}(t)|<|v_i(t)|$. We call $v_{j}$ the moving critical value of this inner loop.
\end{definition}

\begin{figure}[H]
\centering
\labellist
\pinlabel $0$ at 780 840
\pinlabel $v_1(t)$ at 1300 1150
\pinlabel $v_2(t)$ at 100 250
\pinlabel $v_3(t)$ at 1100 150
\pinlabel $v_4(t)$ at 1470 580
\endlabellist
\includegraphics[height=5.5cm]{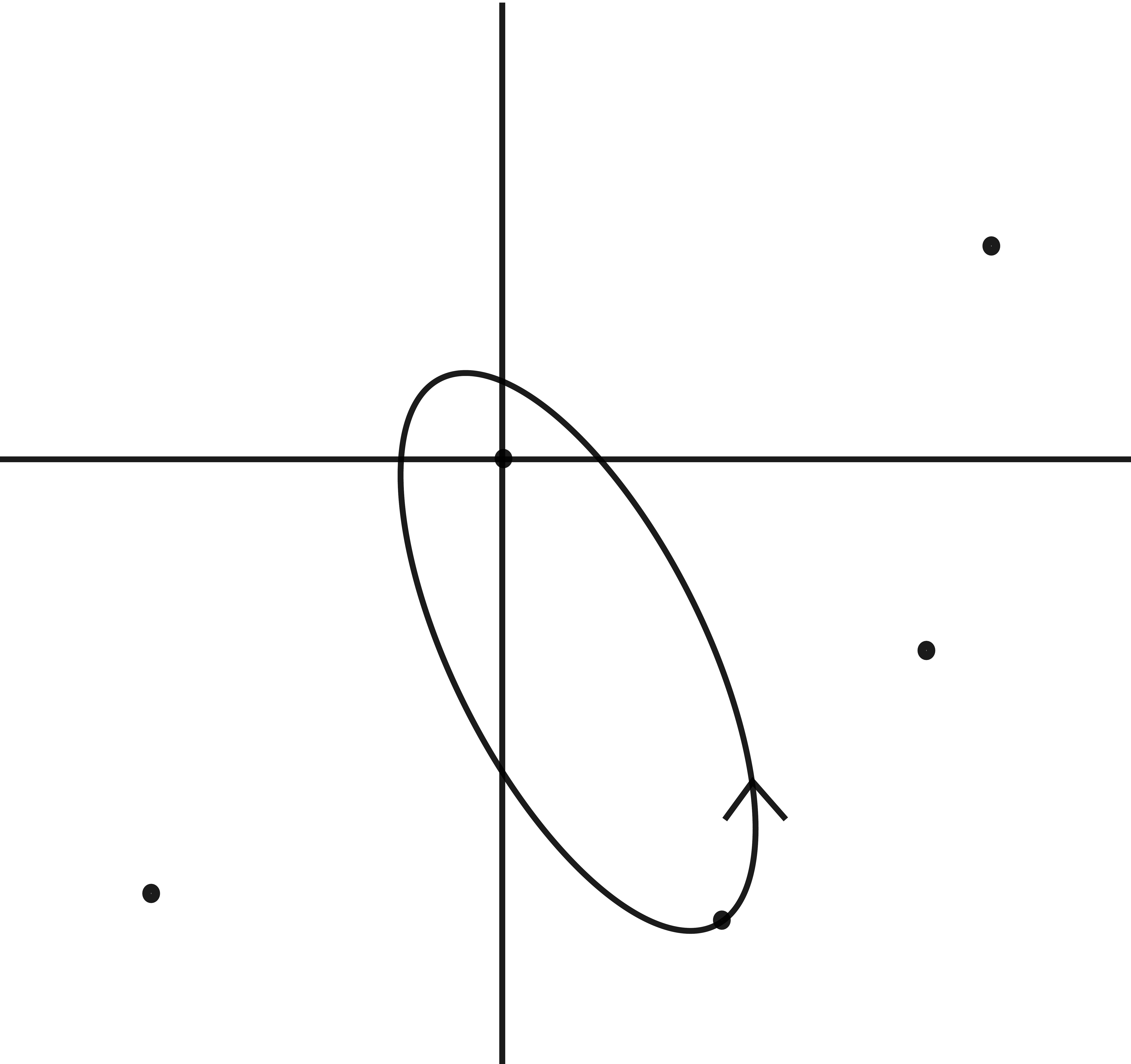}
\caption{The inner loop $\gamma_3$. \label{fig:innerloop}}
\end{figure}
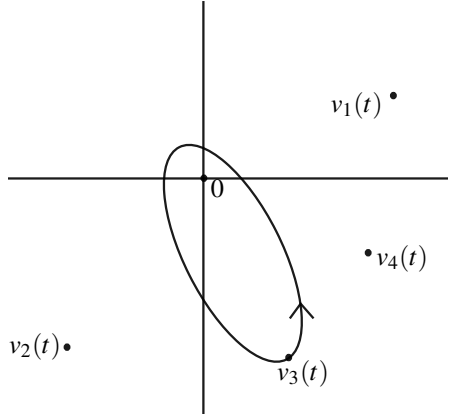

Figure~\ref{fig:innerloop} shows the motion of the critical values during an inner loop. For a path $\gamma$ in a topological space parametrised by $t$ in $[0,2\pi]$ we write $\gamma|_{a}^b$ for the segment of the path $\gamma$ between $\gamma(a)$ and $\gamma(b)$.

\begin{definition}
Let $g_t$ a loop in $X'_n$, whose roots form a P-fibered braid and let $R$ be the corresponding Rampichini diagram, which we assume to be pure. Then $\theta_n(g_t)$ can be viewed as a loop in $\tilde{V}_n$. Let $t_i\in[0,2\pi]$, $i=1,2,\ldots,M$, be such that the arguments of the critical values of $g_{t_i}$ are distinct for each fixed $t_i$ and let $j(i)\in\{1,2,\ldots,n-1\}$, $i=1,2,\ldots,M$. Let $\varepsilon_{j(i)}$ be the sign of the critical value $v_{j(i)}(t)$.
Let $R'$ be the square diagram corresponding to the following loop in $\tilde{V}_n$: 
\begin{equation}\label{eq:loop}
\theta_n(g_t)|_0^{t_1}\circ\gamma_{j(1)}^{\varepsilon_{j(1)}}\circ \theta_n(g_t)|_{t_1}^{t_2}\circ\gamma_{j(2)}^{\varepsilon_{j(2)}}\circ\ldots\theta_n(g_t)|_{t_i}^{t_{i+1}}\circ\gamma_{j(i+1)}^{\varepsilon_{j(i+2)}}\circ\ldots\circ\gamma_{j(M)}^{\varepsilon_j(M)}\circ\theta_n(g_t)|_{t_M}^{2\pi}.
\end{equation}
Then we say that $R'$ is obtained from $R$ by inserting inner loops.
\end{definition}

First note that every loop in $\tilde{V}_n$ corresponds to the critical values of some path in $X_n$ \cite{bode:braided}, so that $R'$ is actually well-defined as the square diagram of that path. (The path is not unique, but the choice of path does not matter.) For the construction of $R'$ we may assume that the loop in Eq.~\eqref{eq:loop} is rescaled, so that it is parametrised by $t\in[0,2\pi]$. The following lemma establishes that the corresponding path in $X'_n$ is actually a loop and moreover, $R'$ is a Rampichini diagram.

\begin{lemma}\label{lem:insert}
Let $R'$ be a square diagram that is obtained from a Rampichini diagram $R$ by inserting inner loops. Then, after a small isotopy of the loop in Eq.~\eqref{eq:loop}, $R'$ is a Rampichini diagram.
\end{lemma}
\begin{proof}
Figure~\ref{fig:rampi_twist} shows how the insertion of an inner loop affects a Rampichini diagram. It corresponds to the insertion of a line that is almost horizontal and that loops around the $\varphi$-coordinate exactly once. Since $|v_{j}(t)|<|v_i(t)|$ at all intersection points, the inserted line lies below all other curves in the diagram and thus its label must be conjugated at each crossing point in the diagram, while the other labels remain unaffected. Since the sign of the inner loop $\gamma_j^{\varepsilon_j}$ matches the sign $\varepsilon_{j}$ of the label of the moving critical value $v_j(t)$, the curves in $R'$ are again strictly monotone after a small deformation of the curves (so that the critical values that are not moving during an inner loop become non-stationary). We now have to show that the labels $\tau_j(2\pi)$ at the top edge of the square diagram are identical to the labels $\tau_j(0)$ at the bottom edge. By the arguments in the proof of Theorem 5.6 in \cite{bode:braided} this implies that the corresponding path in $X'_n$ is a loop and hence $R'$ is a Rampichini diagram.

\begin{figure}[H]
\labellist
\Large
\pinlabel a) at  50 840
\pinlabel b) at 1000 840
\small
\pinlabel 0 at 150 110
\pinlabel 0 at 210 50
\pinlabel $2\pi$ at 850 50
\pinlabel $2\pi$ at 140 750
\pinlabel $(1,4)$ at 900 540
\pinlabel $(3,4)$ at 900 380
\pinlabel $(1,2)$ at 900 210
\pinlabel $(1,4)$ at 400 790
\pinlabel $(1,3)$ at 530 790
\pinlabel $(1,2)$ at 690 790
\pinlabel $(1,3)$ at 1830 255
\pinlabel $(1,3)$ at 610 260
\pinlabel $(1,2)$ at 1830 	200
\pinlabel $(3,4)$ at 1830 380
\pinlabel $(1,4)$ at 1830 530
\pinlabel $(1,4)$ at 1300 790
\pinlabel $(1,3)$ at 1445 790
\pinlabel $(1,2)$ at 1600 790 
\pinlabel 0 at 1080 110
\pinlabel 0 at 1110 50
\pinlabel $2\pi$ at 1750 50
\pinlabel $2\pi$ at 1060 750
\Large
\pinlabel $t$ at 100 450
\pinlabel $\varphi$ at 500 20
\pinlabel $t$ at 1040 450
\pinlabel $\varphi$ at 1400 20
\endlabellist
\centering
\includegraphics[height=5cm]{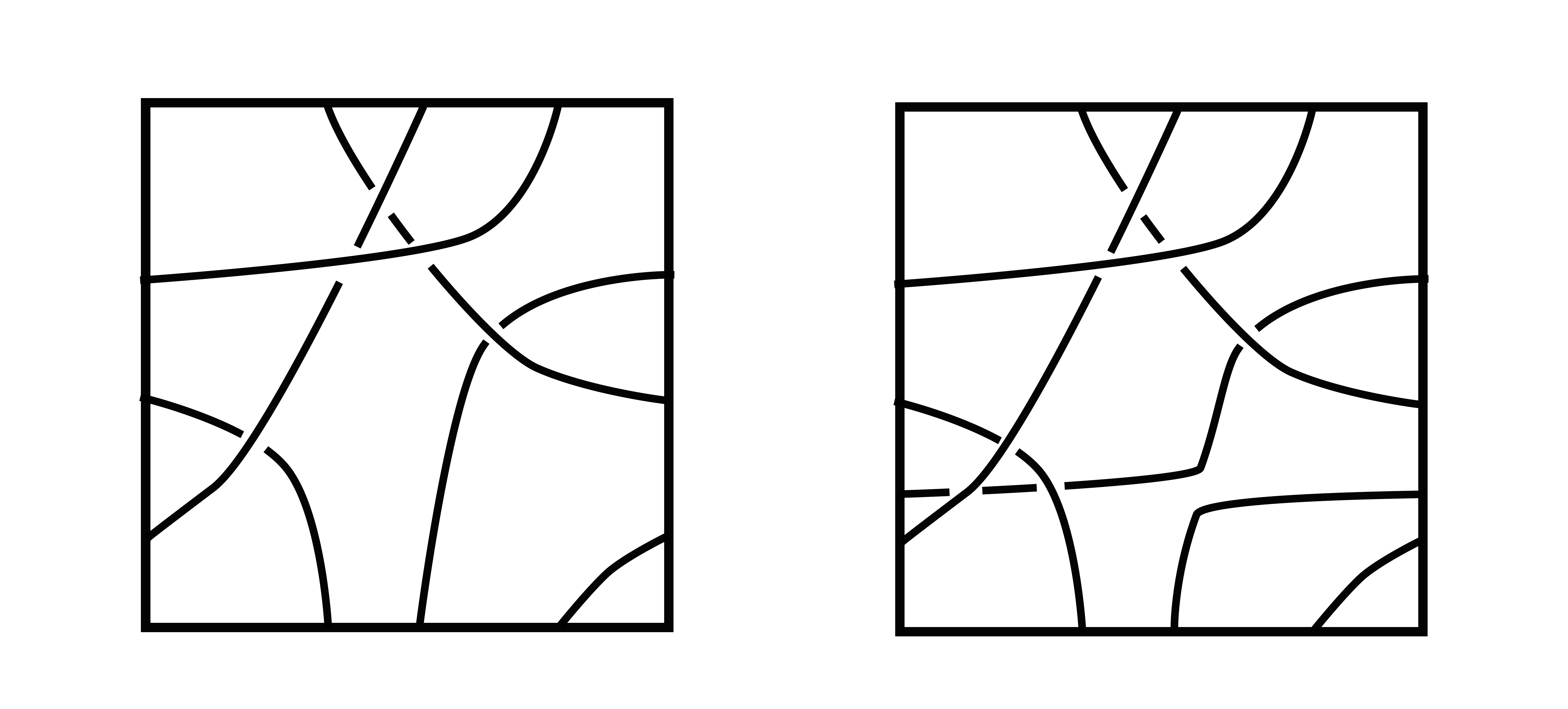}
\caption{a) A simple, pure Rampichini diagram. b) The Rampichini diagram after the insertion of an inner loop.\label{fig:rampi_twist}}
\end{figure}
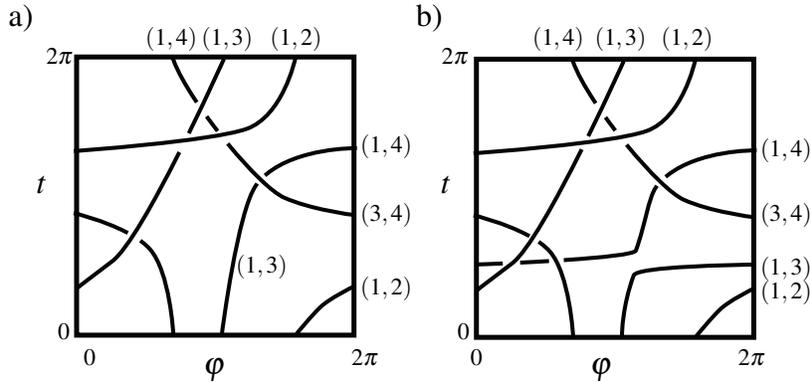

We know that during an inner loop only the label of the moving critical value changes. We have to show that after the completion of the inner loop it is again the original label. Let $\tau_j(t_i)$, $j=1,2,\ldots,n-1$, be the transpositions in $R$ associated with the critical values at the height $t=t_i$ at which an inner loop will be inserted, ordered by increasing value of $\varphi$. Let $k$ be the index of the moving critical value for the inserted inner loop $\gamma_{j(i)}^{\varepsilon_{j(i)}}$, i.e., $v_{j(i)}(t_i)=v(t_i)_k$ in $R$. Let $\tau'_k$ denote its label after the inner loop.

By one of the defining properties of Rampichini diagrams we have that $\underset{j=1}{\overset{n-1}{\prod}}\tau_j(t_i)=(1\to n\to n-1\to\ldots\to3\to2)$. In fact, this equality is true for any height $t$ in a square diagram, not just at $t=t_i$. We thus have
\begin{equation}
\prod_{j=1}^{n-1}\tau_j(t_i)=(1\to n\to n-1\to\ldots\to3\to2)=\left(\prod_{j=1}^{k-1}\tau_j(t_i)\right)\tau'_k\left(\prod_{j=k+1}^{n-1}\tau_j(t_i)\right),
\end{equation}
where the left-hand side describes the cactus at $t=t_i-\varepsilon$ and the right-hand side the cactus at $t=t_i+\varepsilon$ for some small positive $\varepsilon$. This implies that $\tau_k(t_i)=\tau'_k$ after multiplying $\left(\underset{j=1}{\overset{k-1}{\prod}}\tau_j(t_i)\right)^{-1}$ from the left and multiplying $\left(\underset{j=k+1}{\overset{n-1}{\prod}}\tau_j(t_i)\right)^{-1}$ from the right. Thus the labels after an inserted inner loop are the same as the labels before the inner loop. Since $R$ is a Rampichini diagram, the labels at its top edge are equal to the label at its bottom edge. It follows that the same holds for $R'$, which proves the lemma.
\end{proof}

Instead of inserting inner loops into Rampichini diagrams we can also insert them into a diagram as in Figure~\ref{fig:trivial_rampi} corresponding to a trivial loop in $X_n'$, where all curves in the diagram are vertical. We call such a square diagram \textit{trivial}. Note that there are also non-trivial paths, whose square diagrams are trivial, since the only requirement is that the argument of the critical values do not change. Since there are no intersections of the curves with each other or with the right edge of the square, the labels in a trivial square diagram never change. 

\begin{figure}[H]
\labellist
\Large
\pinlabel a) at  50 840
\pinlabel b) at  1000 840
\small
\pinlabel 0 at 150 110
\pinlabel 0 at 210 50
\pinlabel $2\pi$ at 850 50
\pinlabel $2\pi$ at 150 750
\pinlabel $(1,4)$ at 410 790
\pinlabel $(1,3)$ at 530 790
\pinlabel $(1,2)$ at 690 790
\Large
\pinlabel $t$ at 100 450
\pinlabel $\varphi$ at 500 20
\endlabellist
\centering
\includegraphics[height=5cm]{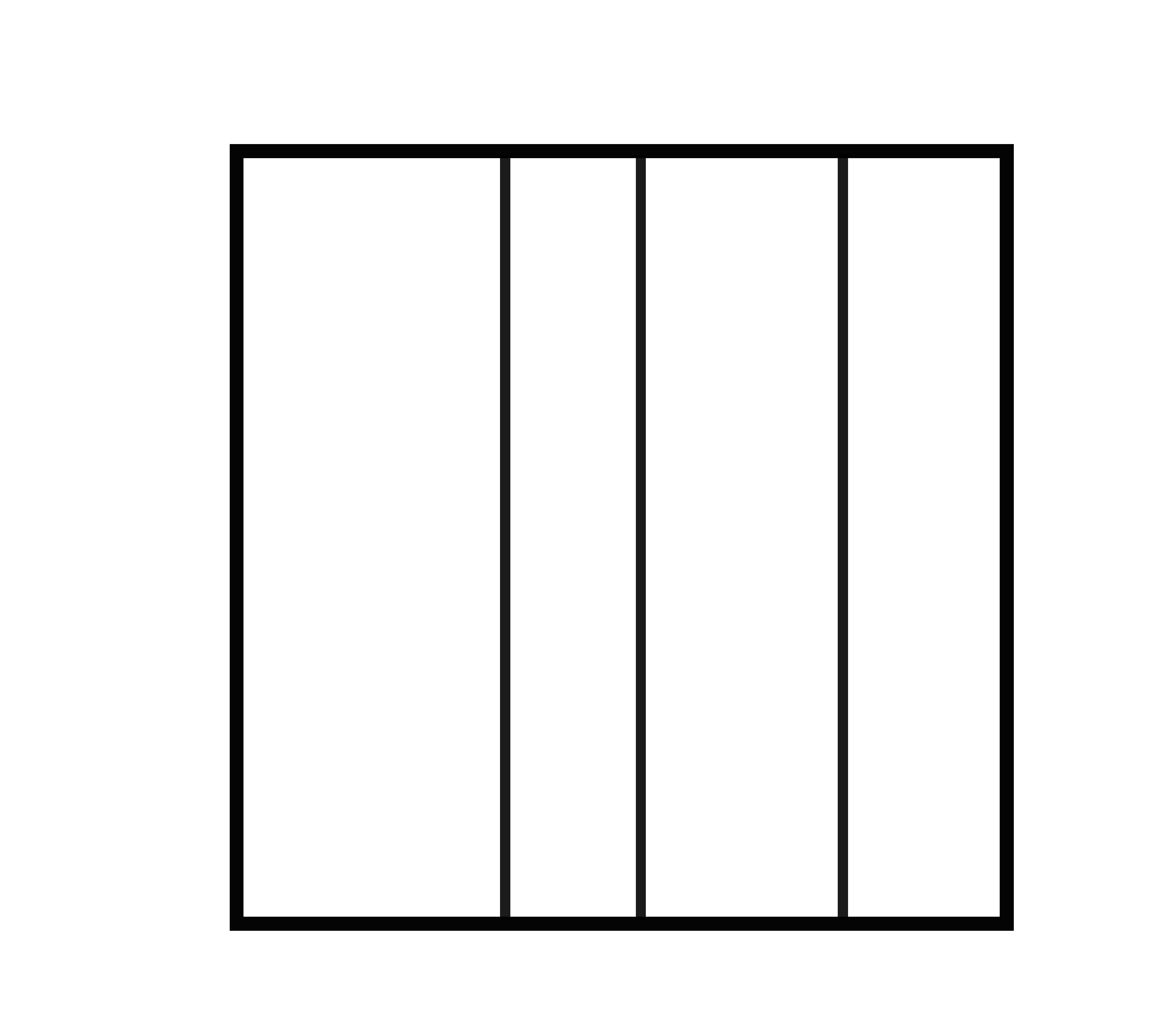}
\labellist
\small
\pinlabel 0 at 60 110
\pinlabel 0 at 110 50
\pinlabel $2\pi$ at 750 50
\pinlabel $2\pi$ at 60 750
\pinlabel $(1,4)$ at 330 790
\pinlabel $(1,3)$ at 450 790
\pinlabel $(1,2)$ at 610 790
\Large
\pinlabel $t$ at 60 410
\pinlabel $\varphi$ at 420 20
\endlabellist
\centering
\includegraphics[height=5cm]{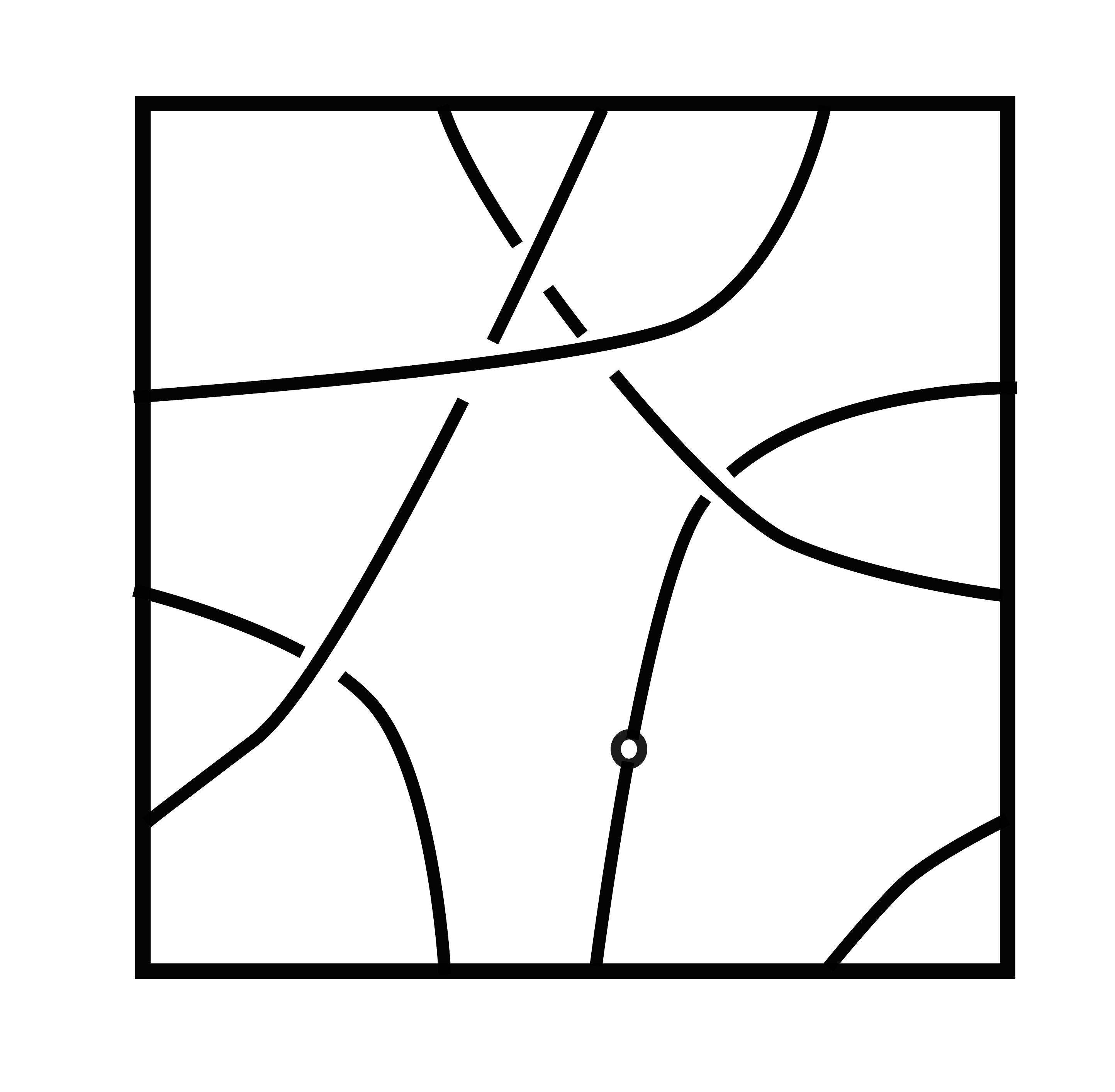}
\caption{a) A trivial square diagram corresponding to a constant loop in the space of polynomials. b) A singular Rampichini diagram.  \label{fig:trivial_rampi}}
\end{figure}
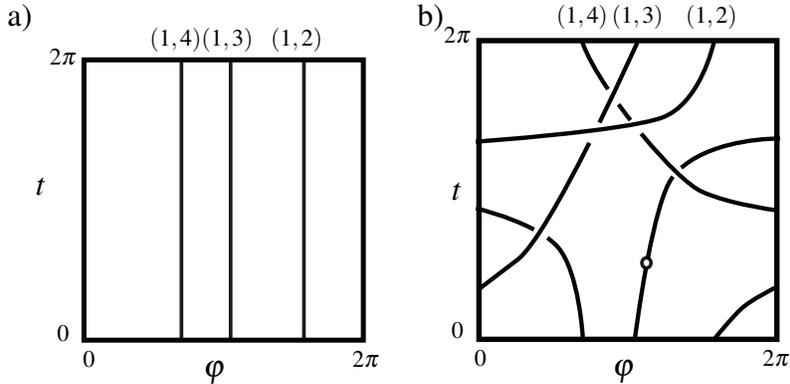

A trivial square diagram exists for any choice of cactus $\tau_j$, $j=1,2,\ldots,n-1$, which by definition satisfies $\underset{i=1}{\overset{n}{\prod}}\tau_i=(1\to n\to n-1\to \ldots\to 3\to 2)$ and signs $\varepsilon_j\in\{\pm1\}$, $j=1,2,\ldots,n-1$. Inserting a finite number of inner loops $\gamma_{j(i)}^{\varepsilon_{j(i)}}$, $i=1,2,\ldots,M$, into the diagram results in a Rampichini diagram $R'$ by the same arguments as in Lemma~\ref{lem:insert}.

%We can now build a motion of the critical values that corresponds to a P-fibered braid and thus to a Rampichini diagram by concatenating the loops $\gamma_j^{\varepsilon_j}$ and deforming it slightly so that $\tfrac{\partial \arg(v_j(t))}{\partial t}\neq 0$ for all $j$ and all $t$.

We can read off a band word for the resulting fiber surface $F_{2\pi-\varepsilon}$ of the corresponding braid from the Rampichini diagram $R'$. Note that every critical value $v_j(t)$ always contributes the same letter, which we call $a_j^{\varepsilon_j}$. Its sign is always equal to the sign $\varepsilon_j$ of the critical value, while its associated transposition is $\tau_{j}^{\prod_{i=j+1}^{n+1}\tau_{i}}$. Here $a^b$ denotes conjugation of a group element $a$ by another group element $b$. The definition of $a_j^{\varepsilon_j}$ clearly depends on the choice of cactus $\{\tau_j\}_{j=1,2,\ldots,n-1}$.

\begin{definition}\label{def:thomo}
A braid $B$ is called T-homogeneous if there is a choice of cactus $\{\tau_j\}_{j=1,2,\ldots,n-1}$ and signs $\varepsilon_j\in\{\pm 1\}$, so that $B$ can be represented by a word that only contains the letters $a_j^{\varepsilon_j}$ and that contains $a_j^{\varepsilon_j}$ for every $j=1,2,\ldots,n-1$.
\end{definition} 

In some texts these braids are also called \textit{strict(ly) T-homogeneous braids} to emphasize the second property, that every $a_j^{\varepsilon_j}$ appears in the word \cite{rudolph2}.

T-homogeneous braids are thus exactly those braids that can be obtained from a trivial square diagram as in Figure~\ref{fig:trivial_rampi}a) by insertion of inner loops. Each inserted inner loop can be interpreted as a loop in $\tilde{V}_n$ and a trivial square diagram corresponds to a constant loop in $\tilde{V}_n$. Therefore every T-homogeneous braid has a pure Rampichini diagram.

\begin{example}
We may choose $\tau_j=(1\to n-j+1)$ and obtain $a_j^{\varepsilon_j}=a_{n-j,n-j+1}^{\varepsilon_j}=\sigma_{n-j}^{\varepsilon_j}$, so that T-homogeneous braids for this choice of cactus are exactly the \textit{homogeneous braids} in Artin generators $\sigma_j$.
\end{example} 

Definition~\ref{def:thomo} describes the family of T-homogeneous braids in terms of their braid words. An alternative, equivalent definition explains the origin of the ``T'' in the name. Let $T$ be an embedded tree in $\mathbb{C}$ with $n$ vertices. We may assign to every edge $e$ a sign $\varepsilon_e\in\{\pm 1\}$. We define a loop $\gamma_e$ in $X_n$ as follows. The basepoint of $\gamma_e$ is the polynomial in $X_n$ whose roots are the positions of the vertices of $T$. Throughout the loop all roots that do not correspond to endpoints of the edge $e$ remain stationary, while the two endpoints of $e$ exchange their position as in Figure~\ref{fig:tree}b), where the way in which they swap depends on the sign $\varepsilon_e$. A geometric braid $B$ is T-homogeneous if its corresponding loop of polynomials $g_t$ is the composition of loops of the form $\gamma_e$, where $e$ is running through the set of edges of $T$, and for every edge $e$ the corresponding loop $\gamma_e$ appears in the factorisation, i.e., $g_t=\prod_{i=1}^{\ell}\gamma_{e_i}$, where $\prod$ denotes composition of loops, and for every edge $e$ there is an $i\in\{1,2,\ldots,\ell\}$ with $e_i=e$. The two definitions are equivalent with different choices of an embedded tree $T$ (up to planar isotopy) corresponding to different choices of a cactus $\{\tau_j\}_{j=1,2,\ldots,n-1}$. Homogeneous braids correspond to a path graph (or ``linear graph'').

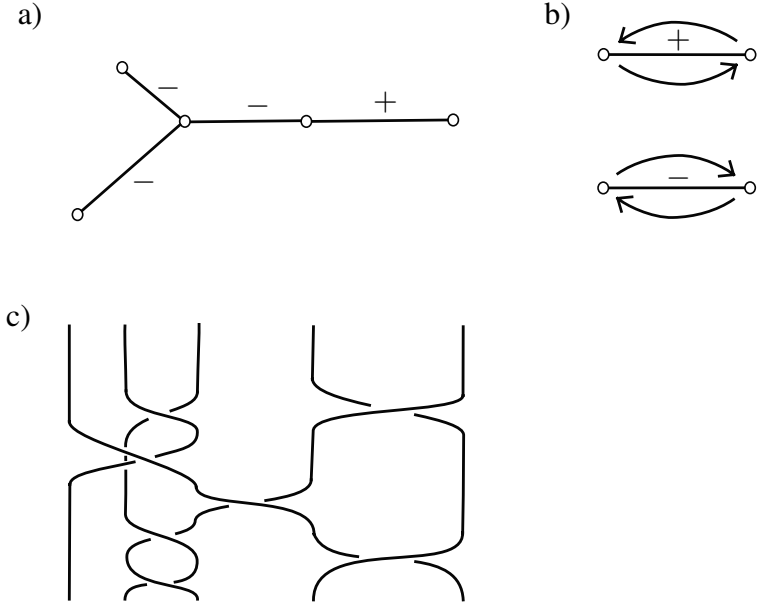
\begin{figure}
\centering
\labellist
\Large
\pinlabel a) at 20 800
\pinlabel b) at 1550 800 
\pinlabel $+$ at 1050 550
\pinlabel $-$ at 680 540
\pinlabel $-$ at 420 590
\pinlabel $-$ at 350 320
\pinlabel $+$ at 1900 730
\pinlabel $-$ at 1900 330
\endlabellist
\includegraphics[height=4.5cm]{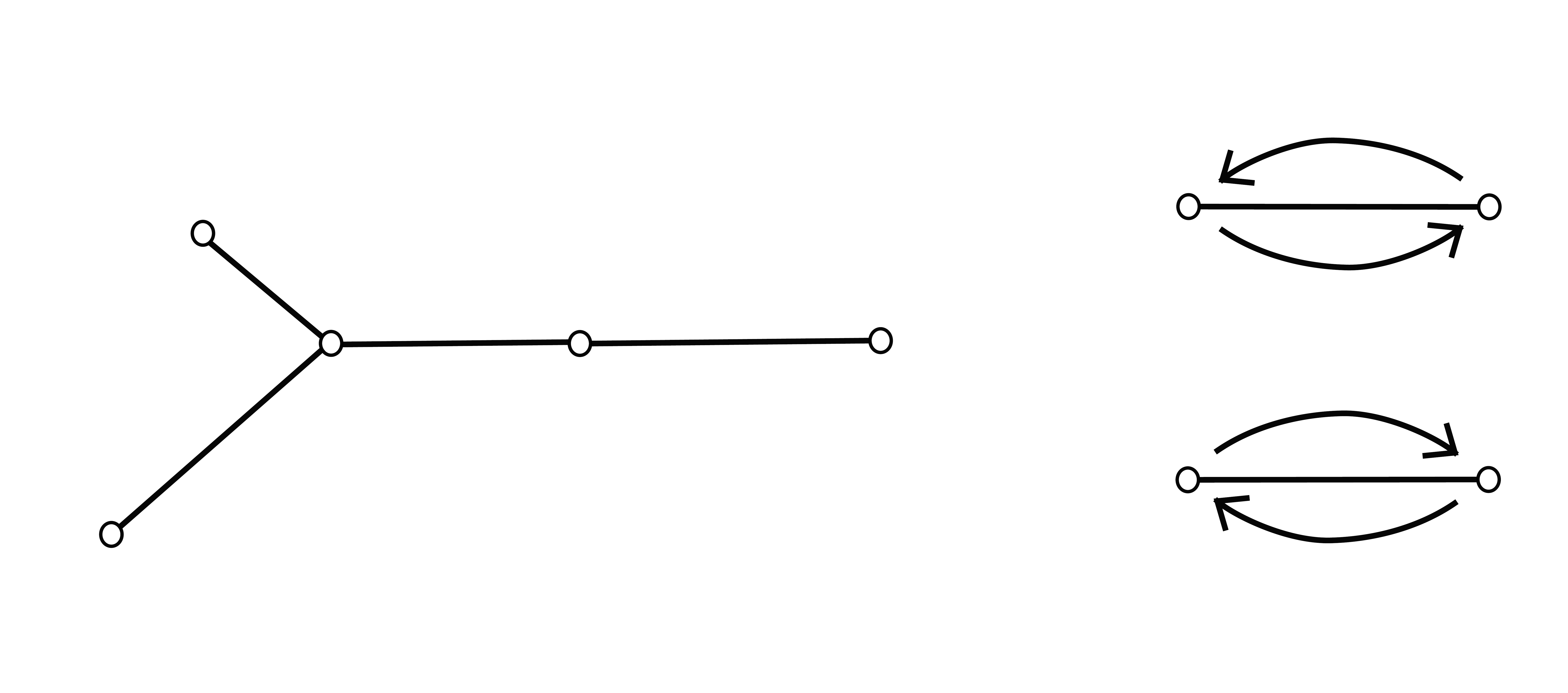}
\labellist
\Large 
\pinlabel c) at 20 880
\endlabellist
\includegraphics[height=4.5cm]{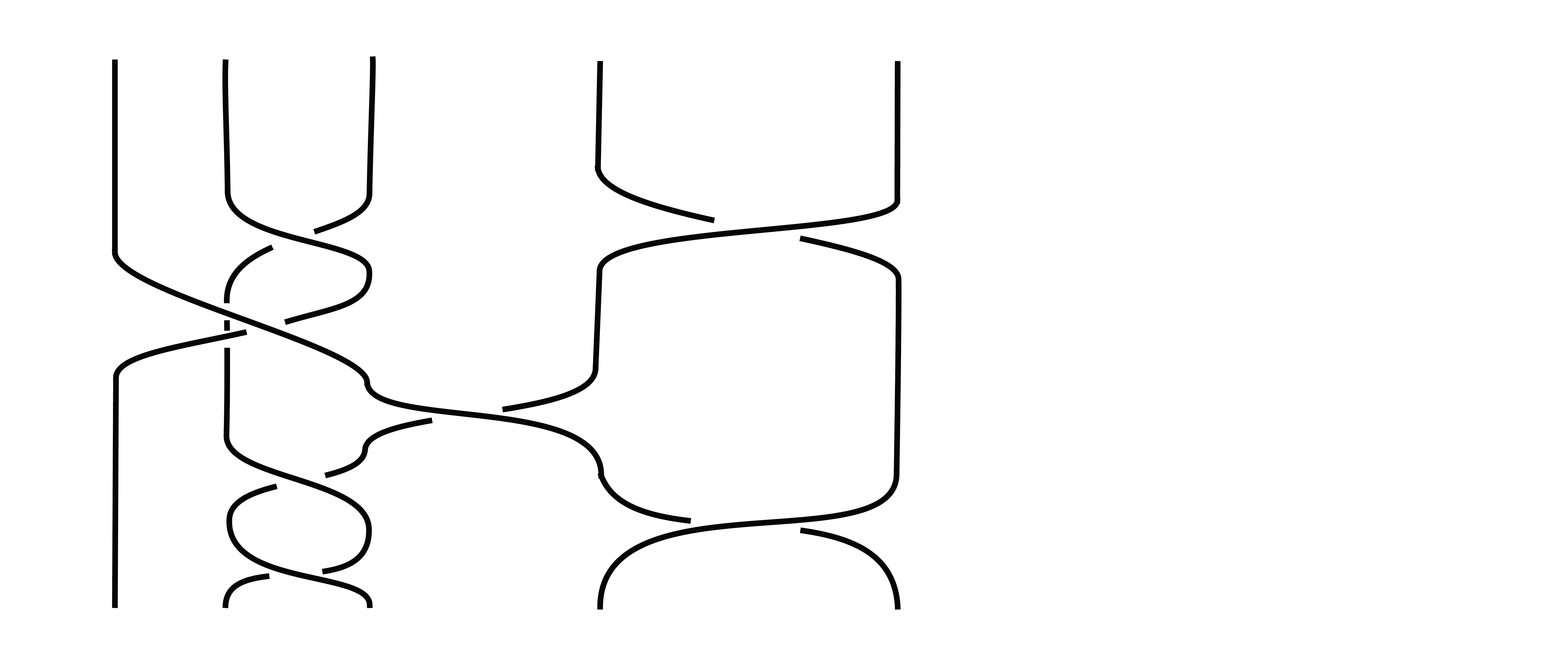}
\caption{a) An embedded tree $T$ in the complex plane. The signs $\varepsilon_e$ are drawn on each edge $e$. b) The loop $\gamma_e$ exchanges the roots on the endpoints of $e$ as in the upper picture if $\varepsilon_e=+1$ and as in the lower picture if $\varepsilon_e=-1$. c) A T-homogeneous braid with $T$ as in Subfigure a). \label{fig:tree}}
\end{figure}

Homogeneous braids are known to close to fibered links \cite{stallings2}. In fact, all T-homogeneous braids are known to close to bindings of totally braided open books \cite{rudolph2}, which is a property that is equivalent to being P-fibered \cite{bode:braided}.

Let $B$ be a P-fibered geometric braid with corresponding loop of polynomials $g_t$ and a pure Rampichini diagram $R$. Deforming its loop of critical values $(v_1(t),v_2(t),\ldots,v_{n-1}(t))$ in $V_n$ lifts to a deformation of $g_t$ and as long as the deformation only changes the absolute values of the $v_j(t)$s, the Rampichini diagram and the fiberedness property remain unchanged. Now consider such a deformation, where we only change one of the curves of critical values $v_j(t)$ such that it becomes 0 at some $t=\tau\in[0,2\pi]$ and call the lifted loop $\tilde{g}_t$. Then $\tilde{g}_{\tau}$ has a double root and therefore the roots of $\tilde{g}_t$ do not form a braid, but a singular braid. Since the argument of the critical value at $t=\tau$ is not defined, we cannot associate to it a label or a transposition like we usually do in Rampichini diagrams. We define a \textit{singular Rampichini diagram} $R_{\{(t_i,j(i)\}_{i=1,2,\ldots,M}}$ to be a Rampichini diagram $R$ where at a finite number of distinct values of $t=t_i$, $i=1,2,\ldots,M$, one of the curves, corresponding to the critical value $v_{j(i)}(t_i)$, in $R$ at height $t$ is replaced by a small circle. The rest of the diagram is unchanged. This is displayed in Figure~\ref{fig:trivial_rampi}b), which can be obtained from the Rampichini diagram in Figure~\ref{fig:rampi_twist}a). The circles represent values of $t$ at which a critical value becomes zero. While Rampichini diagrams visualise certain loops in $V_n$ that lift to loops in $X_n$, singular Rampichini diagrams visualise certain loops in the space of $n-1$ distinct complex numbers (a space that is homeomorphic to $X_{n-1}$), which lift to loops in the space of monic polynomials of degree $n$, but not to loops in $X_n$.

Note that not every loop $h_t$ in the space of monic polynomials of fixed degree $n$ gives rise to a singular Rampichini diagram. Apart from the usual condition that $\tfrac{\partial \arg(v_j)}{\partial t}(t)\neq 0$ for all $v_j(t)\neq 0$, it is necessary that $\underset{t\to t_i}{\lim}\arg(v_{j(i)}(t))$ and $\underset{t\to t_i}{\lim}\tfrac{\partial\arg(v_{j(i)}(t))}{\partial t}$ are well-defined, with the latter being non-zero. Furthermore, it was an assumption that there is only a finite number of singular crossings and no higher multiplicity than 2 for any root of $h_t$.

Since the roots of a loop of polynomials $h_t$ corresponding to a singular Rampichini diagram form a singular braid, the level sets of $\arg(h_t)$ are not braided surfaces anymore. Therefore, the property of Rampichini diagrams that we can read off band words for the fibers, is not true anymore.

\section{Odd P-fibered braids}\label{sec:odd}

P-fibered braids play a big role in constructions of real algebraic links. If $B$ is a P-fibered braid, then the closure of $B^2$ is real algebraic \cite{bode:real}. The corresponding semiholomorphic polynomial can be constructed explicitly as
\begin{equation}
f(u,r\rme^{\rmi t})=r^{2kn}g\left(\frac{u}{r^{2k}},\rme^{2\rmi t}\right),
\end{equation}
where $k$ is a sufficiently large integer, $n$ is the number of strands and $g_t(u)=g(u,t)$ is the loop corresponding to the P-fibered geometric braid $B$ parametrised in terms of trigonometric polynomials, so that the coefficients of $g$ (as a polynomial in $u$) are polynomials in $\rme^{\rmi t}$ and $\rme^{-\rmi t}$. Note that $f$ is radially weighted homogeneous, i.e., $f(\lambda^{2k} u,\lambda v)=\lambda^{2kn}f(u,v)$ for all $\lambda\in\mathbb{R}$. This construction requires 2-periodicity of a braid (i.e., $B^2$ instead of $B$), since this guarantees that $f$, which is a priori a polynomial in $u$, $v$, $\overline{v}$ and $\sqrt{v\overline{v}}$, is actually a polynomial in $u$, $v$ and $\overline{v}$, since all terms with square-roots come with an even exponent.

In \cite{bode:real} we discuss another symmetry, which will be essential for the present paper. Instead of the even symmetry of $B^2$, where all frequencies in the trigonometric parametrisation of the roots are even and thus $g_{t+\pi}=g_t$, we require the odd symmetry $g_{t+\pi}(u)=-g_t(-u)$ for all $t\in[0,\pi)$ and all $u\in\mathbb{C}$. In this case, we say that $g_t$ is \textit{odd}.

Polynomials in $\rme^{\rmi t}$ and $\rme^{-\rmi t}$ with complex coefficients are also called finite Fourier series or complex trigonometric polynomial. We reserve the term trigonometric polynomial for finite Fourier series that are real functions. They can thus be expressed as $\mathbb{R}$-linear combinations of 1, $\cos(kt)$ and $\sin(kt)$, where $k$ goes through a finite range of natural numbers. 

We write $f^{(i)}$ for the $i$th derivative of a function $f$. The $C^k$-norm is denoted by $|f|_k=\underset{i\leq k}{\sup}\ \underset{x\in\mathbb{R}}{\sup}|f^{(i)}(x)|$. 
%This means that the coefficient of $u^j$ is an odd function

% Since trigonometric polynomials with odd frequencies are dense in the space of odd periodic functions with the $C^k$-norm, this is equivalent to requiring that $g_{t+\pi}(u)=-g_t(-u)$ for all $t\in[0,\pi)$ and all $u\in\mathbb{C}$. 

A weaker version of the following result, where all components were required to have the same number of strands, was proved in Lemma V.4 in \cite{bode:real}.

\begin{theorem}\label{thm:weighted}
Let $B$ be a P-fibered geometric braid on $n$ strands with corresponding loop of polynomials $g_t$. Assume that the coefficients of $g_t$ are finite Fourier series. If $n$ is odd and $g_t$ is odd, then the closure of $B$ is real algebraic, with the corresponding polynomial map given by
\begin{equation}\label{eq:weighted}
f(u,r\rme^{\rmi t})=r^{kn}g\left(\frac{u}{r^{k}},\rme^{\rmi t}\right),
\end{equation}
where $g(u,t)=g_t(u)$ and $k$ is a large odd integer. 
\end{theorem}
\begin{proof}
Eq.~\eqref{eq:weighted} can be rewritten as
\begin{equation}
f(u,v)=\sqrt{v\overline{v}}^{kn}g\left(\frac{u}{\sqrt{v\overline{v}}^k},\frac{v}{\sqrt{v\overline{v}}}\right).
\end{equation}

As in the case of the even symmetry of $B^2$, choosing $k$ sufficiently large clears the denominator and since $k$, $n$ and $g_t$ are all odd, all monomials involving $\sqrt{v\overline{v}}$ come with an even exponent, so that $f$ is a polynomial in $u$, $v$ and $\overline{v}$.

Note that $f$ is radially weighted homogeneous, since $f(\lambda^k u,\lambda v)=\lambda^{kn}f(u,v)$. It then follows by the same arguments as in the proof of Theorem I.1 in \cite{bode:real} that the origin is an isolated singularity if and only if $B$ is P-fibered. Furthermore, for each fixed $r$ the zeros of $f|_{|v|=r}$ form the closed braid $B$ in $\mathbb{C}\times rS^1$. By the same arguments as in \cite{bode:polynomial} or \cite{bode:real} the link of the singularity is the closure of $B$.
\end{proof}

Note that the condition that coefficients of $g_t$ should be finite Fourier series is not really a restriction. For an odd loop of polynomials $g_t$ with $\deg_u(g_t)$ odd, the coefficient $a_j(t)$ of $u^j$ is an odd function, i.e., $a_j(t+\pi)=-a_j(t)$, if $j$ is even, and an odd function, i.e., $a_j(t+\pi)=a_j(t)$ if $j$ is odd. Odd/even trigonometric polynomials are dense in the the space of odd/even $2\pi$-periodic functions in the $C^k$-norm for any natural number $k$. Therefore, if $B$ is a P-fibered geometric braid, whose corresponding loop of polynomials $g_t$ is odd, then an arbitrarily close approximation of $B$ is an isotopic P-fibered braid, whose corresponding loop of polynomials is odd and has coefficients that are finite Fourier series.

\begin{definition}
A cactus $\{\tau_j\}_{j=1,2,\ldots,2n-2}$ with $\tau_j=(k(j)\to\ell(j))$ is called odd if it satisfies 
\begin{equation}
\tau_{j+n-1}=(k(j)+n-1\ (\text{mod }2n-1)\to \ell(j)+n-1\ (\text{mod }2n-1))
\end{equation}
for all $j=1,2,\ldots,n-1$.
\end{definition}

\begin{lemma}\label{lem:oddcactus}
Let $g:\mathbb{C}\to\mathbb{C}$, $g(u)=\underset{j=1}{\overset{2n-1}{\prod}}(u-z_j)$ be a monic polynomial of degree $2n-1$ with an odd cactus and critical values $v_j$, $j=1,2,\ldots,n-1$, that satisfy 
\begin{align}
0&< \arg(v_1)<\arg(v_2)<\ldots<\arg(v_{n-1})<\pi\nonumber\\
&<\arg(v_n)<\arg(v_{n+1})<\ldots<\arg(v_{2n-2})<2\pi.
\end{align} 
Then $\tilde{g}:\mathbb{C}\to\mathbb{C}$, $\tilde{p}(z)=\underset{j=1}{\overset{2n-1}{\prod}}(u+z_j)$ has the same cactus as $g$. 
\end{lemma}
\begin{proof}
By assumption $n-1$ critical values $v_j$, $j=1,2,\ldots,n-1$, of $g$ have an argument between $0$ and $\pi$, while the critical values $v_{j+n-1}$, $j=1,2,\ldots,n-1$, of $g$ have an argument between $\pi$ and $2\pi$.
Note that the critical values of $\tilde{g}$ are $\tilde{v}_{j+n-1}=-v_j$, $j=1,2,\ldots,2n-2$, where the indices are taken mod $2n-2$. Hence the critical values $\tilde{v}_j$, $j=1,2,\ldots,n-1$, of $\tilde{g}$ have an argument between $0$ and $\pi$ and the critical values $\tilde{v}_{j+n-1}$, $j=1,2,\ldots,n-1$, have an argument between $\pi$ and $2\pi$.

The singular foliation of the disk induced by $\arg(\tilde{g})$ is exactly the singular foliation induced by $\arg(g)$ rotated by $\pi$. Thus if $j\in\{1,2,\ldots,n-1\}$ and $\tau_j=(k(j)\to\ell(j))$, then $\arg(\tilde{v}_{j+n-1})\in(\pi,2\pi)$ with label $\tilde{\tau}_{j+n-1}=(k(j)+n-1\to\ell(j)+n-1)=\tau_{j+n-1}$, since $g$ has an odd cactus.

Likewise, from $\tau_{j+n-1}=(k(j)+n-1\ (\text{mod }2n-1)\to\ell(j)+n-1\ (\text{mod }2n-1))$, $j=1,2,\ldots,n-1$, we get that $\tilde{v}_j$, $j=1,2,\ldots,n-1$, has the label $\tilde{\tau}_j=(k(j)\to\ell(j))=\tau_j$. Thus $\tilde{g}$ has the same cactus as $p$.
\end{proof}

\begin{lemma}\label{lem:trivialpath}
Let $g$ be a polynomial in $X'_{2n-1}$ with an odd cactus and critical values $v_j$, $j=1,2,\ldots,2n-2$, that satisfy the following property. For every $j\in\{1,2,\ldots,2n-2\}$ there is a $k\in\{1,2,\ldots,2n-2\}$ such that $\arg(v_k)=\arg(v_j)+\pi$. Then there is a path in $X'_{2n-1}$ from $g$ to $\tilde{g}$, whose square diagram is trivial, where $\tilde{g}$ is as in Lemma~\ref{lem:oddcactus}.
\end{lemma}
\begin{proof}
Riemann's uniqueness theorem states that (once we have fixed a labeling of the arcs $A_j$, $j=1,2,\ldots,2n-1$, on $\partial D$ as we have done at the beginning of Subsection~\ref{subsec:rampi}) for every cactus and set of critical values $v_j$ there is a unique monic polynomial with the given cactus and set of critical values $v_j$, up to translations in the complex plane $z\mapsto z+b$ \cite{bode:braided}.

In particular, we can translate the roots of $g$ in parallel such that one of them becomes 0 without changing the cactus or the set of critical values. In other words, fixing one $k\in\{1,2,\ldots,2n-1\}$ the path in $X'_{2n-1}$ that corresponds to the path of polynomials with roots $z_j(t)=z_j-\tfrac{t}{2\pi}z_k$, $t$ going from 0 to $2\pi$, has a trivial square diagram.

%Likewise there is a path in $X'_{2n-1}$ that starts at $\tilde{g}$ and ends at a translation of $\tilde{g}$ such that one of the roots of the endpoints of the path is 0 and such that the square diagram of the path is trivial.

By Lemma~\ref{lem:oddcactus} $g$ and $\tilde{g}$ have the same cactus and by the given symmetry of the critical values $v_j$ of $g$, the critical values of $\tilde{g}$ have the same arguments as $v_j$, $j=1,2,\ldots,2n-2$. Therefore there is a path in $V_{2n-1}$ that does not change the arguments of the critical values and goes from the set of critical values of $g$ to the set of critical values of $\tilde{g}$. Then this path lifts to a unique path in $X'_{2n-1}$ if we require that throughout the path the constant term of every polynomial is 0 \cite{critical}. By Riemann's uniqueness theorem this path thus goes from the translated $g$ to a translation of $\tilde{g}$. Since the arguments of the critical values do not change along the corresponding path in $V_{2n-1}$, the path in $X'_{2n-1}$ has a trivial square diagram. The composition of the path that translates $g$ and this path, followed by a translation that ends at $\tilde{g}$ is then the desired path in $X'_{2n-1}$ with a trivial square diagram. 
\end{proof}

We know from \cite{bode:braided} that every Rampichini diagram arises from some loop $g_t$ in $X_n$ corresponding to a P-fibered braid.
\begin{definition}\label{def:odd}
We say that a Rampichini diagram is odd if there is a corresponding loop $g_t$ in $X_n$ that is odd.
\end{definition}

We define $m_n:\mathbb{B}_n\to\mathbb{B}_{2n-1}$ to be the the group homomorphism that sends a generator $a_{i,j}$ with $i,j\neq n$ to $a_{i+n,j+n}$ and $a_{i,n}=a_{n,i}$ to $a_{n,i+n}$. Furthermore, we write $\imath_n:\mathbb{B}_n\to\mathbb{B}_{2n-1}$ for the inclusion, which sends $a_{i,j}$ to $a_{i,j}$.

\begin{lemma}
Let $B$ be a P-fibered braid word on $n$ strands. Then there is an odd Rampichini diagram for the P-fibered braid word $\imath_n(B)m_{n}(B)$.
\end{lemma}
\begin{proof}
%Let $g_t$ be the loop of polynomials corresponding to $B$ and $\gamma_t$ the corresponding loop in the space of critical values $V_n$. We may deform $\gamma_t$ such that the critical values $v_j(0)$, $j=1,2,\ldots,n-1$ at $t=0$ satisfy $v_i\neq v_j$ for all $i\neq j$. This can be achieved by only varying the absolute value of the different $v_j(t)$s, so that the deformation does not change the Rampichini diagram. We also knot that the deformation lifts to a deformation of $g_t$. Both of these invariances imply that the braided open book has not been changed by this deformation.
Let $R$ be the Rampichini diagram (cf. Figure~\ref{fig:rampi_odd}a)) whose labels at $\varphi=2\pi-\varepsilon$ spell the word $B$. We can define a square diagram $R'$ (cf. the lower half of Figure~\ref{fig:rampi_odd}b)) by gluing a trivial square diagram with $n-1$ vertical curves as in Figure~\ref{fig:trivial_rampi}a) on the right edge of $R$ and continuing all curves on the right edge of $R$ towards the right edge of the new diagram $R'$, crossing below the curves of the trivial diagram. The new cactus at $t=0$ is given by $\tau_j(0)=(k(j)\to\ell(j))$, the original labels of $R$ for $j=1,2,\ldots,n-1$, and $\tau_{j+n-1}=(k(j)+n-1\ (\text{mod }2n-1)\to\ell(j)+n-1\ (\text{mod }2n-1))$ for $j=1,2,\ldots,n-1$. It is thus an odd cactus.

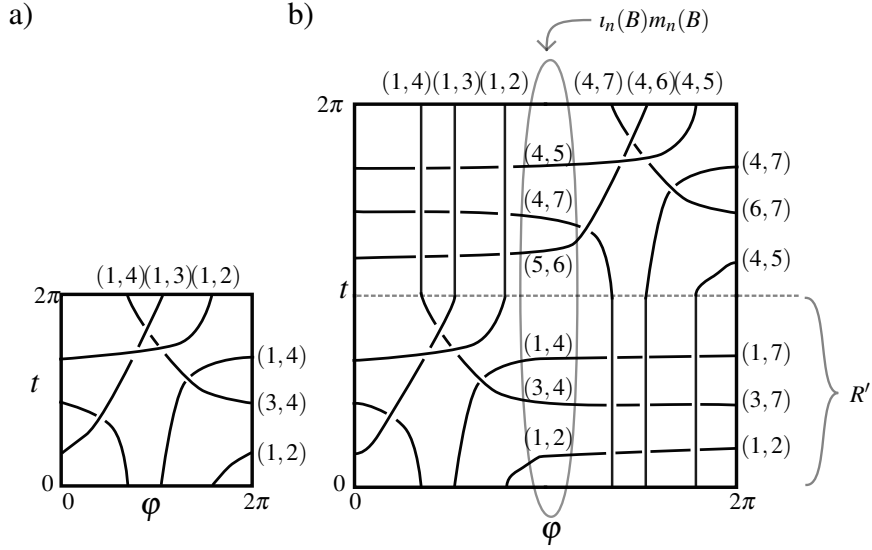
\begin{figure}[h]
\labellist
\Large
\pinlabel a) at  50 1780
\pinlabel b) at 1000 1780
\small
\pinlabel 0 at 140 190
\pinlabel 0 at 210 130
\pinlabel $2\pi$ at 850 130
\pinlabel $2\pi$ at 140 830
\pinlabel $(1,4)$ at 390 890
\pinlabel $(1,3)$ at 550 890
\pinlabel $(1,2)$ at 710 890
\pinlabel $(1,4)$ at 935 620
\pinlabel $(3,4)$ at 935 460
\pinlabel $(1,2)$ at 935 290
\pinlabel $(1,2)$ at 1840 340
\pinlabel $(3,4)$ at 1840 510
\pinlabel $(1,4)$ at 1840 660
\pinlabel $(1,4)$ at 1360 1550
\pinlabel $(1,3)$ at 1530 1550
\pinlabel $(1,2)$ at 1690 1550 
\pinlabel $(4,7)$ at 2010 1550
\pinlabel $(4,6)$ at 2190 1550
\pinlabel $(4,5)$ at 2350 1550
\pinlabel $(1,2)$ at 2580 315
\pinlabel $(3,7)$ at 2580 470
\pinlabel $(1,7)$ at 2580 630
\pinlabel $(5,6)$ at 1840 930
\pinlabel $(4,7)$ at 1840 1150
\pinlabel $(4,5)$ at 1840 1305
\pinlabel $(4,5)$ at 2580 950
\pinlabel $(6,7)$ at 2580 1120
\pinlabel $(4,7)$ at 2580 1280
\pinlabel 0 at 1130 190
\pinlabel 0 at 1185 130
\pinlabel $2\pi$ at 2500 130
\pinlabel $2\pi$ at 1100 1480
\pinlabel $\imath_n(B)m_n(B)$ at 2200 1760
\pinlabel $R'$ at 2900 500 
\Large
\pinlabel $t$ at 100 530
\pinlabel $\varphi$ at 490 100
\pinlabel $t$ at 1150 850
\pinlabel $\varphi$ at 1850 20
\endlabellist
\centering
\includegraphics[height=7.5cm]{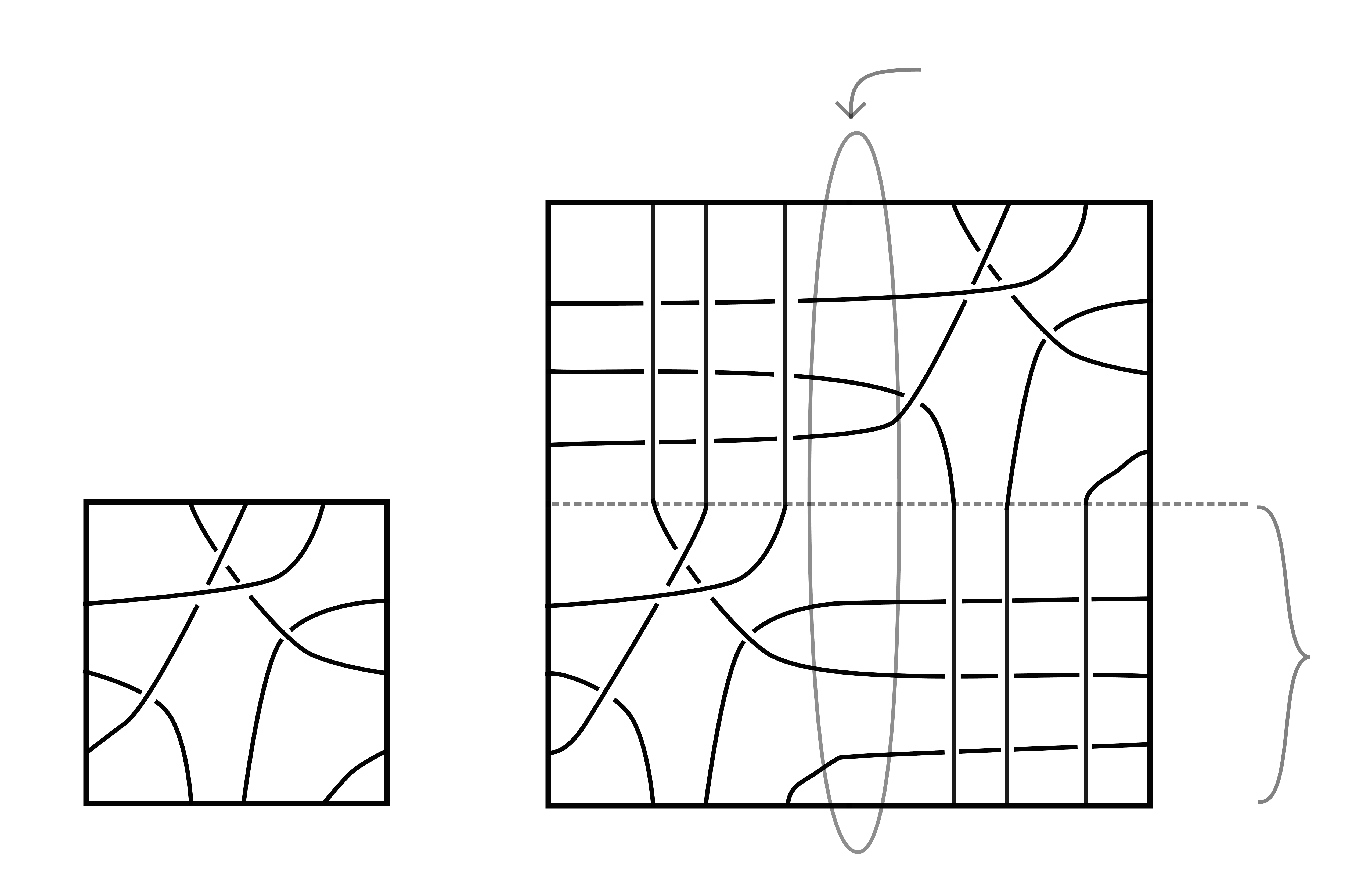}
\caption{a) A Rampichini diagram $R$. b) The corresponding Rampichini diagram $R''$ with odd symmetry. Its lower half is the Rampichini diagram $R'$. The labels at $\varphi=\pi$ (interpreted as band generators) read from the bottom to the top spell $\imath_n(B)m_n(B)$. \label{fig:rampi_odd}}
\end{figure}

The square diagram $R'$ corresponds to a loop in $V_n$ and we claim that this loop lifts to a loop in $X'_n$. By the proof of Theorem 5.6 in \cite{bode:braided} it is enough to check that the labels at the top edge of the square are the same as the labels at the bottom when we follow the rules of how labels change at crossings and at the right edge of the square. Since the vertical curves are the overcrossing strand in each crossing and since they never cross the right edge of the square, their labels do not change at all. 

We call the curves in $R$ and in $R'$ that are strictly monotone increasing \textit{positive curves} and the curves that are strictly monotone decreasing \textit{negative curves}. We have to show that shifting the indices of every label of a positive curve on the right edge of $R$ by $+1\text{ mod }n$ is equal to shifting the indices of the corresponding labels on the right edge of $R'$ by $+1\text{ mod }2n-1$. Likewise we have to show that the labels of the negative curves in $R'$ at $\varphi=\pi$ are equal to the corresponding labels on the right edge of $R$. This implies that the labels in the left half of $R'$ are equal to the corresponding labels in $R$.

If a label of a positive curve on the right edge of $R$ is $(i\to j)$, then the corresponding label on the right edge of $R'$ is given by the conjugate of $(i\to j)$ by $\underset{k=1}{\overset{n-1}{\prod}}\tau_{k+n-1}=(n\to 2n-1\to 2n-2\to \ldots\to n+2\to n+1)$. If both $i$ and $j$ are different from n, the result of this conjugation is equal to $(i\to j)$, since $i,j\notin\{n,n+1,n+2,\ldots,2n-1\}$. If one of them (say $i$) is equal to $n$, then we obtain $(2n-1\to j)$. This is precisely what we wanted, since 
\begin{equation}
(n+1\to j+1)\text{ mod }n=(1\to j+1)=(2n-1+1\to j+1)\text{ mod }2n-1.
\end{equation}

If a label of a negative curve on the left edge of $R$ is $(i\to j)$, then the corresponding label on the right edge of $R$ is $(i-1\to j-1)\text{ mod n}$, while it is $(i-1\to j-1)\text{ mod }2n-1$ on the right edge of $R'$. The label of the same curve in $R'$ at $\varphi=\pi$ after it has crossed all vertical curves is the conjugate of $(i-1\to j-1)\text{ mod }2n-1$ by $\underset{k=n-1}{\overset{1}{\prod}}\tau_{k+n-1}=(n\to n+1\to n+2\to \ldots\to 2n-1)$.  If both $i$ and $j$ are different from $1$, this is simply $(i-1\to j-1)\text{ mod }2n-1=(i-1\to j-1)\text{ mod }n$, which is the desired label. If one of them (say $j$) is equal to $1$, then we obtain $(n\to i-1)\text{ mod }2n-1=(n\to i-1)\text{ mod }n$. Thus in all cases the labels in the left half of $R'$ are the same as the corresponding labels in $R$. Since $R$ is a Rampichini diagram, its labels at $t=0$ are equal to its labels at $t=2\pi$ and thus $R'$ corresponds to a loop $g'_t$ in $X'_n$.

We may deform $R'$ and the corresponding loop of polynomials $g'_t$ in a neighbourhood of $t=0$ such that the basepoint $g'_0$ satisfies the property from Lemma~\ref{lem:trivialpath}: For every critical value $v'_j(0)$ of $g'_0$ there is a critical value $v'_k(0)$, $k\in\{1,2,\ldots,2n-2\}$, such that $\arg(v'_j(0))=\arg(v'_k(0))+\pi$.

Note that $R'$ was constructed in such a way that $g'_0$ has an odd cactus. By Lemma~\ref{lem:trivialpath} there is a path in $X'_{2n-1}$ from $g'_0$ to $\tilde{g'}_0$ (the monic polynomial satisfying $\tilde{g'}_0(u)=-g'_0(-u)$) with trivial square diagram. The square diagram that is the composition of $R'$ and this trivial square diagram can be interpreted as a square diagram of a path $\gamma_t$ from $g'_0$ to $\tilde{g'}_0$. 

For a path $\gamma_t$ in $X'_{2n-1}$ given by $\underset{j=1}{\overset{n}{\prod}}(u-z_j(t))$ we denote by $\tilde{\gamma}_t$ the path that is given by $\underset{j=1}{\overset{n}{\prod}}(u+z_j(t))$. If $\gamma_t$ is the path from $g'_0$ to $\tilde{g'}_0$ described above, we can associate to $\tilde{\gamma}_t$ a square diagram (cf. the top half of Figure~\ref{fig:rampi_odd}b)). Its labels at the bottom edge match its labels at the top edge by the same arguments as above.

We now compose $\gamma_t$ and $\tilde{\gamma}_t$, which is by construction an odd loop in $X'_{2n-1}$. Its corresponding square diagram $R''$ (cf. Figure~\ref{fig:rampi_odd}b)) is thus odd. It contains curves with vertical segments, but after a small isotopy that maintains oddness all curves are strictly monotone increasing or decreasing. Therefore, it is a Rampichini diagram.

We claim that the labels at $\varphi=\pi$ in this Rampichini diagram spell the word $\imath_n(B)m_{n}(B)$. We already know that the lower half of the diagram spells $\imath_n(B)$, since the lower left quadrant has the same labels as the original diagram $R$. By the same arguments, the upper right quadrant is the original diagram $R$ except that all labels on the right edge have been shifted by $n-1 \text{ mod }2n-1$. It follows that the labels in the upper half of $R''$ at $\varphi=\pi$ are the same as the labels on the left edge of $R$ shifted by $n \text{ mod }2n-1$, except if the corresponding label in $R$ was of the form $a_{i,n}$. In this case the label on the right edge in the upper right quadrant is $a_{i+n-1,2n-1}$. On the left edge of $R''$ it becomes $a_{1,i+n}$ and after conjugation by $\underset{j=1}{\overset{n-1}{\prod}}\tau_j=(1\to n\to n-1\to\ldots\to3\to2)$ we have $a_{n,i+n}$. Hence the labels in the top half at $\varphi=\pi$ spell $m_{n}(B)$.
\end{proof}

%Say a label of a positive curve on the right edge of $R$ is $(i\ j)$. Then the corresponding label on the right in the upper right quadrant of $R''$ is $(i+n\ j+n)\text{ mod }2n-1$ and we obtain the corresponding label at $\varphi=\pi$ by shifting it by $-1\text{ mod }2n-1$ and conjugating it by $\prod_{j=1}^{n-1}\tau_j=(1 \ 2\ 3\ \ldots\ n)$.

Note that if a knot $K$ is a closure of a P-fibered braid $B$ on $n$ strands, then the closure of $\imath_n(B)m_n(B)$ is the connected sum $K\# K$. Following the construction above we obtain an odd loop of polynomials for the P-fibered braid $\imath_n(B)m_n(B)$ on $2n-1$ strands. By Theorem~\ref{thm:weighted} we get a semiholomorphic polynomial with an isolated singularity at the origin and $K\# K$ as the link of that singularity. This proves the P-fibered/semiholomorphic version of a result by Looijenga \cite{looijenga}.

\begin{theorem}
Let $K$ be a knot that is the closure of a P-fibered braid on $n$ strands. Then $K\# K$ is real algebraic and the corresponding real polynomial can be taken to be semiholomorphic and of degree $2n-1$ with respect to the complex variable $u$.
\end{theorem}

\section{Trigonometric interpolation and approximation}\label{sec:trig}
%In order to construct polynomials with isolated singularities we start with an odd, pure Rampichini diagram $R$. We want to show that inserting inner loops into $R$ results in a real algebraic link.
The goal of this section is to prove that any singular Rampichini diagram that is obtained from a pure Rampichini diagram by inserting circles denoting singularities can be realised by a loop of polynomials $h_t$, whose coefficients are polynomials in $\rme^{\rmi t}$ and $\rme^{-\rmi t}$.

%we can construct an odd loop of polynomials $h_t$, whose roots form a certain singular braid $B_{sing}$. The classical crossings of $B_{sing}$ correspond to a braid word associated with $R$, while its singular crossings occur at values of $t\in[0,2\pi]$ at which we want to insert inner loops and at the corresponding values $t+\pi$. Furthermore, $h_t$ should have coefficients that are polynomials in $\rme^{\rmi t}$ and $\rme^{-\rmi t}$.

The following result follows immediately from a theorem by Deutsch \cite{deutsch}, which is a generalization of work by Walsh \cite{walsh}.
%\begin{theorem}[Deutsch]\label{thm:deutsch}
%Let $Y$ be 
%\end{theorem}
\begin{lemma}\label{lem:appint}
Let $f:\mathbb{R}\to\mathbb{R}$ be a $2\pi$-periodic $C^k$-function. Then for any chosen set of points $z_j$, $j=1,2,\ldots,M$, and any $\varepsilon>0$, there is a trigonometric polynomial $f_{trig}$ with $f_{trig}^{(i)}(z_j)=f^{(i)}(z_j)$ for all $j=1,2,\ldots,M$, $i=1,2,\ldots,k$, and $|f-f_{trig}|_k<\varepsilon$.
\end{lemma}
\begin{proof}
The proof is completely analogous to that of Corollary 3.2 in \cite{deutsch}. For every $j=1,2,\ldots,M$ and $i=1,2,\ldots,k$ we can define the operator $L_{i,j}(g)=g^{(i)}(z_j)$ on the space of $2\pi$-periodic $C^k$-functions, which evaluates the $i$th derivative of a function at the point $z_j$. Then every $L_{i,j}$ is a continuous linear operator, since
\begin{equation}
|L_{i,j}(g)|=|g^{(i)}(z_j)|\leq \underset{x\in\mathbb{R}}{\sup}|g^{(i)}(x)|\leq \sum_{i=1}^k\underset{x\in\mathbb{R}}{\sup}|g^{(i)}(x)|=|g|_k.
\end{equation} 
The lemma then follows from Theorem 0 in \cite{deutsch}, since trigonometric polynomials are dense in the space of $2\pi$-periodic functions with the $C^k$-norm.
\end{proof}

Applying this lemma to the real and imaginary part of a $2\pi$-periodic $C^k$-function $f:\mathbb{R}\to\mathbb{C}$ results in the analogous result for finite Fourier series.

For a power series $f=\underset{j=0}{\overset{\infty}{\sum}}a_j t^{j}$ we call the smallest index $j$ with $a_j\neq 0$ the \textit{lowest order} of $f$.

\begin{lemma}\label{lem:cons}
Let $B$ be a P-fibered geometric braid on $n$ strands with a pure Rampichini diagram $R$. Let $v_j(t)$, $j=1,2,\ldots,n-1$, denote the critical values of the corresponding loop of polynomials. Let $\{(t_i,j(i))\}_{i=1,2,\ldots,M}$, $t_i\in[0,2\pi]$, $t_i\neq t_j$ if $i\neq j$, $j(i)\in\{1,2,\ldots,n-1\}$, be such that for every $i\in\{1,2,\ldots,M\}$ we have $\arg(v_{j(i)}(t_i))\neq\arg(v_k(t_i))$ for all $k\neq j(i)$. Then there exists a loop $h_t$ in the space of monic polynomials with fixed degree $n$ such that all of its coefficients are finite Fourier series and $h_t$ has an associated singular Rampichini diagram given by $R_{\{(t_i,j(i))\}_{i=1,2,\ldots,M}}$.

\end{lemma}
\begin{proof}
Since $B$ is a P-fibered geometric braid, we know that the critical values $v_j(t)$ of the corresponding loop of polynomials satisfy $\tfrac{\partial \arg(v_j(t))}{\partial t}\neq 0$ for all $t\in[0,2\pi]$. This property does not change under small $C^1$-deformations of $v_j(t)$, i.e., there is an $\varepsilon_0$ such that for all positive $\varepsilon<\varepsilon_0$ and any smooth function $F:[0,2\pi]\to\mathbb{C}$ with $|F|_1<\varepsilon$ we have $v_j(t)+F(f)\neq 0$ and $\tfrac{\partial\arg(v_j(t)+F(t))}{\partial t}\neq 0$ for all $t$. This can be seen from $\arg(v_j(t)+F(t)))=\text{Im}(\text{Log}(v_j(t)+F(t))))$ and thus
\begin{equation}\label{eq:darg}
\frac{\partial \arg(v_j+F))}{\partial t}=\frac{\text{Re}(v_j+F)\tfrac{\partial\text{Im}(v_j+F)}{\partial t}-\text{Im}(v_j+F)\tfrac{\partial\text{Re}(v_j+F)}{\partial t}}{(\text{Re}(v_j+F))^2+(\text{Im}(v_j+F))^2}.
\end{equation}

Let $\varepsilon<\varepsilon_0$ be such a small positive number. In particular, the above implies that $|v_j(t)|>\varepsilon$ for all $j$ and $t$. Pick a small number $\delta>0$ and define $V_i$ as the open interval $(t_i-\delta,t_i+\delta)$. In particular, $\delta$ should be chosen such that in every $V_i$ the argument of $v_{j(i)}(t)$ differs from the argument of any other critical value. Let $f_j:[0,2\pi]\to\mathbb{R}$ be the smooth function defined to be $-\rme^{-\left(\tfrac{t-t_i}{(t-t_i)^2-\delta^2}\right)^2}$ on all $V_i$ with $j(i)=j$ and everywhere else constant 0. 

We may deform the critical values $v_j(t)$ without changing $\arg v_j(t)$ such that $|v_{j(i)}(t)|$ becomes small in $V_i$ for all $i$. More precisely, we want that $|v_{j(i)}(t_i)|$ is a global minimum of $|v_{j(i)}(t)|$ (and in particular $|v_{j(i)}(t_i)|=|v_{j(i')}(t_{i'})|$ for all $i,i'$ with $j(i)=j(i')$) and $|f_jv_{j(i)}(t)|_1<\varepsilon/n$ for all $t$ in any $V_i$ with $j(i)=j$. Note that $\varepsilon$ is the same as above and does not depend on $i$. This deformation of the critical values can be performed without changing $v_j(t)$ in any $V_i$ with $j(i)\neq j$.

%Outside neighbourhoods of the $t_i$s we may leave $|v_{j(i)}(t)|$ unchanged. 
This deformation lifts to a homotopy of the loop of polynomials with distinct roots and fixed degree \cite{bode:braided}. We can now approximate the critical points $c_j(t)$, $j=1,2,\ldots,n-1$, of the resulting polynomials $g_t$ (the endpoint of the homotopy) and its constant term $a_0(t)$ by finite Fourier series using the $C^2$-norm, while interpolating the values of $c_j(t)$, $a_0(t)$ and their first and second derivatives at $t=t_i$, $i=1,2,\ldots,M$. (It is at this point that we use the fact that $R$ is pure. Since $v_j(2\pi)=v_j(0)$, we have $c_j(2\pi)=c_j(0)$, which is necessary for a finite Fourier series.) This approximation and simultaneous interpolation exists by Lemma~\ref{lem:appint}. By choosing the approximation close enough to the original $c_j(t)$ and $a_0(t)$ we can guarantee that there are still no argument-critical points of $g_t$. Since all critical values are smooth (in fact, analytic) functions of the critical points and the constant term, we can arrange that the critical valuues $w_j(t)$, $j=1,2,\ldots,n-1$, of $g_t$ satisfy $|w_j(t)-v_j(t)|_1<\varepsilon/n$ on all $V_i$ with $j(i)\neq j$. Furthermore, since all critical points and the constant term are parametrised by finite Fourier series, every coefficient of $g_t$ is a finite Fourier series. It follows that the critical values $w_j(t)$, $j=1,2,\ldots,n-1$, of $g_t$ are also parametrised by finite Fourier series and are non-zero. Moreover, the interpolation guarantees that we still have $|w_{j(i)}(t_i)|$ as a global minimum of $|w_{j(i)}(t)|$ and choosing the approximation sufficiently close gives $|f_jw_{j}(t)|_1<\varepsilon/n$. Note that this inequality is automatically satisfied outside of the $V_i$s with $j(i)=j$, where $f_j\equiv 0$.

%The analogue of Eq.~\eqref{eq:darg} for $w_j+F$ implies that $w_j+F$ is non-zero and does not have any argument-critical points for any $F:[0,2\pi]\to\mathbb{C}$ with $|F|_1<\varepsilon'$ for any sufficiently small $\varepsilon'$. Let $\varepsilon'$ such a small positive number and we assume that $\varepsilon'<\varepsilon$.

We claim that we can obtain the desired loop of polynomials $h_t$ by only changing the constant term of $g_t$. More specifically, we are going to add $n-1$ finite Fourier series $A_1(t)$, $A_2(t),\ldots,A_{n-1}(t)$ to its constant term. These functions are defined successively as follows. We perform a simultaneous trigonometric interpolation and approximation of the function $f_j$ and its first $N$ derivatives, where $N$ is some sufficiently large natural number. The resulting interpolating trigonometric polynomial $F_j$ should agree with the values of $f_j$ and its first $N$ derivatives on all $t=t_i$. Then we define $A_j(t)=F_j(t)w_{j,j-1}(t)$, where $w_{k,j}(t)=w_k(t)+\underset{i=1}{\overset{j}{\sum}}A_i(t)$. Note that adding a term $A_j(t)$ to the constant term of a loop of polynomials shifts all critical values by $A_j(t)$, so that $w_{k,j}(t)$, $k=1,2,\ldots,n-1$, are the critical values of $g_t+\underset{i=1}{\overset{j}{\sum}}A_i(t)$. We prove the following claim by induction on $j$.

\textit{Claim:} If the approximation of $f_j$ by $F_j$ is sufficiently close for all $j=1,2,\ldots,n-1$, then the following properties hold for all $j=1,2,\ldots,n-1$:
\begin{itemize}
\item $w_{k,j}(t)=0$ if and only if $k\leq j$ and $t=t_i$ for some $i$ with $j(i)=k$.
\item $\arg(w_{k,j}(t))$ and $\tfrac{\partial\arg(w_{k,j}(t))}{\partial t}$ have well-defined limits as $t$ goes to $t_i$ for all $i$ with $j(i)=k$, with the latter limit being non-zero.
\item The critical values $w_{k,j}(t)$, $k=1,2,\ldots,n-1$, have no argument-critical points, wherever $w_{k,j}(t)\neq 0$.
\item For all $k$ we have $|w_{k,j}-v_k|_1<(j+1)\varepsilon/n$ on all $V_i$ with $j(i)\neq k$.
\item For all $k>j$ we have $|f_kw_{k,j}|_1<\varepsilon/n$ on all $V_i$.
\item $|w_{k,j}(t_i)|$ is a global minimum of $|w_{k,j}(t)|$ for all $i$ with $j(i)=k$.
\end{itemize}

Before we go into the proof of the claim, note that each $A_j$ is indeed a finite Fourier series, since $F_j$ is a trigonometric polynomial and $w_j$ is a finite Fourier series for all $j=1,2,\ldots,n-1$.

\textbf{Base case: $j=1$:}\\
First we focus on $w_{1,1}$. The function $A_1(t)$ was constructed such that $w_{1,1}(t)=0$ if and only if $t=t_i$ and $j(i)=1$. This can be seen as follows. There is a neighbourhood of $t=t_i$ with $j(i)=1$ such that in that neighbourhood any sufficiently close approximation of $f_1$ in the $C^2$-norm has positive second derivative. It follows that $t=t_i$ is the unique point where $F_1$ takes the value $-1$ in that neighbourhood and thus $t=t_i$ is the only zero of $w_{1,1}$ in that neighbourhood. We can guarantee that there are no zeros of $w_{1,1}$ outside of these neighbourhoods, since their complement in $[0,2\pi]$ is compact and we may choose the approximation arbitrarily close.

Furthermore, $\tfrac{\partial\arg(w_{1,1}(t))}{\partial t}(t)\neq 0$ whenever $w_{1,1}\neq 0$, since in this case $\arg w_{1,1}(t)=\arg w_1(t)$. This equality also implies that $\arg(w_{1,1}(t))$ and $\tfrac{\partial\arg(w_{1,1}(t))}{\partial t}$ have well-defined limits as $t$ goes to $t_i$ with $j(i)=1$, which are equal to $\arg(w_1(t_i))$ $\tfrac{\partial\arg(w_1)}{\partial t}(t_i)\neq 0$, respectively.

We know that $w_1(t)$ differs from $v_1(t)$ on $V_i$ with $j(i)\neq 1$ by at most $\varepsilon/n$ in the $C^1$-norm. Since $f_1\equiv 0$ on every such $V_i$, it follows that if $F_1$ is a sufficiently close approximation of $f_1$, it satisfies $|F_1(t)w_1(t)|_1<\varepsilon/n$ on all such $V_i$. We thus obtain
\begin{equation}
|w_{1,1}(t)-v_1(t)|_1\leq|F_1(t)w_1(t)|_1+|w_1(t)-v_1(t)|_1<2\varepsilon/n
\end{equation} 
on all $V_i$ with $j(i)\neq 1$.

Now we turn our attention to $w_{k,1}$ with $k>1$. We want to show that for any $k\neq 1$ the new critical values $w_{k,1}(t)$ are non-zero and satisfy $\tfrac{\partial w_{k,1}}{\partial t}(t)\neq 0$. We know that $|f_1w_1|_1<\varepsilon/n$ on all $V_i$ (on those with $j(i)=1$ by construction, on the others because $f_1\equiv 0$). It follows that if $F_1$ is a sufficiently close approximation of $f_1$, it satisfies $|F_1(t)w_1(t)|_1<\varepsilon/n$ on all $V_i$. Recall that in the $V_i$s with $j(i)\neq k$, the critical value $w_k(t)$ with $k\neq 1$ differs from the original $v_k(t)$ by at most $\varepsilon/n$ in the $C^1$-norm, so that 
\begin{equation}
|w_{k,1}(t)-v_k(t)|_1\leq|F_1(t)w_1(t)|_1+|w_k(t)-v_k(t)|_1<2\varepsilon/n.
\end{equation} 
Thus $w_{k,1}(t)$ is non-zero and has no argument-critical points in $V_i$ for all $i$ with $j(i)\neq k$.

Outside of the $V_i$s with $j(i)\neq k$ the same is true if $|F_1w_1|_1$ is sufficiently small on $[0,2\pi]\backslash(\underset{\underset{j(i)\neq k}{i\text{ s.t.}}}{\cup}V_i)$, which can be arranged by choosing $F_1$ sufficiently close to $f_1$ on the compact set $[0,2\pi]\backslash(\underset{\underset{j(i)\neq k}{i\text{ s.t.}}}{\cup}V_i)$. 

Obviously, $\arg(w_{k,1}(t))$ and $\tfrac{\partial\arg(w_{k,1}(t))}{\partial t}$ have well-defined limits as $t$ approaches $t_i$ with $j(i)=k>1$, since $w_{k,1}$ with $k>1$ has no zeros. Thus its argument and its derivative are well-defined at $t=t_i$. Since there are no argument-critical points, this value of the derivative is non-zero. We have thus proved the first four items of the claim for $j=1$.

Since for all $i$ with $j(i)\neq 1$ we have $f_1\equiv 0$ on all $V_i$, any sufficiently close approximation $F_1$ satisfies $|f_k w_{k,1}|_1=|f_k (w_k+F_1 w_1)|_1<\varepsilon/n$ on $V_i$, $k>1$. Here we have used that $|f_kw_k|_1<\varepsilon/n$ on $V_i$ with $j(i)=k$ and $f_k\equiv 0$ everywhere else. 

%Then any sufficiently close approximation $F_k$ of $f_k$ satisfies the fifth property $|F_k w_{k,1}|_1= |F_k (w_k+F_1w_1)|_1<\varepsilon/n$ on all $V_i$.

The sixth item is obviously true for $k=1$, since $w_{1,1}(t)=0$ if and only if $t=t_i$ by the first item. For $k>1$ note that by the interpolation property of $F_1$ we have 
\begin{equation}
F_1(t_i)=F_1^{(1)}(t_i)=F_1^{(2)}(t_i)=\ldots=F_1^{(N)}(t_i)=0
\end{equation} 
for all $i$ with $j(i)>1$. Since $w_k(t)$ is given by a non-vanishing finite Fourier series, it is real-analytic and $|w_k(t)|$ is real-analytic, too. We want the natural number $N$, which determines to which order $F_1$ interpolates the derivatives of $f_1$ at $t=t_i$, $j(i)>1$, to be larger than the lowest order of the power series of $|w_k(t)|-|w_k(t_i)|$ centered at $t=t_i$. (We are omitting the fact that this lowest order could depend on $i$ and $k$. We simply choose $N$ such that it is sufficient for all $i$ and $k$.) Since $|w_k(t_i)|$ is a local minimum, this lowest order has a positive coefficient and by the interpolation property $|w_{k,1}(t)|-|w_{k,1}(t_i)|$ has the same lowest order term. It follows that $|w_{k,1}(t_i)|$ is a local minimum and the only local minimum in a neighbourhood. On the compact complement of these neighbourhoods in $[0,2\pi]$ we can guarantee that $|w_{k,1}(t)|>|w_{k,1}(t_i)|$ by choosing a sufficiently close approximation $F_1$ of $f_1$. On $[0,2\pi]\backslash\left(\underset{\underset{j(i)=1}{i\text{ s.t. }}}{\cup}V_i\right)$, where $f_1\equiv 0$, this is obvious. On $V_i$ with $j(i)=1$ we may use that $|w_k(t)|=|v_k(t)|>\varepsilon$, so that 
\begin{align}
|w_{k,1}(t)|&=|w_k(t)+A_1(t)|\geq ||w_k(t)|-|A_{1}(t)||\nonumber\\
&>\varepsilon-\varepsilon/n>\varepsilon/n>|f_kw_k|_1\geq|f_k(t_i)w_k(t_i)|\nonumber\\
&= |w_{k,1}(t_i)|.
\end{align}
Here we have used that $|A_1(t)|=|F_1(t)w_1(t)|<\varepsilon/n$. Thus $|w_{k,1}(t_i)|$, $j(i)=k$, are global minima of $|w_{k,1}(t)|$ for all $k$. This concludes the proof of the six items in the claim for the base case of $j=1$.

\textbf{Induction:}
We assume that the claim holds for some $j-1\in\{1,2,\ldots,n-2\}$ and want to show that it then also holds for $j$.

By the same arguments as in the base case we have $w_{j,j}(t)=0$ if and only if $t=t_i$ for $i$ with $j(i)=j$. Away from these points we have $\arg(w_{j,j}(t))=\arg(w_{j,j-1}(t))$, so that there are no argument-critical points of $w_{j,j}$ and $\arg(w_{j,j}(t))$ and $\tfrac{\partial\arg(w_{j,j}(t))}{\partial t}$ have well-defined limits, equal to $\arg(w_{j,j-1}(t_i))$ and $\tfrac{\partial\arg(w_{j,j-1})}{\partial t}(t_i)$, as $t$ goes to $t_i$ with $j(i)=j$.
%Furthermore, the interpolation property of $F_1$ at $t=t_i$ for $i$ with $j(i)\neq 1$ guarantees that $|w_{j(i),1}(t_i)|=\varepsilon$ for all $i$ with $j(i)\neq 1$ and it is a global minimum of $|w_{j(i),1}(t)|$.

For $w_{k,j}$ consider that by the induction hypothesis $|w_{k,j-1}-v_k|_1<j\varepsilon/n$ on all $V_i$ with $j(i)\neq k$. The induction hypothesis also states that $|f_j w_{j,j-1}|_1<\varepsilon/n$ on all $V_i$. This means that on all $V_i$ with $j(i)\neq k$ we have
\begin{equation}
|w_{k,j}-v_k|_1\leq|F_jw_{j,j-1}|_1+|w_{k,j-1}-v_k|_1<(j+1)\varepsilon/n
\end{equation}
if $F_j$ is sufficiently close to $f_j$. This implies that $w_{k,j}(t)$ is non-zero and has no argument-critical points in $V_i$, $j(i)\neq k$. Since for $k>j$ the critical value $w_{k,j-1}(t)$ is non-zero and has no argument-critical points, the same is true for $w_{k,j}(t)$ on the compact set $[0,2\pi]\backslash(\underset{\underset{j(i)\neq k}{i \text{ s.t.}}}{\bigcup}V_i)$ if the approximation of $f_j$ by $F_j$ is sufficiently close. Since $w_{k,j}(t)$, $k>j$, is non-vanishing, its argument and the derivative of its argument are well-defined everywhere and naturally have a well-defined, non-zero limit at $t=t_i$. This proves the first four items of the claim for $w_{k,j}$ with $k\geq j$ and the fourth item for $w_{k,j}$, $k<j$.

For the fifth item consider that on $V_i$ with $j(i)\neq k$ the function $f_k$ is constant zero, while on $V_i$ with $j(i)=k>j$ the function $f_j$ is constant zero. In either case we have that 
\begin{equation}
|f_k w_{k,j}|_1=|f_k (w_{k,j-1}+f_jw_{j,j-1})|_1<\varepsilon/n
\end{equation} 
on $V_i$ by the induction hypothesis.

The sixth item in the claim is trivial for all $k$ with $k\leq j$, since in these cases we already know that $t=t_i$ with $j(i)=k$ are the only zeros of $w_{k,j}(t)$. For $k>j$ it is the same argument as in the base case ($j=1$). The fact that $F_j$ interpolates $f_j$ and its first $N$ derivatives guarantees that $t=t_i$ is a local minimum and by choosing the approximation sufficiently close, we can show that these are global minima of $|w_{k,j}(t)|$.

What is left to show are the first three items of the claim for $w_{k,j}(t)$ with $k<j$. This is particularly subtle in neighbourhoods of $t=t_i$ with $j(i)=k$, where $w_{k,j}$ has a zero.

%We want to proceed similarly for the definition of the other finite Fourier series $A_j(t)$. We would like to find some trigonometric polynomial $F_j(t)$ via interpolation and simultaneous approximation and define $A_j(t)$ as the product of $F_j(t)$ and $w_{j,j-1}(t)$. This will shift the critical values in the complex plane so that $w_{j,j}(t_i)=0$ if and only if $j(i)=j$. We then have that $\arg w_{j,j}=\arg w_{j,j-1}$ and in particular, $\tfrac{\partial \arg w_{j,j}}{\partial t}=\tfrac{\partial \arg w_{j,j-1}}{\partial t}\neq 0$ whenever it is defined. Furthermore, the same calculation as above (for $j=1$) guarantees that there are no argument-critical points of $w_{k,j}$ with $k>j$. However, we need to be a bit more careful in order to avoid introducing any argument-critical points of $w_{k,j}$ with $k< j$, in particular in a neighbourhood of $t_i$ with $j(i)=k<j$, where there is a zero of $w_{k,j}$.

Since $w_{k,j-1}$, $k<j$, is a finite Fourier series, both the numerator and the denominator of $\tfrac{\partial \arg w_{k,j-1}}{\partial t}$, given by 
\begin{equation}
NUM=\text{Re}(w_{k,j-1}(t))\tfrac{\partial \text{Im}(w_{k,j-1})}{\partial t}(t)-\text{Im}(w_{k,j-1}(t))\tfrac{\partial \text{Re}(w_{k,j-1})}{\partial t}(t)
\end{equation} 
and 
\begin{equation}
DEN=(\text{Re}(w_{k,j-1}(t)))^2+(\text{Im}(w_{k,j-1}(t)))^2,
\end{equation}
are real analytic. As both of these go to zero as $t$ approaches $t_i$ with $j(i)=1$, the derivative $\tfrac{\partial \arg w_{k,j-1}}{\partial t}$ is not defined at these points. However, by assumption it has a well-defined limit as $t$ approaches $t_i$ and by the interpolation property of each $F_j'$, $j'<j$, it is equal to $\tfrac{\partial w_{k}}{\partial t}(t_i)\neq 0$.

Since both $NUM$ and $DEN$ are real-analytic, we may consider their power series at $t=t_i$. Let $\underset{m}{\sum} b_m (t-t_i)^m$ be the power series of $DEN$ and let $m(i)$ denote the lowest index with $b_{m(i)}\neq 0$. (Since the limit of their quotient is well-defined and non-zero, using the power series of $DEN$ to define $m(i)$ would have the same result). From the interpolation properties of $F_i$, $i=1,2,\ldots,j$, we know that the lowest order term of $F_k$ is quadratic and the lowest order of $F_i$, $i\neq k$, is $N+1>2$. Since $w_{k,k-1}(t)$ is non-zero for all $t$, we have that $m(i)=4$. We know thus that for any $\delta'>0$ there are neighbourhoods $W_i\subset V_i$ of $t_i$ on which $|NUM|$ and $|DEN|$ are both greater than $\delta'|(t-t_i)^m|$, where $m>m(i)=4$.

Since we chose $F_j$ such that $|F_j|_N<\varepsilon'$ outside of $V_i$ with $j(i)=j$ for some arbitrarily small $\varepsilon'$, we have $|F_j(t)|<|(t-t_i)^N|\varepsilon'$ and $|\tfrac{\partial F_j}{\partial t}(t)|<|(t-t_i)^{N-1}|\varepsilon'$ for all $t\in V_i$, $j(i)\neq j$. It follows that in $W_i\subset V_i$, $j(i)\neq j$, we have $|w_{k,j}(t)|\geq \delta'|(t-t_i)^m|-|(t-t_i)^N|\varepsilon''\text{max}_{t\in W_i}|w_{j,j-1}(t)|$, $N> m>m(i)=4$, which has no zero in $W_i$ apart from $t=t_i$ if $\varepsilon''$ is sufficiently small. Therefore, the $t_i$s with $j(i)=k$ are the only zeros of $w_{k,j}$ in $W_i$, $j(i)\neq j$. Since these were also the only zeros of $w_{k,j-1}$ anywhere, a sufficiently close approximation $F_j$ of $f_j$ guarantees that there are no zeros of $w_{k,j}$ outside of the $W_i$s with $j(i)\neq j$ either. As in the previous cases this follows from $f_j\equiv 0$ outside of $V_i$, $j(i)=j$, and $|w_{k,j}|\geq|v_k|-|w_{k,j}-v_k|>(n-j-1)\varepsilon/n$ on $V_i$, $j(i)=j$.

Similarly, the modulus of the numerator $NUM_{k,j}$ of $\tfrac{\partial \arg w_{k,,j}}{\partial t}$ satisfies
\begin{align}
|NUM_{k,j}|=&\left|\text{Re}(w_{k,j-1}+A_j)\frac{\partial\text{Im}(w_{k,j-1}+A_j)}{\partial t}-\text{Im}(w_{k,j-1}+A_j)\frac{\partial\text{Re}(w_{k,j-1}+A_j)}{\partial t}\right|\nonumber\\
%\end{equation}
%\begin{align}
%\left|\frac{\partial \arg w_{1,2}}{\partial t}\right|
%|NUM_{1,2}|=&\left|\text{Re}(w_{1,1}+A_2)\frac{\partial\text{Im}(w_{1,1}+A_2)}{\partial t}-\text{Im}(w_{1,1}+A_2)\frac{\partial\text{Re}(w_{1,1}+A_2)}{\partial t}\right|\nonumber\\
=&\left|NUM_{k,j-1}+\text{Re}(w_{k,j-1})\frac{\partial\text{Im}(A_j)}{\partial t}-\text{Im}(w_{k,j-1})\frac{\partial\text{Re}(A_j)}{\partial t}\right.\nonumber\\
&\left. +\text{Re}(A_j)\frac{\partial\text{Im}(w_{k,j-1})}{\partial t}-\text{Im}(A_j)\frac{\partial\text{Re}(w_{k,j-1})}{\partial t}\right.\nonumber\\
&\left. +\text{Re}(A_j)\frac{\partial\text{Im}(A_{j})}{\partial t}-\text{Im}(A_j)\frac{\partial\text{Re}(A_{j})}{\partial t}\right|\nonumber\\
>&\delta|(t-t_i)^m|-2|(t-t_i)^N|\varepsilon'\text{max}_{t\in W_i}\left|\tfrac{\partial w_{k,j-1}}{\partial t}(t)\right|\text{max}_{t\in W_i}\left|w_{j,j-1}(t)\right|\nonumber\\
&-2|(t-t_i)^{N-1}|\varepsilon'\text{max}_{t\in W_i}|w_{k,j-1}(t)|\text{max}_{t\in W_i}\left|w_{j,j-1}(t)\right|\nonumber\\
&-2|(t-t_i)^{N}|\varepsilon'\text{max}_{t\in W_i}|w_{k,j-1}(t)|\text{max}_{t\in W_i}\left|\tfrac{\partial w_{j,j-1}}{\partial t}(t)\right|\nonumber\\
&-2|(t-t_i)^{2N-1}|\varepsilon'^2\text{max}_{t\in W_i}\left|w_{j,j-1}(t)\right|^2\nonumber\\
&-4|(t-t_i)^{2N}|\varepsilon'\left|\tfrac{\partial w_{j,j-1}}{\partial t}(t)\right|\text{max}_{t\in W_i}\left|w_{j,j-1}(t)\right|\nonumber\\
>&0,
\end{align}
in $W_i$, $j(i)\neq j$, where we use $A_j=F_j w_{j,j-1}$ and $|\text{Re}(z)|\leq |z|$ for all $z\in\mathbb{C}$ (and likewise for $\text{Im}(z)$). The last inequality holds because $\varepsilon'$ can be chosen arbitrarily small and $N>m$.

Thus there are no argument-critical points of $w_{k,j}$ in $W_i$, $j(i)\neq j$. We already showed above that there are no argument-critical points of $w_{k,j}$ in $V_i$, $j(i)\neq k$. Since $w_{k,j-1}$ is non-zero and has no argument-critical points in the complement of these $W_i$ and $V_i$, a sufficiently close approximation $F_j$ of $f_j$ guarantees that there are no argument-critical points of $w_{k,j}$ in the complement either.

Note that the interpolation property of $F_j$ guarantees that $\arg(w_{k,j}(t))$ and $\tfrac{\partial\arg(w_{k,j}(t))}{\partial t}$ have well-defined limits as $t$ goes to $t_i$, $j(i)=k$, which are equal to $\arg(w_{k,j-1}(t_i))$ and $\tfrac{\partial \arg(w_{k,j-1})}{\partial t}(t_i)\neq 0$, respectively. This concludes the proof of the claim by induction.

The desired loop of polynomials $h_t$ is given by $g_t+\underset{j=1}{\overset{n-1}{\sum}}A_{j}(t)$ with critical values $w_{j,n-1}(t)$, $j=1,2,\ldots,n-1$. The first three items of the claim imply that the square diagram associated to $h_t$ is a singular Rampichini diagram and by construction it is equal to $R_{\{(t_i,j(i))\}_{i=1,2,\ldots,M}}$.
%The Rampichini diagram of $h_t$ differs from that of $g_t$ only by a $C^1$-approximation of the lines and by the horizontal lines inserted at the heights $t=t_i$. (Note that this statement uses the fact that the $t_i$s are distinct from the critical heights $T_k$.) In particular, the labels of the two Rampichini diagrams are identical. A vertical line $\varphi=2\pi-\delta$ for some small enough $\delta$ intersects the lines in the Rampichini diagram of $h_t$ away from the heights $t=t_i$, so that the labeling transpositions spell the same band word as in the Rampichini diagram of $g_t$, namely that of $F$.
\end{proof}

\begin{lemma}\label{lem:oddversion}
Let $B$ be a P-fibered braid on $n=2n'-1$ strands with an odd, pure Rampichini diagram $R$. Suppose that the set $\{(t_i,j(i))\}_{i=1,2,\ldots,M}$, with $M$ even, from Lemma~\ref{lem:cons} satisfies the following symmetry. For every $i=1,2,\ldots,M/2$, we have $t_{i+M/2}=t_i+\pi\text{ mod }2\pi$ and $k(i+M/2)=k(i)+n'-1\text{ mod }2n'-2$, where $v_{j(i+M/2)}(t_{i+M/2})=v(t_{i+M/2})_{k(i+M/2)}$ and $v_{j(i)}(t_i)=v(t_i)_{k(i)}$. Then $h_t$ in Lemma~\ref{lem:cons} can be taken to be odd.
\end{lemma}
\begin{proof}
This lemma only requires small modifications to the proof of Lemma~\ref{lem:cons}. Since $R$ is odd, we may take $g_t$ to be odd. Since the set $\{(t_i,j(i))\}_{i=1,2,\ldots,M}$ has the symmetry above, the data set for the interpolation can be taken to have an odd symmetry. In other words, all functions $f_j$ that are approximated can be taken to satisfy $f_j(x+\pi)=-f_j(x)$. Since the set of odd trigonometric polynomials is dense in the space of odd periodic functions, Lemma~\ref{lem:appint} guarantees that the interpolating function $F_j$ can also be taken to be odd. Note that the critical values $w_j(t)$ of $g_t$ are odd functions, so that $A_j$ is also odd. We obtain $h_{t+\pi}(u)=g_{t+\pi}(u)+\underset{j=1}{\overset{n-1}{\sum}}A_j(t+\pi)=-g_t(-u)-\underset{j=1}{\overset{n-1}{\sum}}A_j(t)=-h_t(-u)$.
\end{proof}

\begin{lemma}\label{lem:ranal}
The functions $|w_{j,n-1}(t)|$ are real-analytic.
\end{lemma}
\begin{proof}
Since $w_{j,n-1}(t)$ is a finite Fourier series, its absolute value is real-analytic, except possibly at values of $t$, where $w_{j,n-1}(t)=0$. By construction 
\begin{equation}
w_{j,n-1}(t)=(1+F_j(t))w_{j,j-1}(t)+\sum_{k>j}A_k(t)
\end{equation} 
and each $A_k(t)$ has a lowest order of at least $5$, while the lowest order of the power series of $(1+F_j(t))^2|w_{j,j-1}(t)|^2$ is equal to $4$. Thus the root of $|w_{j,n-1}(t)|^2$ at $t=t_i$ is exactly of order 4. We can thus write $|w_{j,n-1}(t)|^2=(t-t_i)^4W(t)$, where $W$ is real-analytic with $W(t_i)> 0$. Then $|w_{j,n-1}(t)|$ is given by $(t-t_i)^2\sqrt{W(t)}$, which is a well-defined non-negative function, which squares to $|w_{j,n-1}(t)|^2$ and is real-analytic, since $W(t)$ is non-zero in a neighbourhood of $t_i$.
\end{proof}

\section{Isolated singularities}\label{sec:deform}

%Let $R$ be a simple Rampichini diagram corresponding to a P-fibered braid on $n$ strands. Let $v_j(t)$ be the critical values of the corresponding loop of polynomials $g_t$. Consider now the following loop in $V_n$:
%\begin{align}
%w_j(t)&=\nonumber\\
%w_{n-j}&=.
%\end{align} 

%This loop in $V_n$ has a simple Rampichini diagram $R'$ as in Figure ? and in particular lifts to a loop $G_t$ in the space of monic polynomials with degree $2n-1$ and simple roots. Note that this loop satisfies the odd symmetry $G_{t+\pi}=-G_t$.

%Consider now a singular Rampichini diagram $R_{\{t_i,j(i)\}_{i=1,2,\ldots,M}}$ that is related to $R$. Then there is a similarly symmetric singular Rampichini diagram $R'_{\{t_i,j(i)\}_{i=1,2,\ldots,2M}}$ that is related to $R$, where $t_{i+M}=t_i+\pi$ and $j(i+M)=2n-1-j(i)$, shown in Figure?.
Let $R$ be an odd, pure Rampichini diagram corresponding to a P-fibered braid $B$ on $n$ strands.
By Lemma~\ref{lem:cons} we obtain a singular Rampichini diagram $R_{\{t_i,j(i)\}_{i=1,2,\ldots,M}}$ from a loop $h_t$ in the space of polynomials, whose coefficients are trigonometric polynomials, for any given set of $t_i$ and $j(i)$. Furthermore, by Lemma~\ref{lem:oddversion} the new trigonometric loop $h_t$ still is odd, satisfying $h_{t+\pi}(u)=-h_t(-u)$, if the data $(t_i,j(i))$ displays an odd symmetry. We write $B_{sing}$ for the singular braid that is formed by the roots of $h_t$.

We can use the construction from Theorem~\ref{thm:weighted} to obtain a semiholomorphic, radially weighted homogeneous polynomial
\begin{equation}\label{eq:scale}
f(u,r\rme^{\rmi t})=r^{kn}h(\tfrac{u}{r^{k}},\rme^{\rmi t}),
\end{equation}
where $h(u,t)=h_t(u)$, $n=\deg(h_t)$ and $k$ is a sufficiently large odd integer. The fact that this is a polynomial (as opposed to a rational map or a function involving square roots) follows from the same reasoning as Theorem~\ref{thm:weighted}.

Note that the variety of $f$ is the cone of the closed singular braid $B_{sing}$. Since the singular crossings are critical points, $f$ does not have an isolated singularity (not even weakly isolated). It is a degenerate mixed function in the sense of Oka \cite{oka} and in the sense of \cite{bodeeder}. In \cite{oka} Oka generalizes the definition of Newton polygons of complex polynomials to mixed functions and in \cite{bodeeder} we investigate deformations of non-degenerate mixed functions and found that adding higher order terms (above the boundary of the Newton polygon of the mixed function) do not change the link of a singularity. Since we are now dealing with a degenerate function, our situation is different. We claim that we can add a term of the form $r^m A(\rme^{\rmi t})$, where $m$ is a large even integer and $A$ is an even polynomial in $\rme^{\rmi t}$ and $\rme^{-\rmi t}$, i.e., $A(t+\pi)=A(t)$, and construct in this way a polynomial with an isolated singularity. The function $A$ is found via trigonometric interpolation.

Recall that the critical values of $h_t$ are denoted by $w_{j,n-1}(t)$, whose zeros occur at $t=t_i$, $i=1,2,\ldots,M$, corresponding to singular crossings in $B_{sing}$ and singularities in the singular Rampichini diagram. We assume that the conditions from Lemma~\ref{lem:oddversion} are satisfied, so that $M$ is even and the $t_i$ are indexed such that $t_{i+M/2}=t_i+\pi\text{ mod }2\pi$ for all $i=1,2,\ldots,M/2$. A singular crossing of $B_{sing}$ corresponds to a circle in the singular Rampichini diagram at $t_i$ on the curve corresponding to the critical value $w_{j(i),n-1}(t_i)$. This means that $w_{j,n-1}(t)=0$ if and only if $j=j(i)$ and $t=t_i$. By construction $\arg(w_{j(i),n-1}(t))$ and $\tfrac{\partial\arg(w_{j(i),n-1})}{\partial t}$ have well-defined limits as $t$ goes to $t_i$, which we denote by $\arg(w_{j(i),n-1}(t_i))$ and $\tfrac{\partial\arg(w_{j(i),n-1})}{\partial t}(t_i)$, respectively.

We will find $A$ via a trigonometric interpolation that fixes the values of $A$ and $\tfrac{\partial A}{\partial t}$ at each $t_i$, $i=1,2,\ldots,M$. We can take a solution of this interpolation problem to be of degree $M$. Set $m=\max\{kn+1,M\}$, where $k$ is the integer chosen in Eq.~\eqref{eq:scale}.

%Recall that by Lemma~\ref{lem:ranal} the absolute values $|w_{j,n-1}(t)|$ of the critical values are real analytic in $t$.  every $t_i$ with $j(i)=j$ the lowest order of its power series centered at $t_i$ is $2$. 
We define $A:\mathbb{R}\to\mathbb{C}$ to be the even finite Fourier series that solves
\begin{align}
\arg(A(t_i))&=\arg(w_{j(i),n-1}(t_i))+\pi,\nonumber\\
\frac{\partial |A|}{\partial t}(t_i)&=0,\nonumber\\
\frac{\partial \arg(A)}{\partial t}&=\frac{kn+m}{2 m-kn}\frac{\partial \arg(w_{j(i),n-1})}{\partial t}(t_i),
\end{align}
for all $i=1,2,\ldots,M/2$. By the even symmetry of $A$ this also fixes the value of $A$ and its derivative at $t_i$, $i=M/2+1,M/2+2,\ldots,M$. It is implicit in the interpolation data that $A(t_i)\neq 0$ for all $i$.

\begin{lemma}\label{lem:intA}
Let $A:\mathbb{R}\to\mathbb{C}$ be an even finite Fourier series as above. Then $p=f+r^m A$ is a semiholomorphic polynomial with an isolated singularity.
\end{lemma}
\begin{proof}
The function $f$ is a semiholomorphic polynomial. Since $m$ is at least the degree of $A$, we can cancel the denominator of $r^m A(\rme^{\rmi t})=(v\overline{v})^{m/2} A(\tfrac{v}{\sqrt{v\overline{v}}})$. Since $A$ is even, all remaining factors of $\sqrt{v\overline{v}}$ come with an even exponent, so that $f+r^mA$ is a polynomial in $u$, $v$ and $\overline{v}$.

One advantage of working with semiholomorphic polynomials is that a critical point must satisfy $\tfrac{\partial p}{\partial u}=0$. Since $\tfrac{\partial p}{\partial u}=\tfrac{\partial f}{\partial u}$, this can only happen at critical points of $f$, considered as a family of complex polynomials, parametrised by $v$. Note that $p|_{v=0}=u^n$, so that the origin is the only critical point with $v=0$. The $u$-coordinate of a point with $\tfrac{\partial f}{\partial u}=0$ is thus $r^k c_j(t)$, where $c_j(t)$ is a critical point of $h_t$. Thus all points, where the derivative of $p$ with respect to $u$ vanishes are of the form $(r^k c_j(t),r\rme^{\rmi t})$ with $p(r^k c_j(t),r\rme^{\rmi t})=r^{kn}w_{j,n-1}(t)+r^mA(t)$. We now consider the real Jacobian matrix of $p$ at such a point in coordinates $(\text{Re}(u),\text{Im}(u),r,t)=(\text{Re}(u),\text{Im}(u),|v|,\arg(v))$:
\begin{equation}
\begin{pmatrix}
0 & 0 & kn r^{kn-1}\text{Re}(w_{j,n-1}(t))+mr^{m-1}\text{Re}(A(t)) & r^{kn} \tfrac{\partial \text{Re}(w_{j,n-1}(t))}{\partial t}+r^m \tfrac{\partial \text{Re}(A(t))}{\partial t}\\
0 & 0 & kn r^{kn-1}\text{Im}(w_{j,n-1}(t))+mr^{m-1}\text{Im}(A(t)) &  r^{kn} \tfrac{\partial \text{Im}(w_{j,n-1}(t))}{\partial t}+r^m \tfrac{\partial \text{Im}(A(t))}{\partial t}
\end{pmatrix}.
\end{equation}
In order to show that the singularity at the origin is isolated, we show that for small positive values of $r$ this matrix has full rank. For this we compute the determinant of the 2-by-2 matrix given by the non-zero entries above. The determinant is equal to
\begin{align}
&knr^{2kn-1}(\text{Re}(w_{j,n-1}(t))\text{Im}(w_{j,n-1}'(t))-\text{Im}(w_{j,n-1}(t))\text{Re}(w_{j,n-1}'(t)))\nonumber\\
&+knr^{kn+m-1}(\text{Re}(w_{j,n-1}(t))\text{Im}(A'(t))-\text{Im}(w_{j,n-1}(t))\text{Re}(A'(t)))\nonumber\\
&+mr^{kn+m-1}(\text{Re}(A(t))\text{Im}(w_{j,n-1}'(t))-\text{Im}(A(t))\text{Re}(w_{j,n-1}'(t)))\nonumber\\
&+mr^{2m-1}(\text{Re}(A(t))\text{Im}(A'(t))-\text{Im}(A(t))\text{Re}(A'(t))),\label{eq:det}
\end{align}
where the dash denotes the derivative with respect to $t$. By assumption we have that $m>kn$.
%Since we only want to know if the determinant is non-zero, we can divide by $r^{2kn-1}$. 

For a set of open neighbourhoods $U_i$ (to be determined later) of the $t_i$s the resulting expression goes to 
\begin{align}\label{eq:uilimit}
&(\text{Re}(w_{j,n-1}(t))\text{Im}(w_{j,n-1}'(t))-\text{Im}(w_{j,n-1}(t))\text{Re}(w_{j,n-1}'(t))\nonumber\\
=&|w_{j,n-1}(t)|^2\arg(w_{j,n-1})'(t)\neq0
\end{align} 
on $[0,2\pi]\backslash\left(\underset{j(i)=j}{\underset{i \text{ s.t.}}{\cup}}U_i\right)$ as $r$ goes to zero. Thus we obtain a non-zero determinant for all points $(r^kc_j(t),r\rme^{\rmi t})$ with sufficiently small $r$ except possibly for those with $j=j(i)$ and $t\in U_i$ for some $i$.

For every $i=1,2,\ldots,M$ we define $R_1(t)=|A(t)|$ and $\delta_i(t)$ such that 
\begin{equation}
A(t)=R_1(t)\rme^{\rmi (\arg(w_{j(i),n-1}(t))+\delta_i(t))}
\end{equation} 
on $U_i$. Furthermore, we write $\alpha_i(t)$ for $\arg(w_{j(i),n-1}(t))$ and $R_2(t)=|w_{j(i),n-1}(t)|$. This means that
\begin{align}
\text{Re}(A(t))&=R_1(t) (\cos(\arg(w_{j(i),n-1}(t)))\cos(\delta_i(t))-\sin(\arg(w_{j(i),n-1}(t)))\sin(\delta_i(t)),\nonumber\\
\text{Im}(A(t))&=R_1(t)(\cos(\arg(w_{j(i),n-1}(t)))\sin(\delta_i(t))+\sin(\arg(w_{j(i),n-1}(t)))\cos(\delta_i(t)).
\end{align} 
We now look at the individual terms in Eq.~\eqref{eq:det}. First,
\begin{align}
&\text{Re}(w_{j(i),n-1}(t))\text{Im}(w_{j(i),n-1}'(t))-\text{Im}(w_{j(i),n-1}(t))\text{Re}(w_{j(i),n-1}'(t))\nonumber\\
=&R_2(t)^2\alpha_i(t)',
\end{align}
which is 0 if and only if $t=t_i$.

Second, 
\begin{align}
&\text{Re}(w_{j(i),n-1}(t))\text{Im}(A'(t))-\text{Im}(w_{j(i),n-1}(t))\text{Re}(A(t))'\nonumber\\
=&R_2(t)\left\{\cos(\alpha_i(t))\right.\nonumber\\
&\left[R_1(t)\left(-\sin(\alpha_i(t))\alpha_i(t)'\sin(\delta_i(t))+\cos(\alpha_i(t))\cos(\delta_i(t))\delta_i(t)'\right.\right.\nonumber\\
&\left.+\cos(\alpha_i(t))\alpha_i(t)'\cos(\delta_i(t))-\sin(\alpha_i(t))\sin(\delta_i(t))\delta_i(t)'\right)\nonumber\\
&\left.+R_1'(t)(\cos(\alpha_i(t))\sin(\delta_i(t))+\sin(\alpha_i(t))\cos(\delta_i(t)))\right]\nonumber\\
&-\sin(\alpha_i(t))\left[R_1(t)\left(-\sin(\alpha_i(t))\alpha_i(t)'\cos(\delta_i(t))-\cos(\alpha_i(t))\sin(\delta_i(t))\delta_i(t)'\right.\right.\nonumber\\
&\left.-\cos(\alpha_i(t))\alpha_i(t)'\sin(\delta_i(t))-\sin(\alpha_i(t))\cos(\delta_i(t))\delta_i(t)'\right]\nonumber\\
&\left.+R_1'(t)(\cos(\alpha_i(t))\cos(\delta_i(t))-\sin(\alpha_i(t))\sin(\delta_i(t)))\right\}\nonumber\\
=&R_2(t)(R_1(t)\cos(\delta_i(t))(\delta_i(t)'+\alpha_i(t)')+R_1'\sin(\delta_i(t))).
\end{align}

Furthermore,
\begin{align}
&\text{Re}(A(t))\text{Im}(w_{j(i),n-1}(t))'-\text{Im}(A(t))\text{Re}(w_{j(i),n-1}(t))'\nonumber\\
=&R_1(t)\left[(\cos(\alpha_i(t))\cos(\delta_i(t))-\sin(\alpha_i(t))\sin(\delta_i(t)))(R_2'(t)\sin(\alpha_i(t))+R_2(t)\cos(\alpha_i(t))\alpha_i(t)')\right.\nonumber\\
&\left.-(\cos(\alpha_i(t))\sin(\delta_i(t))+\sin(\alpha_i(t))\cos(\delta_i(t)))(R_2'(t)\cos(\alpha_i(t))-R_2(t)\sin(\alpha_i(t))\alpha_i(t)')\right]\nonumber\\
=&R_1(t)R_2(t)\cos(\delta_i(t))\alpha_i(t)'-R_1(t)R_2'(t)\sin(\delta_i(t)).
\end{align}

Lastly, 
\begin{equation}
\text{Re}(A(t))\text{Im}(A(t))'-\text{Im}(A(t))\text{Re}(A(t))'=R_1(t)^2(\alpha_i(t)'+\delta_i(t)').
\end{equation}

Thus Eq.~\eqref{eq:det} becomes
\begin{align}\label{eq:det2}
&\alpha_i(t)'(knr^{2kn-1}R_2(t)^2+(kn+m)r^{kn+m-1}R_1(t)R_2(t)\cos(\delta_i(t))+mr^{2m-1}R_1(t)^2)\nonumber\\
&+\delta_i(t)'(knr^{kn+m-1}R_1(t)R_2(t)\cos(\delta_i(t))+mr^{2m-1}R_1(t)^2)\nonumber\\
&+\sin(\delta_i(t))r^{kn+m-1}(knR_1'(t)R_2(t)-mR_1(t)R_2'(t)).
\end{align}

Each term is a real analytic function of $t$ (cf. Lemma~\ref{lem:ranal}) and thus can be locally written as a power series centered at $t_i$. Consider the lowest order term of the series of $\sin(\delta_i(t))(knR_1'(t)R_2(t)-mR_1(t)R_2'(t))$ and recall that $2$ is the lowest order of $R_2$. Since $0$ is the lowest order of $R_1$ and $1$ is the lowest order of $\sin(\delta_i)$ (since $\delta_i(t_i)$ is either $0$ or $\pi$, and $\delta_i'(t_i)\neq 0$ by the interpolation property), the lowest order of $\sin(\delta_i)mR_1R_2'$ is $2$, while the lowest order of $\sin(\delta_i)knR_1'R_2$ is greater than $2$. Likewise, the lowest orders of the series for $\alpha_i'(kn+m)R_1R_2\cos(\delta_i)$ and $\delta_i'knR_1R_2\cos(\delta_i)$ centered at $t_i$ are also both $2$.

Let $r_{2}$ denote the lowest order coefficient of $R_2$. We compute the coefficient of $(t-t_i)^{2}$ for 
\begin{align}\label{eq:sincos}
&\alpha_i(t)'(kn+m)R_1(t)R_2(t)\cos(\delta_i(t))\nonumber\\
&+\delta_i(t)'knR_1(t)R_2(t)\cos(\delta_i(t))+\sin(\delta_i(t))(knR_1'(t)R_2(t)-mR_1(t)R_2(t)')
\end{align}
and obtain
\begin{align}
&\alpha_i'(t_i)(kn+m)r_{2}R_1(t_i)\cos(\delta_i(t))\nonumber\\
&+\delta_i'(t_i)knr_{2}R_1(t_i)\cos(\delta_i(t))-\cos(\delta_i(t))\delta_i'(t_i)mR_1(t_i)2 r_{2}\nonumber\\
=&r_{2}R_1(t_i)\cos(\delta_i(t_i))(\alpha_i'(t_i)(kn+m)+\delta_i'(t_i)(kn-2m)),
\end{align}
which is $0$, since $\delta_i'(t_i)=\tfrac{\alpha_i(t_i)(kn+m)}{2m-kn}$ by the interpolation property.

Therefore, the lowest order of Eq.~\eqref{eq:sincos} is at least $3$. It follows that there is a neighbourhood $U_i$ of $t_i$ where the absolute value of Eq.~\eqref{eq:sincos} is less than or equal to the absolute value of
\begin{equation}
2R_1(t)R_2(t)\sqrt{knm}\alpha_i(t)',
\end{equation}
whose lowest order is $2$. This implies that the absolute value of the determinant in Eq.~\eqref{eq:det2} is at least
\begin{align}
&|\alpha_i(t)'r^{2kn-1}(\sqrt{kn}R_2(t)-\sqrt{m}R_1(t)r^{m-kn})^2+\delta_i(t)'mr^{2m-1}R_1(t)^2|\nonumber\\
=&|\alpha_i(t)'|r^{2kn-1}(\sqrt{kn}R_2(t)-\sqrt{m}R_1(t)r^{m-kn})^2+|\delta_i(t)'|mr^{2m-1}R_1(t)^2\nonumber\\
\geq&|\delta_i(t)'|mr^{2m-1}R_1(t)^2>0,
\end{align}
since $r>0$ and $U_i$ can be chosen such that $R_1(t)>0$ and $|\delta_i(t)'|>0$. The first equality holds because $\alpha_i(t)'$ and $\delta_i(t)'$ have the same sign. 

Thus there are no critical points $(r^kc_j(t),r\rme^{\rmi t})$ of $p$ with $t\in U_i$ and $r>0$. Note that the $U_i$s do not depend on $r$. Since $S^1\backslash\left(\underset{i=1}{\overset{M}{\cup}} U_i\right)$ is compact, sufficiently small choices of $r>0$ guarantee that $(r^kc_j(t),r\rme^{\rmi t})$ with $t\in[0,2\pi]\backslash\left(\underset{i=1}{\overset{M}{\cup}} U_i\right)$ is not a critical point of $p$ either (cf. Eq~\eqref{eq:uilimit}). Therefore, the origin is an isolated singularity of $p$.
\end{proof}

\section{The proof of Theorem \ref{thm:main}}\label{sec:proof}

In order to finish the proof of Theorem \ref{thm:main} we have to determine the links of the constructed singularities via their Rampichini diagrams.

\begin{lemma}
Let $B$ be the vanishing set of $p|_{r=\varepsilon}:\mathbb{C}\times S^1\to\mathbb{C}$, with $\varepsilon>0$ sufficiently small. Then $B$ is a closed geometric braid, whose braid isotopy class is P-fibered.\\
A Rampichini diagram for $B$ is obtained from $R_{\{(t_i,j(i))\}_{i=1,2,\ldots,M}}$ by inserting inner loops $\gamma_{j(i)}^{\varepsilon_i}$ at $t=t_i$, $i=1,2,\ldots,M/2$, where $\varepsilon_i=\sign\left(\tfrac{\partial \arg(w_{j(i),n-1})}{\partial t}(t_i)\right)$, and removing the circles at $t=t_i$, $i=M/2+1,M/2+2,\ldots,M$.
\end{lemma}
\begin{remark}
We do not claim that $B$ itself is a P-fibered \textit{geometric} braid, as $p|_{r=\varepsilon}$ might have argument-critical points. However, $B$ is braid isotopic to a P-fibered geometric braid. 
\end{remark}
\begin{proof}
The function $p|_{r=\varepsilon}$ corresponds to a loop of monic complex polynomials of degree $n$ with critical values $\varepsilon^{kn}w_{j,n-1}(t)+\varepsilon^mA(t)$. 

%We have to show that $\tfrac{\partial \arg(\varepsilon^{kn}w_{j,n-1}+\varepsilon^mA)}{\partial t}$ never vanishes.

We first have to show that no critical value is equal to $0$, which implies that $B$ is a closed braid. For this note that at $t=t_i$ we have $w_{j(i),n-1}(t_i)=0$ and $A(t_i)\neq 0$, so that $\varepsilon^{kn}w_{j(i),n-1}(t_i)+\varepsilon^mA(t_i)\neq 0$. Using the notation from the previous section we may write
\begin{equation}\label{eq:nonzerow}
\rme^{\rmi \alpha_i(t)}\varepsilon^{kn}\left(R_2(t)+R_1(t)\varepsilon^{m-kn}\rme^{\rmi \delta_i(t)}\right).
\end{equation}
Since $\delta_i(t_i)=0$ and $\delta_i'(t_i)\neq 0$, there is a neighbourhood $U_i$ of $t_i$, where $\delta_i(t)$ is neither 0 nor $\pi$, so that Eq.~\eqref{eq:nonzerow} is non-zero for all $t$ in $U_i$. The complement of these neighbourhoods is a compact subset of $S^1$, on which $R_2$ is non-zero, so that we can ensure that Eq.~\eqref{eq:nonzerow} is non-zero by choosing $\varepsilon$ sufficiently small. This shows that no critical value is 0. Thus there are no double roots of $p|_{r=\varepsilon}$. It follows that $B$ is a closed braid, since $p$ is a semiholomorphic polynomial.

Now we consider the square diagram associated to the loop of polynomials $p|_{r=\varepsilon}$. Outside of the $U_i$s the critical values approximate $w_{j,n-1}(t)$ arbitrarily well, so that outside the $U_i$s the square diagram looks exactly like the square diagram of $h_t$, which is a singular Rampichini diagram, all of whose singularities lie inside the $U_i$s. Likewise, the critical values whose corresponding curves do not have a singularity in $U_i$ are also arbitrarily well approximated by the critical values of $p|_{r=\varepsilon}$. Thus we only have to study what happens in a neighbourhood of the singular points of the singular Rampichini diagram associated to $h_t$.

There are two cases to consider. First, let $i\in\{1,2,\ldots,M/2\}$. Eq.~\eqref{eq:nonzerow} shows that $t=t_i$ is the unique value in $U_i$ for which $\arg(\rme^{\rmi\alpha_i(t)}\varepsilon^{kn}(R_2(t)+R_1(t)\varepsilon^{m-kn}\rme^{\rmi \delta_i(t)}))=\alpha_i(t_i)+\pi$. At $t=t_i$ the derivative of $\arg(\rme^{\rmi\alpha_i(t)}\varepsilon^{kn}(R_2(t)+R_1(t)\varepsilon^{m-kn}\rme^{\rmi \delta_i(t)})$ with respect to $t$ is equal to $\delta_i'(t_i)$, which by construction has the same sign as $\alpha_i'(t_i)$. 

Let $\tilde{\varepsilon}_i>0$ and $U_i=(t_i-\tilde{\varepsilon}_i,t_i+\tilde{\varepsilon}_i)$. Note that we may choose $\tilde{\varepsilon}_i$ so small that both $\arg(w_{j(i),n-1})(t_i-\tilde{\varepsilon}_i))$ and $\arg(w_{j(i),n-1})(t_i-\tilde{\varepsilon}_i))$ are arbitrarily close to $\arg(w_{j(i),n-1}(t_i))$. Furthermore, we can assume that in $U_i$ the critical value $\varepsilon^{kn}w_{j(i),n-1}(t)+\varepsilon^mA(t)$ has the smallest absolute value of all critical values. Then the argument of $\varepsilon^{kn}w_{j(i),n-1}(t_i-\tilde{\varepsilon}_i)+\varepsilon^mA(t_i-\tilde{\varepsilon}_i)$ is arbitrarily close to the argument of $w_{j(i),n-1}(t_i-\tilde{\varepsilon}_i)$ and the argument of $\varepsilon^{kn}w_{j(i),n-1}(t_i+\tilde{\varepsilon}_i)+\varepsilon^mA(t_i+\tilde{\varepsilon}_i)$ is arbitrarily close to the argument of $w_{j(i),n-1}(t_i+\tilde{\varepsilon}_i)$. Since the path $\varepsilon^{kn}w_{j(i),n-1}(t)+\varepsilon^mA(t)$ crosses the line $\varphi=\alpha_i(t_i)+\pi$ exactly once and since it has the smallest absolute value among all critical values, this path can be deformed into one without any argument-critical points and that winds around the origin exactly once (cf. Figure~\ref{fig:deformation}). 

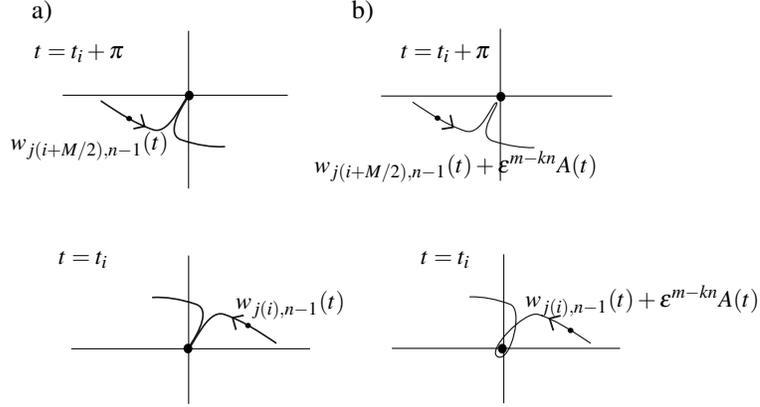
\begin{figure}
\labellist
\pinlabel a) at -100 1900
\pinlabel b) at 1450 1900
\small
\pinlabel $t=t_i+\pi$ at 80 1700
\pinlabel $t=t_i$ at 100 700
\pinlabel $t=t_i+\pi$ at 1850 1700
\pinlabel $t=t_i$ at 1850 700
\pinlabel $w_{j(i+M/2),n-1}(t)$ at 130 1250
\pinlabel $w_{j(i),n-1}(t)$ at 1100 480
\pinlabel $w_{j(i+M/2),n-1}(t)+\varepsilon^{m-kn}A(t)$ at 1900 1150
\pinlabel $w_{j(i),n-1}(t)+\varepsilon^{m-kn}A(t)$ at 2800 500
\endlabellist
\centering
\includegraphics[height=5cm]{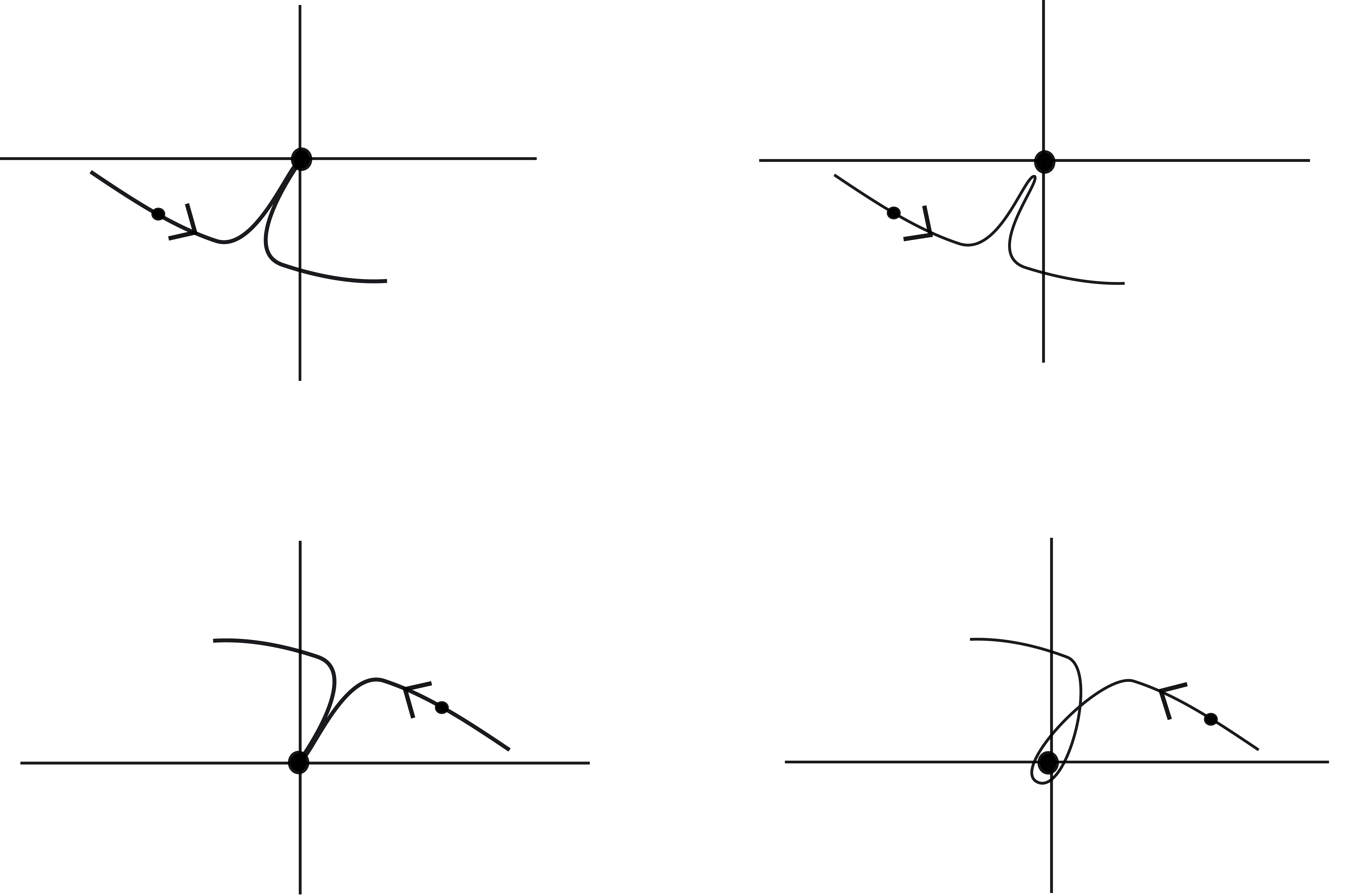}
\caption{a) The motion of the critical values $w_{j(i),n-1}(t)$ and $w_{j(i+M/2),n-1}(t)$ in neighbourhoods of $t=t_i$ and $t=t_{i+M/2}=t_i+\pi$, respectively. b) The motion (up to isotopy in $(\mathbb{C}\backslash\{0\})\times S^1$) of the deformed critical values $w_{j(i),n-1}(t)+\varepsilon^{m-kn}A(t)$ and $w_{j(i+M/2),n-1}(t)+\varepsilon^{m-kn}A(t)$ in neighbourhoods of $t=t_i$ and $t=t_{i+M/2}=t_i+\pi$, respectively. \label{fig:deformation}}
\end{figure}

This deformation of the loop of critical values lifts to a deformation of the braid $B$ \cite{bode:braided}. By construction the direction of this path (clockwise $(-1)$ or counterclockwise $(+1)$) is given by $\sign(\alpha_i'(t_i))$ and thus consistent with the monotonicity of the corresponding curve in the rest of the square diagram. Note that the deformation of the path taken by $\varepsilon^{kn}w_{j(i),n-1}(t)+r^mA(t)$ is an inner loop $\gamma_{j(i)}^{\sign(\alpha_i'(t_i))}$, since the moving critical value has the smallest absolute value throughout $U_i$.

Thus after a deformation of the critical values, which lifts to a braid isotopy of $B$, there are no argument-critical points in $U_i$, $i=1,2,\ldots,M/2$, and by Lemma~\ref{lem:insert} the labels at $t=t_i-\tilde{\varepsilon}_i$ match the labels at $t=t_i+\tilde{\varepsilon}_i$.

Now we consider $i=M/2+1,M/2+2,\ldots,M$. We know that at $t=t_i$ we have 
\begin{equation}
\arg(\rme^{\rmi\alpha_i(t)}\varepsilon^{kn}(R_2+R_1\varepsilon^{m-kn}\rme^{\rmi \delta_i(t)}))=\alpha_i(t_i).
\end{equation}  
Eq.~\eqref{eq:nonzerow} shows that $t=t_i$ is the unique point in $U_i$ where this is true and there is no point in $U_i$ with $\arg(\rme^{\rmi\alpha_i(t)}\varepsilon^{kn}(R_2+R_1\varepsilon^{m-kn}\rme^{\rmi \delta_i(t)}))=\alpha_i(t_i)+\pi$. It follows that the path taken by the critical value $\varepsilon^{kn}w_{j(i),n-1}(t)+\varepsilon^mA(t)$ in $U_i$ does not twist around the origin. It has winding number zero and can thus be deformed to a curve without argument-critical points, connecting $\varepsilon^{kn}w_{j(i),n-1}(t_i-\tilde{\varepsilon}_i)+\varepsilon^mA(t_i-\tilde{\varepsilon}_i)$ and $\varepsilon^{kn}w_{j(i),n-1}(t_i+\tilde{\varepsilon}_i)+\varepsilon^mA(t_i+\tilde{\varepsilon}_i)$ (cf. Figure~\ref{fig:deformation}).

Again this deformation of the loop of critical values lifts to a deformation of the polynomial and thus to a braid isotopy of $B$. In the $U_i$s with $i\in\{M/2+1,M/2+2,\ldots,M\}$ the deformed square diagram of $p|_{r=\varepsilon}$ thus looks exactly like the singular Rampichini diagram of $h_t$, except that the singularities have been removed (cf. Figure~\ref{fig:deformation}). 

Therefore, the complete deformed square diagram of $p|_{r=\varepsilon}$ is obtained from the singular Rampichini diagram $R_{\{(t_i,j(i))\}_{i=1,2,\ldots,M}}$ of $h_t$ by inserting inner loops of the appropriate signs at the singularities at $t=t_i$, $i=1,2,\ldots,M/2$, and simply removing the circles representing the singularities at $t=t_i$, $i=M/2+1,M/2+2,\ldots,M$.

Since there are no argument-critical points, the deformed diagram is a Rampichini diagram and $B$ is P-fibered.
\end{proof}

By the same arguments as in Theorem I.1 in \cite{bode:polynomial} the link of the singularity is equivalent to the closure of $B$.

%We know that the vanishing set of $f$ is the cone over a singular braid $B_{sing}$ with an odd singular Rampichini diagram $R_{\{(t_i,j(i))\}_{i=1,2,\ldots,M}}$.

%\begin{lemma}
%The Rampichini diagram of $B$ is obtained from $R_{\{(t_i,j(i))\}_{i=1,2,\ldots,M}}$ by inserting inner loops $\gamma_{j(i)}^{\varepsilon_i}$ at $t=t_i$, $i=1,2,\ldots,M/2$, where $\varepsilon_i=\sign\left(\tfrac{\partial \arg(w_{j(i),n-1})}{\partial t}(t_i)\right)$, and removing the circles at $t=t_i$, $i=M/2+1,M/2+2,\ldots,M$.
%\end{lemma}
%\begin{proof}
%\end{proof}

Note that this completes the proof of Theorem~\ref{thm:main}. 
\begin{proof}[Proof of Theorem~\ref{thm:main}]
Let $B_2$ be a P-fibered braid, whose Rampichini diagram $R'$ can be obtained from an odd, pure Rampichini diagram $R$ by the insertion of inner loops. Then there is a corresponding odd loop of polynomials $h_t$, parametrised by trigonometric polynomials, whose singular Rampichini diagram is $R$, but with singularities at $t=t_i$, $i=1,2,\ldots,M/2$, the positions where we want to insert inner loops, and at $t=t_{i+M/2}=t_i+\pi$. We construct the functions $f$, $A$ and $p$ and by the previous lemmas $p$ has an isolated singularity at the origin and the closure of $B_2$ is the link of the singularity. The polynomial $p$ is semiholomorphic and its degree with respect to $u$ is equal to $n$, the number of strands of $B_2$.  
\end{proof}

Theorem \ref{thm:main} deals with P-fibered braids whose Rampichini diagrams are obtained from odd, pure Rampichini diagrams by inserting inner loops. The result remains true if we replace ``odd'' by ``simple''.

\begin{theorem}\label{thm:main2}
Let $B_1$ be a P-fibered braid on $n$ strands with a simple, pure Rampichini diagram. Let $B_2$ be a P-fibered braid whose Rampichini diagram is obtained from that of $B_1$ by the insertion of inner loops. Then the closure of $B_2$ is real algebraic.\\
Furthermore,  the corresponding real polynomial map with an isolated singularity can be taken to be semiholomorphic (i.e., it can be written as a polynomial in complex variables $u$, $v$ and the complex conjugate $\overline{v}$) and of degree $2n-1$ with respect to the complex variable $u$.
\end{theorem}
\begin{proof}
Examples of Rampichini diagrams of $B_2$ and $B_1$ are shown in Figure~\ref{fig:rampi_sequence}a) and b), respectively. We want to show that $B_2$ can also be obtained from a P-fibered braid with odd, pure Rampichini diagram by the insertion of inner loops. As in Section~\ref{sec:odd} we can construct an odd Rampichini diagram $R_{odd}$ for $\imath_n(B_1)m_n(B_1)$. Then the desired Rampichini diagram is obtained from $R_{odd}$ by inserting inner loops into the lower left quadrant of $R_{odd}$ (Figure~\ref{fig:rampi_sequence}d)). We can thus construct an odd singular Rampichini diagram with singularities at $t=t_i\in(0,\pi)$, $i=1,2,\ldots,M/2$, $j(i)\in\{1,2,\ldots,n-1\}$ and $t=t_{i+M/2}=t_i+\pi$, $i=1,2,\ldots,M/2$, $j(i+M/2)\in\{n,n+1,\ldots,2n-2\}$ corresponding to an odd loop of polynomials $h_t$ as in the proof of Theorem~\ref{thm:main} (cf. Figure~\ref{fig:rampi_sequence}c)).

\begin{figure}[h]
\labellist
\pinlabel a) at 100 3800
\pinlabel b) at 1000 3800
\pinlabel c) at 200 2950
\pinlabel d) at 200 1500
\endlabellist
\centering
\includegraphics[height=12cm]{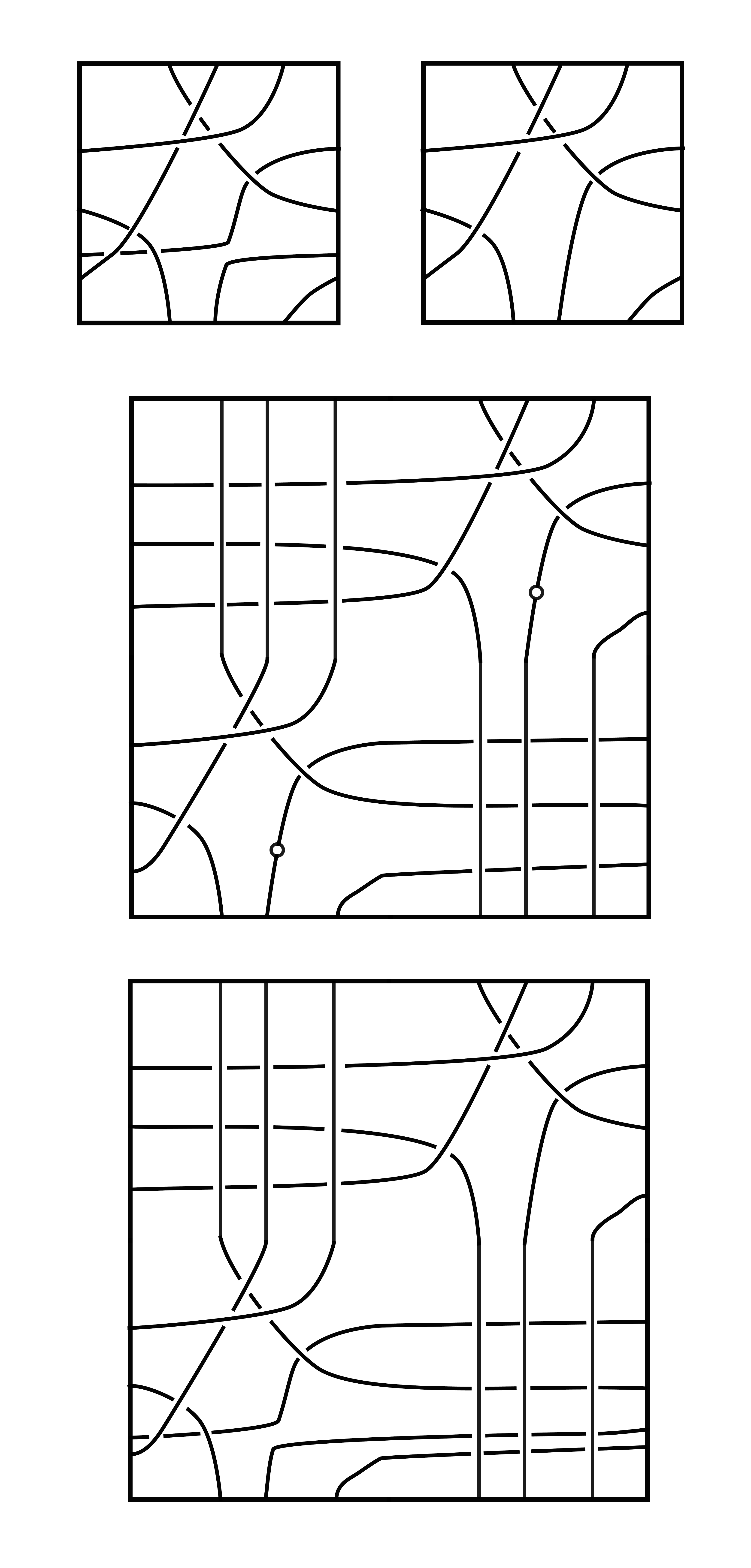}
\caption{A sequence of Rampichini diagrams, where we have omitted the labels. a) A Rampichini diagram that is obtained from a simple, pure Rampichini diagram by inserting an inner loop. b) The corresponding simple, pure Rampichini diagram. c) The singular Rampichini diagram with odd symmetry. d) The final Rampichini diagram. \label{fig:rampi_sequence}}
\end{figure}
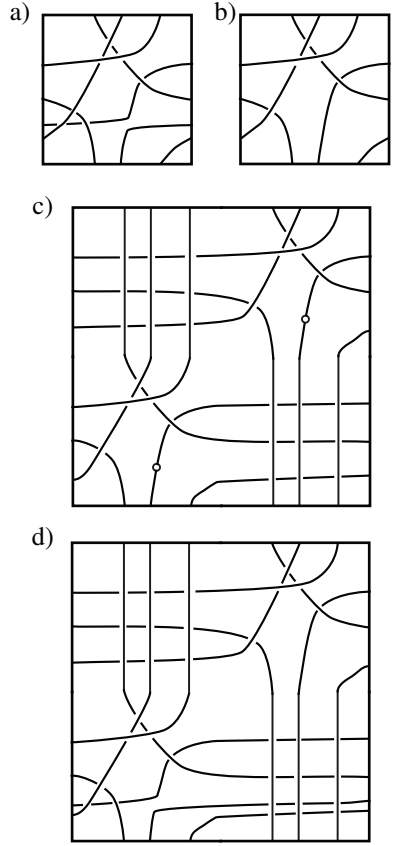

The braid that one obtains from this odd singular Rampichini diagram by inserting inner loops at the singularities with $t_i\in(0,\pi)$ and deleting the circles of the singularities with $t_i\in(\pi,2\pi)$ is then $\imath_n(B_2)m_n(B_1)$. Since the Rampichini diagram of $B_1$ is simple, its closure is the unknot $O$, so that the closure of $\imath_n(B_2)m_n(B_1)$ is a connected sum of the closure of $B_2$ and $O$ and hence equal to the closure of $B_2$.

By Theorem~\ref{thm:main} every P-fibered braid that can be obtained from an odd, pure Rampichini diagram closes to a real algebraic link. Thus the closure of $B_2$ is real algebraic.
\end{proof}

\section{T-homogeneous braids and braids with positive Garside factor}\label{sec:homo}

We want to show that the family of T-homogeneous braid closures can be constructed as real algebraic links as in Theorem~\ref{thm:main} and Theorem~\ref{thm:main2}.

\begin{definition}
We say a braid word $A$ on $n$ strands is a subword of a braid word $B$ on $n$ strands if deleting some letters (i.e., bands) from $B$ results in $A$.
\end{definition}

%Start with T-homogeneous braid word. Pick subword with one BKL-generator for each critical values. Make it symmetric, has simple Rampichini diagram. Perform the construction with $t_k$ equal to the positions of removed generators. Every critical value has the desired label. The upper half forms the connected sum factor of an unknot (Markov destabilization). The corollary follows.

\begin{proof}[Proof of Theorem~\ref{thm:Thomo}]
We know that every T-homogeneous braid $B$ on $n$ strands is P-fibered and by definition $B$ can be represented by a strictly homogeneous word $w$ in a certain subset of BKL-generators, say $a_1,a_2,\ldots,a_{n-1}$ with signs $\varepsilon_j\in\{\pm 1\}$, $j=1,2,\ldots,n$. We can choose a subword of $w$ of length $n-1$ containing each $a_j^{\varepsilon_j}$. This subword is also T-homogeneous and hence P-fibered. Let $R$ denote its Rampichini diagram, which is simple and pure.

It follows from the definition of T-homogeneity that the Rampichini diagram of $B$ can be obtained from $R$ by inserting inner loops. Theorem~\ref{thm:main2} then implies that the closure of $B$ is real algebraic.
\end{proof}

Corollary~\ref{cor:homo} immediately follows, since homogeneous braids are T-homogeneous. By construction the corresponding polynomial is semiholomorphic. The degree of the polynomial with respect to the complex variable $u$ is $2n-1$, where $n$ is the number of strands of the homogeneous braid representative.

We proved in \cite{bodesat} that every link $L$ is a sublink of a real algebraic link that consists of two non-trivially linked copies of $L$. Stallings showed that for every link $L$ there is an unknot $O\in S^3\backslash L$, non-trivially linked with $L$, such that $L\cup O$ is the closure of a homogeneous braid \cite{stallings2}. Combining this with Corollary~\ref{cor:homo} results in a proof of the following corollary.

\begin{corollary}
For every link $L$ in $S^3$ there is an unknot $O\in S^3\backslash L$, non-trivially linked with $L$, such that $L\cup O$ is real algebraic.
\end{corollary}

%Since our construction produces semiholomorphic polynomials, this proves the conjecture by Öztürk that every homogeneous braid closure is the link of an isolated singularity of a semiholomorphic polynomial.
%\section{Braids with positive Garside factor}\label{sec:gar}

Now we prove Theorem~\ref{thm:delta}. It states that any braid that is the product of a positive power of the dual Garside element $\delta$ and a BKL-positive word closes to a real algebraic link. As in the previous proof we show that this family of braids is obtained from a simple, pure Rampichini diagram by inserting inner loops.

\begin{lemma}\label{lem:BKLseq}
For every positive generator $a_{i,j}$ there is a sequence of BKL-moves on the braid word $a_{1,n}a_{1,n-1}\ldots,a_{1,2}$ that results in a BKL-word whose last letter is $a_{i,j}$.
\end{lemma}
\begin{proof}
Assume $i>j$ and take the generator $a_{1,i}$ and move it towards $a_{1,j}$ conjugating the labels in between by $a_{1,i}$ at each step. After this the word looks like this:
\begin{equation}
a_{1,n}a_{1,n-1}\ldots a_{1,i+1}a_{i,i-1}a_{i,i-2}\ldots,a_{i,j+1}a_{1,i}a_{1,j}a_{1,j-1}\ldots a_{1,2}.
\end{equation}
Then exchange $a_{1,i}$ and $a_{1,j}$, conjugating $a_{1,i}$ by $a_{1,j}$ to obtain:
\begin{equation}
a_{1,n}a_{1,n-1}\ldots a_{1,i+1}a_{i,i-1}a_{i,i-2}\ldots,a_{i,j+1}a_{1,j}a_{i,j}a_{1,j-1}\ldots a_{1,2}.
\end{equation}
Now move the generator $a_{i,j}$ all the way to the end of the word conjugating every generator in between by $a_{i,j}$. We thus obtain a word whose last letter is $a_{i,j}$.
\end{proof}

Note that such a sequence can be represented by a square diagram $R$, whose labels at the bottom edge spell $a_{1,n}a_{1,n-1}\ldots,a_{1,2}$ and whose labels at the top edge spell the final word of the sequence whose last letter is $a_{i,j}$. The lines are strictly monotone increasing and the words spelled by the labels at any horizontal slice of the diagram spell the intermediate BKL-word in the sequence. 

Therefore, we can construct for every positive generator a square diagram $R'(a_{i,j})$, whose labels at the bottom and top edge both spell $a_{1,n}a_{1,n-1}\ldots$, $a_{1,2}$, whose curves do not intersect the right edge of the square ($\varphi=2\pi$), and that for some value of $t$ has the property that the label of the critical value with largest $\varphi$-coordinate is $a_{i,j}$. This diagram is obtained from $R$ by composing it with the inverse of $R$, i.e., a square diagram that reverses the sequence of BKL-words from Lemma~\ref{lem:BKLseq}. Note that since all labels are positive, the inverse of $R$ is still realized by strictly monotone increasing lines, see Figure~\ref{fig:Rlast}a). The reason why $R'(a_{i,j})$ is not a Rampichini diagram is that it is not $2\pi$-periodic with respect to the $t$-coordinate. The endpoints at the top edge are strictly to the right of the starting points at the bottom edge.

\begin{proof}[Proof of Theorem~\ref{thm:delta}]
Let $B=\delta P$. Note that $\delta$ is a P-fibered braid with a simple, pure Rampichini diagram as in Figure~\ref{fig:Rlast}b). In particular, the labels at $t=0$ read $a_{1,n},a_{1,n-1},\ldots$, $a_{1,2}$ from left to right. We now have to show that every band in $P$ can be obtained by inserting inner loops. Note that, since $\delta$ is a positive word, all curves in the Rampichini diagram are strictly monotone increasing.

\begin{figure}[h]
\labellist
\Large
\pinlabel a) at  30 820
\pinlabel b) at 960 820
\small
\pinlabel 0 at 90 110
\pinlabel 0 at 130 50
\pinlabel $2\pi$ at 790 50
\pinlabel $2\pi$ at 90 730
\pinlabel $(1,5)$ at 1160 780
\pinlabel $(1,4)$ at 1300 780
\pinlabel $(1,3)$ at 1450 780
\pinlabel $(1,2)$ at 1600 780
\pinlabel $(4,5)$ at 1800 210
\pinlabel $(3,4)$ at 1800 330
\pinlabel $(2,3)$ at 1800 480
\pinlabel $(1,2)$ at 1800 590
\pinlabel $(3,5)$ at 565 250
\pinlabel 0 at 1040 100
\pinlabel 0 at 1070 40
\pinlabel $2\pi$ at 1720 50
\pinlabel $2\pi$ at 1020 730
\Large
\pinlabel $t$ at 100 430
\pinlabel $\varphi$ at 500 20
\pinlabel $t$ at 1030 430
\pinlabel $\varphi$ at 1400 20
\endlabellist
\centering
\includegraphics[height=4cm]{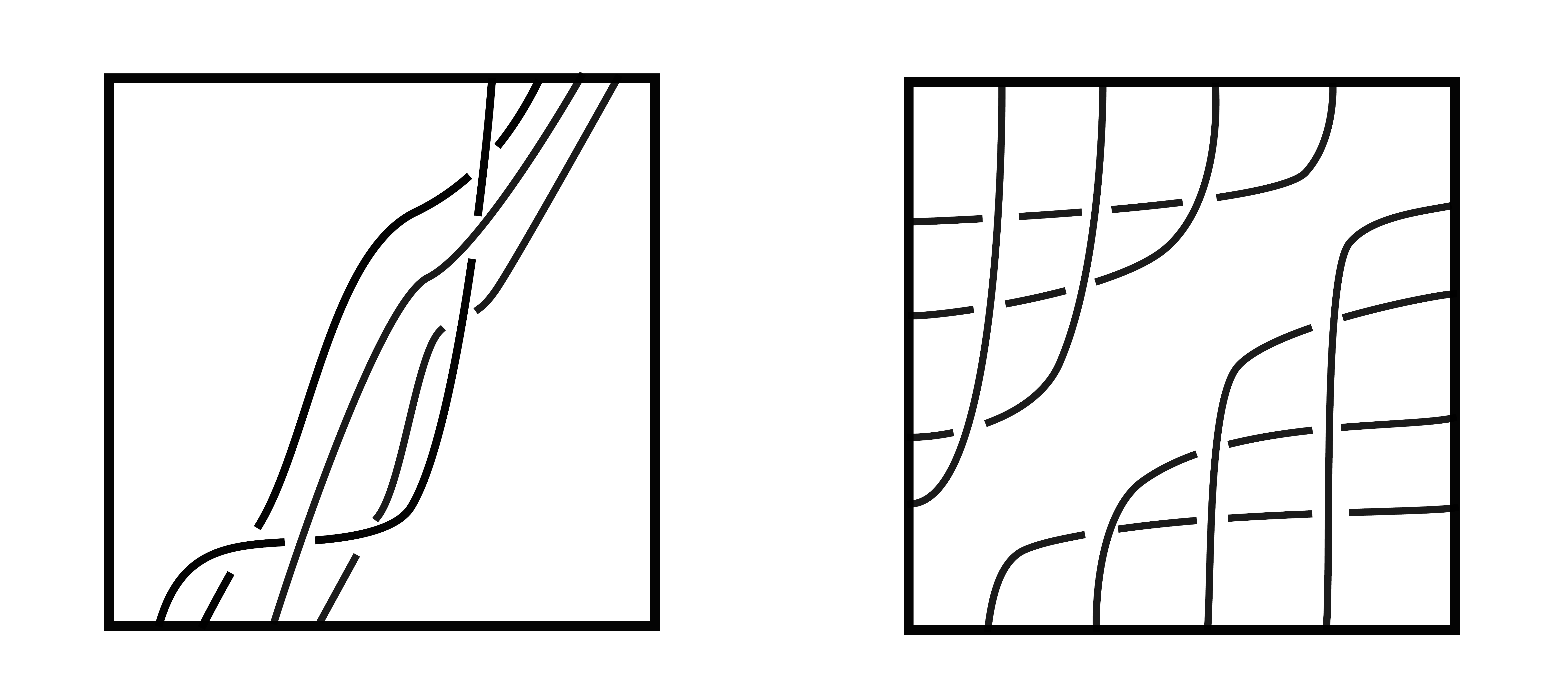}
\labellist
\Large 
\pinlabel c) at -100 1550
\pinlabel $t$ at 30 750
\pinlabel $\varphi$ at 850 20
\small
\pinlabel $(1,2)$ at 1800 1400
\pinlabel $(2,3)$ at 1800 1300
\pinlabel $(3,4)$ at 1800 1155
\pinlabel $(4,5)$ at 1800 1065
\pinlabel 0 at 40 90
\pinlabel 0 at 150 40
\pinlabel $2\pi$ at 40 1540
\pinlabel $2\pi$ at 1670 40
\pinlabel $(3,5)$ at 650 220
\pinlabel $(2,3)$ at 1160 570
\endlabellist
\includegraphics[height=5cm]{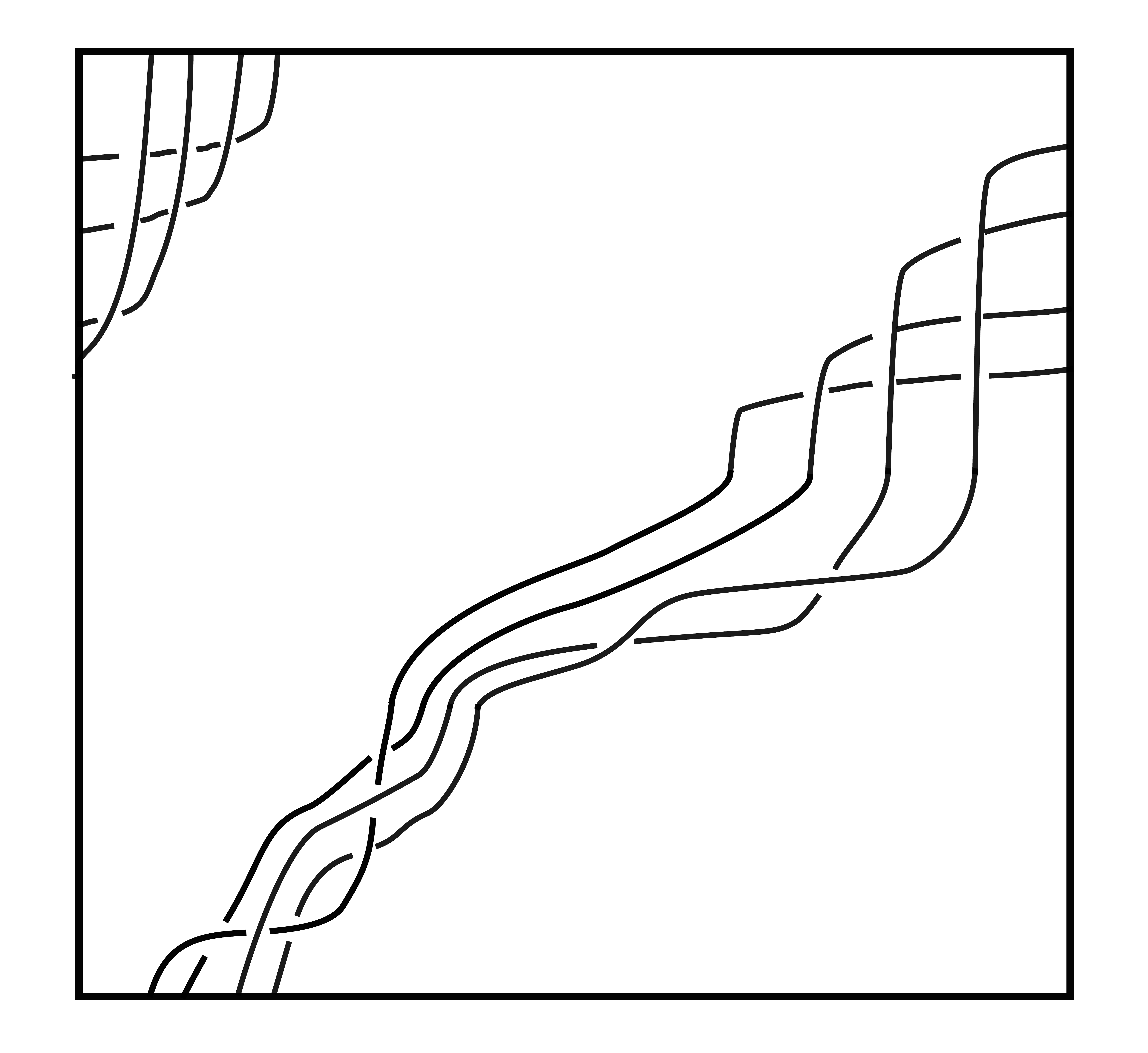}
\caption{a) $R'(a_{3,5})$ for $n=5$. The transpositions at the top and bottom edge both read $(1,5)$, $(1,4)$, $(1,3)$, $(1,2)$ from left to right. b) A Rampichini diagram for $\delta$ with $n=5$. c) A Rampichini diagram for $\delta$ built from $R'(a_{3,5})$, $R'(a_{2,3})$ and the diagram in b). Inserting inner loops where the labels $(3,5)$ and $(2,3)$ are results in a Rampichini diagram for $a_{3,5}a_{2,3}\delta$.\label{fig:Rlast}}
\end{figure}
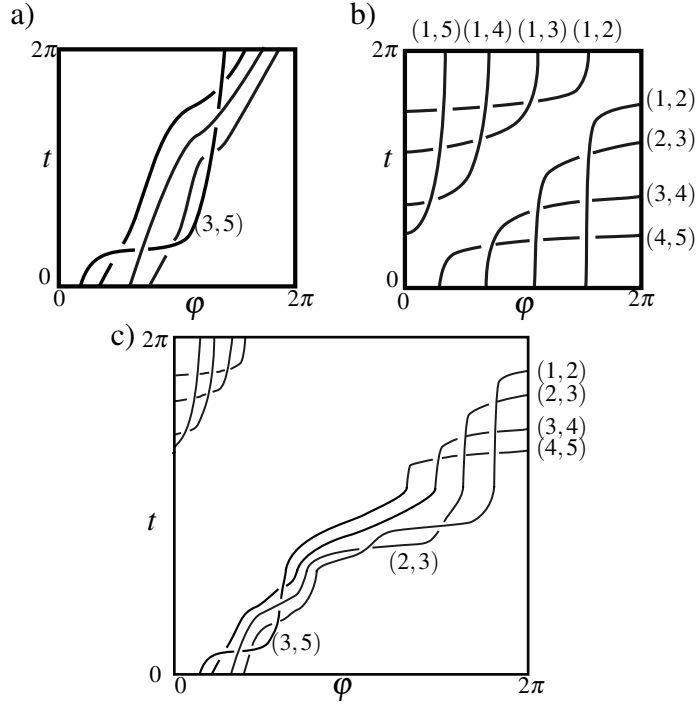

Let $P=\underset{k=1}{\overset{M}{\prod}}a_{i_k,j_k}$. We glue the bottom edge of the square diagram $R'(a_{i_k,j_k})$ along the top edge of $R'(a_{i_{k-1},j_{k-1}})$ for all $k=2,3,\ldots,M$. This results in a square diagram whose labels at the bottom edge and at the top edge are $a_{1,n},a_{1,n-1},\ldots,a_{1,2}$. However the endpoints of the curves at the top edge of the square are strictly to the right of the endpoints of the curves at the bottom edge. We glue the top edge of this square diagram along the bottom edge of the Rampichini diagram of $\delta$. In order to be able to perform these concatenations of square diagrams, we have to shift the curves in the upper square diagram to the right, so that the endpoints of its curves on its bottom edge coincide with the endpoints of the lower square diagram on its top edge. Likewise we have to shift curves (in parallel) so that the endpoints at the top edge of the new diagram coincide with the endpoints on the bottom edge of $R'(a_{i_1,j_1})$. Note that none of these parallel shifts affect the crossing pattern or cactus at any value of $t$, nor do they change the strict monotonicity of the curves (cf. Figure~\ref{fig:Rlast}).

The resulting square diagram is a simple, pure Rampichini diagram of $\delta$ and there are $t_k$, $k=1,2,\ldots,M$, with $t_k>t_{k-1}$ such that $\tau_{n-1}(t_k)=a_{i_k,j_k}$. Thus inserting inner loops whose moving critical value corresponds to $\tau_{n-1}(t_k)$ results in a Rampichini diagram of $P\delta$.

It follows from Theorem~\ref{thm:main2} that the closure of $P\delta$, which is the closure of $B$, is real algebraic.
\end{proof}

\end{document}